\documentclass[12pt]{article}

\usepackage{graphicx}
\usepackage{latexsym,amssymb}
\usepackage{amsthm}
\usepackage{indentfirst}
\usepackage{amsmath}
\usepackage{color}
\usepackage{xcolor}
\usepackage{fourier}
\usepackage[all]{xy}
\usepackage[colorlinks=blcak,backref=page]{hyperref}

\usepackage{setspace}     
\usepackage{marginnote}

\newtheorem*{Main Theorem}{Main Theorem}

\newtheorem{Theorem}{Theorem}[section]
\newtheorem*{Theorem A}{Theorem A}
\newtheorem*{Theorem A'}{Theorem A'}
\newtheorem*{Theorem B'}{Theorem B'}

\newtheorem{Definition}[Theorem]{Definition}
\newtheorem{Proposition}[Theorem]{Proposition}
\newtheorem{Lemma}[Theorem]{Lemma}

\newtheorem*{Notation}{Notation}
\newtheorem{Remark}[Theorem]{Remark}
\newtheorem{Example}[Theorem]{Example}

\newtheorem{Corollary}[Theorem]{Corollary}

\newtheorem*{Claim}{Claim}
\newtheorem{Claim-numbered}[Theorem]{Claim}

 \def\NN{{\mathbb N}} 

\def\QQ{{\mathbb Q}} \def\RR{{\mathbb R}} \def\SS{{\mathbb S}}
\def\TT{{\mathbb T}}

 \def\ZZ{{\mathbb Z}}

\def\Hol{{\rm Hol}}

\def\cA{{\cal A}}  \def\cG{{\cal G}}  
    
\def\cC{{\cal C}}    \def\cU{{\cal U}}
   \def\cP{{\cal P}} 
  \def\cK{{\cal K}}  
\def\cF{{\cal F}}  \def\cL{{\cal L}}

\newcommand{\id}{\operatorname{Id}}
\newcommand{\Id}{\operatorname{Id}}

\newcommand{\diff}{{\operatorname{Diff}}}

\def\diff{\operatorname{Diff}}
\def\GL{\operatorname{GL}}

\def\SO{\operatorname{SO}}

\def\dim{\operatorname{dim}}

\def\Per{\operatorname{Per}}

\def\supp{\operatorname{Supp}}

\def\det{\operatorname{det}}

\def\Leb{\operatorname{Leb}}

\begin{document}

\title{
Lyapunov spectrum rigidity and simultaneous linearization for random Anosov diffeomorphisms
}

\author{Aaron Brown and Yi Shi}
%%	 \quad and\quad  
%%	 Yi Shi\footnote{Y.S. is supported by National Key R\&D Program of China 2021YFA1001900 and NSFC 12090015.}}

%\date{}

\maketitle

\begin{abstract}
In this paper we study the Lyapunov spectrum rigidity for random walks of expanding maps on unit circle $\SS^1$ and Anosov diffeomorphisms on $d$-torus $\TT^d$. Let $\nu$ be a probability supported on the set of expanding maps on $\SS^1$ or a neighborhood of a generic Anosov automorphisms on $\TT^d$. 
If the Lyapunov spectrum of the $\nu$-stationary SRB-measure coincides with the Lyapunov spectrum of the algebraic action, then we can simultaneously linearize $\nu$ almost every system to an affine action.
Moreover, we prove the positive Lyapunov exponent rigidity for random walks of irreducible positive matrices acting on $\TT^2$.

%Let $\{A_i\}\subseteq\GL(2,\ZZ)$ be a family of irreducible positive matrices and $\{f_i\}\subseteq\diff^r(\TT^2)$ be perturbations of $\{A_i\}$. We show that if the positive Lyapunov exponent of stationary SRB-measure for random walk of $\{f_i\}$ equals random walk of $\{A_i\}$, then there exists $h\in\diff^{r-}(\TT^2)$ simultaneously conjugates $\{f_i\}$ to affine actions $\{A_i+v_i\}$.
\end{abstract}

\tableofcontents
%%%\section*{General comments/decissions/suggestions}
%%%\begin{enumerate}
%%%\item Should $\cU$ be a $C^2$ neighborhood or just $C^2$-bounded $C^1$-small neighborhood?
%%%\item Cones need a dimension in general (on $\TT^d$)
%%%\item def of finest dominated splitting.  
%%%\item introduce notion of "cone hyperbolic" to avoid saying "common stable and unstable conefields"
%%%
%%%\end{enumerate} 

\section{Introduction}

Let $M$ be a closed $d$-dimensional $C^\infty$ Riemannian manifold. Denote $C^r(M)$ the space of $C^r$ maps and $\diff^r(M)$ the group of $C^r$-diffeomorphisms both equipped with natural $C^r$-topology. 
 We assume $r\ge 2$ through this paper.
Let $\{f_i\}_{i=1}^m$ be a tuple in $C^r(M)$ or $\diff^r(M)$.   Let $\omega=(\omega_i)_{i\in\NN}$ be a sequence of independent random variables uniformly distributed on $\{1,\cdots,m\}$. Consider the Markov process 
$x_{n+1}=f_{\omega_n}(x_n)$ on $M$ and denote $f^n_{\omega}=f_{\omega_{n-1}}\circ\cdots\circ f_{\omega_0}$. 

We say a Borel probability measure $\mu$ on $M$ is stationary for this Markov process if  
$$
\frac1m\sum_{i=1}^m\mu(f^{-1}_i(A))=\mu(A), \qquad \forall \text{~Borel~set}~A\subseteq M.
$$
A stationary measure is ergodic if it is not a non-trivial convex sum of two distinct stationary measures. Fix an ergodic stationary measure $\mu$,  there exist Lyapunov exponents 
$$
\lambda_1(\mu)\leq\lambda_2(\mu)\leq\cdots\cdots\leq\lambda_d(\mu)
$$
such that for $\mu$-a.e. $x\in M$ and every $v\in T_xM$ with $v\neq 0$, the limit $\lim_{n\to\infty}1/n\log\|D_xf^n_{\omega}v\|$ exists and equals one of these Lyapunov exponents.

The goal of this paper is studying the simultaneous smooth conjugacy of random walks from the Lyapunov exponents rigidity of stationary measures. This phenomenon first observed by Dolgopyat and Krikorian.
They \cite{DK} showed that a family of rotations $\{R_i\}$ generating $\SO(d+1)$ is Lyapunov spectrum rigid, i.e. for the perturbations $\{f_i\}$ of $\{R_i\}$, if Lyapunov exponents of the random walk generated by $\{f_i\}$ vanish (equal to $\{R_i\}$), then there exists smooth conjugacy simultaneously conjugates $\{f_i\}$ to rotations. Later, DeWitt \cite{DeW2} proved Lyapunov spectrum rigidity for isometric random walks on isotropic manifolds.

In this paper, we focus on the simultaneous smooth conjugacies of random walks of expanding maps on unit circle $\SS^1$ and Anosov diffeomorphisms on $d$-torus $\TT^d$ with Lyapunov spectrum rigidity.

\subsection{Random expanding maps on $\SS^1$}
 
A local diffeomorphism $f:M\to M$ is a expanding map if its derivative $Df$ uniformly expands every unit vector in $TM$. It is clear that the space expanding maps is open in $C^r(M)$. 
For an expanding map $f:M\to M$, Shub \cite{Sh} proved that $f$ is topologically conjugate to an affine expanding endomorphism of an infra-nilmanifold provided $\pi_1(M)$ is nilpotent. Later Franks \cite{Fr} showed
that $\pi_1(M)$ has polynomial growth. Finally, Gromov \cite{Gr}
proved that a finitely generated group is virtually nilpotent, which implies every expanding map is topologically conjugate to an affine expanding endomorphism on an infra-nilmanifol. This gives a complete topological classification of expanding maps.

For the smooth classification of expanding maps, the 1-dimensional case has been completely solved by Shub and Sullivan \cite{ShS}. They showed that if the topological conjugacy between two $C^r$-smooth circle expanding maps is absolutely continuous, then the conjugacy is $C^r$-smooth. Recently, Gogolev and Hertz \cite{GRH} give a smooth classification of (non-algebraic) expanding maps in higher dimensions in terms of  Jacobians of return maps at periodic points. %{\color{red} emphasize that \cite{GRH} is about higher dimn.}

Our first result concerns global rigidity of the Lyapunov  exponent for the unique  stationary SRB-measure for i.i.d.\ random walks by expanding circle maps.  
Here the   stationary SRB-measure is the unique stationary measure absolutely continuous with respect to the arc-length parameter.  See  Definition \ref{def:SRB} for more general and  precise definition and \ref{thm:expanding-SRB} below for existence and uniqueness.  
The Lyapunov exponent for the  stationary SRB-measure satisfies an inequality; equality holds if and only if the random walk is simultaneously smoothly conjugate to an affine random walk.

%simultaneously linearization of   We study the Lyapunov This Lyapunov exponent 

%Our first result concerns simultaneously linearization of random walks of  expanding circle maps by Lyapunov exponent rigidity of stationary SRB-measure. Here the stationary SRB-measure means it has absolutely continuous conditional measures along unstable manifolds, see Definition \ref{def:SRB} for precise definition. 

\begin{Theorem}\label{thm:S1}
	Let $\{f_i\}_{i=1}^m$ be a family of $C^r$-expanding maps on $\SS^1$ with degrees $\{d_i\}_{i=1}^m\subseteq\NN_{\geq2}$. Let $\mu_{\rm SRB}$ be the (unique) stationary SRB-measure for  the random walk generated by $\{f_i\}_{i=1}^m$. Then the Lyapunov exponent of $\mu_{\rm SRB}$ satisfies
\begin{equation}\label{eq:SRBinequ}
	\lambda(\mu_{\rm SRB})\leq \frac1m\sum_{i=1}^m \log |d_i|.
	\end{equation}
	Moreover, equality holds in \eqref{eq:SRBinequ} if and only if there exists $h\in\diff^r(\SS^1)$ and $\{\rho_i\}_{i=1}^m\subseteq\SS^1$ such that
	$h\circ f_i\circ h^{-1}:\SS^1\to\SS^1$ is an affine map of the form
	$$
	h\circ f_i\circ h^{-1}(x)=d_i\cdot x+\rho_i,
	\qquad i=1\cdots,m.
	$$
\end{Theorem}

\begin{Remark}
	We have the following remarks:
	\begin{enumerate}
		\item An analogue of Theorem \ref{thm:S1} holds for i.i.d.\ random walk generated by any probability measure $\nu$ supported on the set of $C^r$-expanding maps on $\SS^1$ satisfying a certain  moment condition; see Theorem \ref{thm:S1-general} for general statement.
		\item In the case of equality in \eqref{eq:SRBinequ}, the stationary SRB-measure $\mu_{\rm SRB}$ is necessarily invariant by each generator $\{f_i\}_{i=1}^m$ and $$\mu_{\rm SRB}=h_*({\rm Leb}_{\SS^1}).$$
	\end{enumerate}	 
\end{Remark}

\subsection{Random Anosov diffeomorphisms on $\TT^d$}

The question of topological classification for Anosov diffeomorphisms was proposed by Smale \cite[Problem (3.5)]{Sm}, and  remains widely open. The classical result of Franks \cite{Fr} and Manning \cite{Mn} shows that an Anosov diffeomorphism on a torus, a nilmanifold or a infra-nilmanifold is topologically conjugate to an affine automorphism.  In general, we do not expect the conjugacy to be smooth; indeed, spectral data of the derivative at periodic points provides obstructions to $C^1$ conjugacies.  Matching of spectral data associated to periodic points is sufficient to obtain a smooth conjugacy for Anosov diffeomorphisms on 2-torus (\cite{dlL,MM1,dlLM}) or on $d$-torus assuming the linearization is \emph{generic} (\cite{GG,Go08,GKS,DeG,KSW1}; see Definition \ref{def:generic}). Similarly, for volume-preserving Anosov diffeomorphisms that are perturbations of automorphisms, coincidence of the Lyapunov spectrum (with respect to volume) with the Lyapunov spectrum of the automorphism also implies the conjugacy is smooth  (\cite{SY,GKS1,DeW1}).

For irreducible, higher-rank Anosov $\ZZ^k$-actions on tori or a nilmanifolds, Rodriguez Hertz and Wang \cite{HW} proved that all such actions are smoothly conjugate to an affine action; see also the earlier result of Fisher, Kalinin, and Spatzier \cite{FKS}.  Unlike the case of a single diffeomorphism, smoothness of the conjugacy in  \cite{HW} arises from the high-rank group structure alone, without additional assumptions on the periodic data or Lyapunov spectrum.  
 
%We say $A\in\GL(d,\ZZ)$ is a generic automorphism if its characteristic polynomial is irreducible over $\QQ$ and {\red{has simple eigenvalues over $\mathbb{C}$}}; see Definition \ref{def:generic}. The proportion of generic automorphisms contained in $\GL(d,\ZZ)$  goes to 1 as $\|A\|\rightarrow+\infty$
{We study a subset of $A\in\GL(d,\ZZ)$ of  generic automorphism, that have the property that the proportion of generic automorphisms contained in $\GL(d,\ZZ)$  goes to 1 as $\|A\|\rightarrow+\infty$; see Definition \ref{def:generic}.  In particular generic automorphisms are irreducible over $\QQ$ and eigenvalues occur with distinct absolute value or as complex conjugate pairs.}

For random walks of Anosov diffeomorphisms on $d$-torus, we have the following theorem.

\begin{Theorem}\label{thm:d-dim}
	Let $A\in\GL(d,\ZZ)$ be a generic automorphism and let $\{f_i\}_{i=1}^m\subseteq\diff^2(\TT^d)$ be perturbations of $A$. Let $\mu_{\rm SRB}$ be the (unique) stationary SRB-measure of the random walk generated by $\{f_i\}_{i=1}^m$. If  the Lyapunov exponents of $\mu_{\rm SRB}$ coincide with those of $A$, % (counted with multiplicity), 
	then there exist $h\in\diff^{1+}(\TT^d)$ and $\{v_i\}_{i=1}^m\subseteq\TT^d$ such that
		$$
		h\circ f_i\circ h^{-1}=A+v_i, \qquad i=1,\cdots,m.
		$$ 
\end{Theorem}

\begin{Remark}
	We have the following remarks:
	\begin{enumerate}
		\item The key observation is that if positive Lyapunov exponents are equal to the linear action, then the stable bundles are non-random; see first items of Theorem \ref{thm:2-dim} and Theorem \ref{thm:commuting}. The non-randomness of both stable and unstable bundles implies simultaneous topological linearization.
		\item If we take $\{f_i\}_{i=1}^m\subseteq\diff^{\infty}(\TT^d)$ which are close to $A$, then according to the bootstrap result of \cite{KSW1,KSW2}, the conjugacy $h$ also belongs to $\diff^{\infty}(\TT^d)$.
		\item If $d=2$, we can take $\{f_i\}_{i=1}^m$ are contained in a neighborhood of an Anosov diffeomorphism $f\in\diff^2(\TT^2)$, see Theorem \ref{thm:2-dim}.
	\end{enumerate}
\end{Remark}

For perturbations of  family of automorphisms on $\TT^2$ induced by 
 non-commuting positive matrices, coincidence of the  positive Lyapunov exponent (with respect to the stationary SRB measure) is sufficient for simultaneous smooth conjugacy.  

We say a family of matrices $\{A_i\}_{i=1}^m\subseteq\GL(d,\ZZ)$ is irreducible if the only subspace $V\subseteq\RR^d$ satisfying  $A_i(V)=V$ for every $i=1,\cdots,m$ is either $V=\{0\}$ or $V=\RR^d$. We say $A_i=(a_{ij})_{d\times d}\in\GL(d,\ZZ)$ is a positive matrix if all its elements are positive $a_{ij}>0$.  Positive matrices in $\GL(2,\ZZ)$ are always  hyperbolic and expand vectors in the first and third quadrants.

The random walk obtained by perturbing the generators of a linear action on $\TT^2$ induced by  non-commuting positive matrices admits an ergodic stationary SRB measure.  Let  $\lambda^+(\mu_{\rm SRB})$ denote the positive Lyapunov exponent for the random walk relative to the stationary SRB measure.  Similarly, we define the top Lyapunov exponent of the random product of matrices $\{A_i\}$ to be the a.s.\ limit $$\lambda^{+}(\{A_i\}):= \lim_{n\to \infty} \frac  1 n \log \|A_{i_{n-1}} \cdots  A_{i_0}\|$$ 
where $A_i$ are chosen i.i.d.\ with probability $ 1/ m$.  

\begin{Theorem}\label{thm:positive-T2}
	Let $\{A_i\}_{i=1}^m\subseteq\GL(2,\ZZ)$ be an irreducible family of  positive matrices and let $\{f_i\}_{i=1}^m\subseteq\diff^r(\TT^2)$ be   perturbations of $\{A_i\}_{i=1}^m$ with $r\geq2$.
Then $$\lambda^+(\mu_{\rm SRB}) \le \lambda^{+}(\{A_i\}).$$  Moreover, equality  $\lambda^+(\mu_{\rm SRB}) = \lambda^{+}(\{A_i\})$ holds 
if and only if there exist $h\in\diff^{r-}(\TT^2)$ and $\{v_i\}_{i=1}^m\subseteq\TT^2$ such that
    $$
    h\circ f_i\circ h^{-1}=A_i+v_i, \qquad i=1,\cdots,m.
    $$ 

%    If the positive Lyapunov exponent of stationary SRB-measure for random walk of $\{f_i\}_{i=1}^m$ equals random walk of $\{A_i\}_{i=1}^m$, then there exist $h\in\diff^{r-}(\TT^2)$ and $\{v_i\}_{i=1}^m\subseteq\TT^2$ such that
%    $$
%    h\circ f_i\circ h^{-1}=A+v_i, \qquad i=1,\cdots,m.
%    $$ 
\end{Theorem}

\begin{Remark}
	We have the following remarks:
	\begin{enumerate}
		\item %Here we also only need each $f_i$ be $C^1$-close to $A_i$. 
		The regularity of the conjugacy $h$ satisfies $r-=r$ if $r$ is non-integer and $r->r-\epsilon$ for every $\epsilon>0$ if $r$ is integer.  The mild loss of regularity comes from the use of Journ\'e's theorem \cite{J}.
		\item The main novelty of this theorem is that we only need the equality of positive Lyapunov exponent, and no assumption on the negative Lyapunov exponent. This comes from using that the action of $\{A_i\}$ is irreducible and thus there exists  $A_i\neq A_j$ whose  stable foliations are transverse.
		\item For sufficiently small perturbations the results of \cite{BLOR} show that the stationary SRB measure $\mu_{\rm SRB}$ is absolutely continuous with respect to the ambient volume.  We note however, that we do not directly use the absolute continuity of $\mu_{\rm SRB}$ in the proof of this theorem.  
	\end{enumerate}
\end{Remark}

\section{Preliminary definitions and constructions}

Let $M$ be a $C^\infty$ closed Riemannian manifold and let $\Gamma\subseteq\diff^r(M)$ be a subgroup. We say a Borel probability measure $\mu$ on $M$ is $\Gamma$-invariant if $\mu(f^{-1}(A))=\mu(A)$ for every $f\in\Gamma$ and every Borel set $A\subseteq M$. In general, when $\Gamma$ is not amenable, there may not exist any $\Gamma$-invariant measures, and thus we need a weaker notion of invariance.

Let $\nu$ be a Borel probability measure on $\Gamma$. We say a Borel probability measure $\mu$ on $M$ is $\nu$-stationary if
$$
\int f_*\mu~d\nu(f)=\mu,  \qquad \text{or equivalently} \qquad
\int\mu(f^{-1}(A))~d\nu(f)=\mu(A)
$$
for every Borel subset $A\subseteq M$. It is clear that if $\mu$ is $\Gamma$-invariant, then it is $\nu$-stationary for any probability $\nu$ supported on $\Gamma$. Given a $\nu$-stationary measure $\mu$, if $f_*\mu=\mu$ for $\nu$-a.e. $f\in\Gamma$, we say $\mu$ is $\nu$-almost surely (a.s.) invariant. The same definitions hold when we take $\nu$ to be a Borel probability measure on $C^r(M)$. 

Denote by $\Sigma_+=\left(\diff^r(M)\right)^\NN$ the set of sequence of diffeomorphisms 
$$
\omega=(\omega_i)_{i\in\NN}=(f_{\omega_0},f_{\omega_1},f_{\omega_2},\cdots)\in\Sigma_+,
\qquad \text{where} \qquad
f_{\omega_i}\in\diff^r(M).
$$ 
Given a probability $\nu$ on $\diff^2(M)$, we equip $\Sigma_+$ with the product measures $\nu^\NN$. Fix $r=2$ for the remainder.  Since $\diff^2(M)$ is a Polish space, so is $\Sigma_+$ and the measure $\nu^\NN$ is Radon. Let $\sigma:\Sigma_+\to\Sigma_+$ be the shift map
$$
\sigma:(f_{\omega_0},f_{\omega_1},f_{\omega_2},\cdots)
       \mapsto
       (f_{\omega_1},f_{\omega_2},f_{\omega_3},\cdots),
$$
then $\nu^\NN$ is $\sigma$-invariant. 

Given $\omega=(f_{\omega_0},f_{\omega_1},f_{\omega_2},\cdots)\in\Sigma_+$ and $n\in\NN$, we can define a cocycle
$$
f_\omega^0={\rm Id},\qquad f_\omega=f_\omega^1=f_{\omega_0},\qquad 
f_\omega^n=f_{\omega_{n-1}}\circ\cdots\circ f_{\omega_1}\circ f_{\omega_0},\quad \forall n\geq1.
$$
We denote $\chi^+(M,\nu)$ the random dynamical system on $M$ defined by random compositions $\{f_\omega^n\}_{\omega\in\Sigma_+}$. 

Similarly, denote $\Sigma=\left(\diff^2(M)\right)^\ZZ$ the set consisting of sequences of diffeomorphisms
$$
\omega=(\cdots,f_{\omega_{-1}},f_{\omega_0},f_{\omega_1},\cdots)\in\Sigma.
$$ 
The corresponding product measure $\nu^\ZZ$ on $\Sigma$ is also invariant by the shift map $\sigma:\Sigma\to\Sigma$.

For $\omega=(\cdots,f_{\omega_{-1}},f_{\omega_0},f_{\omega_1},\cdots)\in\Sigma$, we define the random compositions $f_\omega^n$ for $n\in\NN$ as following
\begin{align*}
	f_\omega^n&=f_{\omega_{n-1}}\circ\cdots\circ f_{\omega_1}\circ f_{\omega_0}; \\
	f_\omega^{-n}&=
	(f_{\omega_{-n}})^{-1}\circ(f_{\omega_{-(n-1)}})^{-1}\circ\cdots\circ(f_{\omega_{-1}})^{-1}.
\end{align*}

For a probability $\nu$ on $\diff^2(M)$, 
we fix a $\nu$-stationary measure $\mu$. A subset $A\subseteq M$ is $\chi^+(M,\nu)$-invariant, if for $\nu$-a.e. $f$ and $\mu$-a.e. $x\in M$, $x\in A$ implies $f(x)\in A$ and $x\notin A$ implies $f(x)\notin A$.  We say $\mu$ is ergodic if for every $\chi^+(M,\nu)$-invariant set $A$, we have $\mu(A)\cdot\mu(M\setminus A)=0$. 

We have the following random Oseledec's multiplicative theorem. 

\begin{Proposition}[{\cite[Proposition I.3.1]{LQ}}]\label{prop:Lyapunov}
	Let $\nu$ be a probability measure on $\diff^2(M)$ satisfying
	\begin{align}\label{eq:integrable}
		\int\log^+(\|f\|_{C^2})+\log^+(\|f^{-1}\|_{C^2}) d\nu<\infty,
	\end{align}
	and let $\mu$ be an ergodic $\nu$-stationary measure.
	
	There exist numbers (called Lyapunov exponents) $-\infty<\lambda_1<\lambda_2<\cdots<\lambda_l<\infty$ such that for $\nu^\NN$-a.e. $\omega\in\Sigma_+$ and $\mu$-a.e. $x\in M$, there exists a flag of linear subspaces
	$$
	\{0\}=V_\omega^0(x)\subsetneqq V_\omega^1(x)\subsetneqq \cdots \subsetneqq V_\omega^l(x)=T_xM,
	$$
	satisfying
	$$
	\lim_{n\to+\infty}\frac{1}{n}\log\|D_xf_\omega^n(v)\|=\lambda_k,
	\qquad \forall v\in V_\omega^k(x)\setminus V_\omega^{k-1}(x).
	$$
	Moreover, $m_k={\rm dim}V_\omega^k(x)-{\rm dim}V_\omega^{k-1}(x)$ is constant a.s. and $V_\omega^k(x)$ are equivariant: 
	$$
	D_xf_\omega V_\omega^k(x)=V_{\sigma(\omega)}^k(f_\omega(x)).
	$$
\end{Proposition}

\subsection{Skew product and fiberwise Lyapunov exponents}

Consider the product space $\Sigma_+\times M$ and the skew product system related to the random dynamical systems $\chi^+(M,\nu)$ on $\Sigma_+\times M$:
$$
F_+:~\Sigma_+\times M \to \Sigma_+\times M, \qquad
F_+(\omega,x)=\left(\sigma(\omega), f_\omega(x)\right).
$$
From \cite[Chapter 1]{LQ}, a probability measure $\hat\mu$ is a $\nu$-stationary measure if and only if $\nu^{\NN}\otimes\hat\mu$ is $F_+$-invariant.

For the product space $\Sigma\times M$ we also define the skew-extension map $F:\Sigma\times M\to\Sigma\times M$ which is also defined as
$$
F(\omega,x)=\left(\sigma(\omega), f_\omega(x)\right), \qquad \omega\in\Sigma,~x\in M.
$$
We have the following correspondence between $\nu$-stationary measures and invariant measures of skew product system.

\begin{Proposition}[{\cite[Proposition I.1.2]{LQ}}]\label{prop:measure}
	Let $\hat\mu$ be a $\nu$-stationary probability measure.  There exists a unique $F$-invariant probability measure $\mu$ on $\Sigma\times M$ such that
	$$
	\pi_*^+(\mu)=\nu^{\NN}\otimes\hat\mu,
	$$
	where $\pi^+:\Sigma\times M\to\Sigma_+\times M$ is the natural projection and 
	$\mu=\lim_{n\to+\infty}(F^n)_*(\nu^\ZZ\otimes\hat\mu)$.
	
	Moreover, if we disintegrate $\mu$ with respect to the projection $\pi:\Sigma\times M\to\Sigma$, there exists a family of Borel probability measures $\{\mu_\omega\}_{\omega\in\Sigma}$ satisfying
	$$
	\mu=\int_{\Sigma}\mu_\omega d\nu^\ZZ(\omega).
	$$
	Finally, for $\nu^\ZZ$-a.e. $\omega\in\Sigma$ and $\xi\in\Sigma$ with $\omega_i=\xi_i$ for all $i<0$, we have $%	\qquad \text{and} \qquad
	\mu_{\omega}=\mu_{\xi}.$

\end{Proposition}

Proposition \ref{prop:measure} associates to a $\nu$-stationary measure $\hat \mu$ on $M$ a canonical invariant probability measure $\mu$ for the skew product $F\colon \Sigma\times M\to \Sigma\times M$ with $\pi_*\mu=\nu^\ZZ$.  
We may also consider more general $F$-invariant probability measures $\mu$ on $\Sigma\times M$ with $\pi_*\mu=\nu^\ZZ$ that are not necessarily natural extensions of stationary measures as in Proposition \ref{prop:measure}.  

Let $\mu$ be an $F$-invariant probability measure on $\Sigma\times M$ such that $\pi_*(\mu)=\nu^\ZZ$. The fiberwise differential over the fiberwise tangent bundle,  defined as
\begin{align*}
	DF:\Sigma\times TM&\to\Sigma\times TM, \\
	 \left(\omega,(x,v)\right)&\mapsto\left(\sigma(\omega),(f_{\omega}(x),D_xf_{\omega}(v))\right),
\end{align*} 
defines a linear cocycle over a measure preserving system
$F:(\Sigma\times M,\mu)\to(\Sigma\times M,\mu)$.
Equip $M$ with any Riemannian metric.   Suppose that $\mu$ is ergodic and satisfies the integrability condition 
	\begin{align}\label{eq:integrablegeneralskew}
		\int\log^+(\|D_x f_\omega\|_{C^2}+\log^+(\|(D_x f_\omega)^{-1}\|_{C^2}) d\mu(\omega,x)<\infty.
	\end{align}
We may then apply Oseledec's theorem to the cocycle $DF$.  In particular, there are numbers $\{\lambda_k\}_{k=1}^l$, called  the fiberwise Lyapunov exponents, and a $\mu$-measurable splitting %:% for $\mu$-a.e. $(\omega,x)\in\Sigma\times M$
$$
\{\omega\}\times T_xM=\bigoplus_{k=1}^l E^k_{\omega}(x)
%%\qquad \text{and} \squad
%%\{\lambda_k\}_{k=1}^l\subseteq\RR,
$$
such that for every $v\in E^k_{\omega}(x)\setminus\{0\}$,
$$
\lim_{n\to\pm\infty}\frac{1}{n}\log\|DF^n(v)\|=\lim_{n\to\pm\infty}\frac{1}{n}\log\|D_xf_{\omega}^n(v)\|=\lambda_k.
$$

%Given a $DF$-invariant, $\mu$-measurable bundle $

%%Here $\{\lambda_k\}_{k=1}^l$ are the fiberwise Lyapunov exponents of $\mu$, which are the Lyapunov exponents of the $\nu$-stationary $\hat\mu$ defined in Proposition \ref{prop:Lyapunov}.

%\red{make these paragraphs about general skew products}
%Let $\hat\mu$ be an ergodic $\nu$-stationary measure and $\mu$ be the corresponding $F$-invariant measure such that $\pi_*(\mu)=\nu^\ZZ$. The fiberwise differential over the fiberwise tangent bundle is defined as
%\begin{align*}
%	DF:\Omega\times TM&\to\Omega\times TM, \\
%	 \left(\omega,(x,v)\right)&\mapsto\left(\sigma(\omega),(f_{\omega}(x),D_xf_{\omega}(v))\right),
%\end{align*} 
%which is a linear cocycle over a measure preserving system
%$F:(\Omega\times M,\mu)\to(\Omega\times M,\mu)$.
%
%From the integrability condition (\ref{eq:integrable}) in Proposition \ref{prop:Lyapunov}, we can apply Oseledec's theorem to $DF$ obtaining a measurable splitting: for $\mu$-a.e. $(\omega,x)\in\Omega\times M$
%$$
%\{\omega\}\times T_xM=\bigoplus_{k=1}^l E^k_{\omega}(x)
%\qquad \text{and} \qquad
%\{\lambda_k\}_{k=1}^l\subseteq\RR,
%$$
%such that for every $v\in E^k_{\omega}(x)\setminus\{0\}$,
%$$
%\lim_{n\to\pm}\frac{1}{n}\log\|DF^n(v)\|=\lim_{n\to\pm}\frac{1}{n}\log\|Dxf_{\omega}^n(v)\|=\lambda_k.
%$$
%Here $\{\lambda_k\}_{k=1}^l$ are the fiberwise Lyapunov exponents of $\mu$, which are the Lyapunov exponents of the $\nu$-stationary $\hat\mu$ defined in Proposition \ref{prop:Lyapunov}.

\subsection{Fiberwise Anosov systems}
Let $M$ be a compact manifold.  We equip  $\Sigma=\left(\diff^2(M)\right)^\ZZ$ the product metric, for $\omega=(\omega_i),\xi=(\xi_i)\in\Sigma$,
$$
d_{\Sigma}(\omega,\xi)=\sum_{i\in\ZZ}
\frac{\|f_{\omega_i}-f_{\xi_i}\|_{C^2}+\|f^{-1}_{\omega_i}-f^{-1}_{\xi_i}\|_{C^2}}{2^{|i|}}.
$$
For the product space $\Sigma\times M$, we define the metric
$$
d\left((\omega,x),(\xi,y)\right):=\max\left\{d_{\Sigma}(\omega,\xi), d_M(x,y)\right\}
$$
for every $(\omega,x),(\xi,y)\in\Sigma\times M$, where $d_M(\cdot,\cdot)$ is induced by the Riemannian metric on $M$.
Metrics on $\Sigma_+=\left(\diff^2(M)\right)^\NN$ and $\Sigma_+\times M$ are defined similarly same. 

We write $M_\omega=\{\omega\}\times M$ for the fiber of $\Omega\times M$ over $\omega$ and write $TM_\omega := \{\omega\}\times TM$.

\begin{Definition}\label{def:Anosov}
	Let $\cA\subset\diff^2(M)$ and let $\Omega=\cA^{\ZZ}\subset\Sigma$.   We say the skew product $F:\Omega\times M\to\Omega\times M$,
	$$
	F(\omega,x)=\left(\sigma(\omega),f_\omega(x)\right), \qquad
	\omega\in\Omega,~x\in M,
	$$
	is fiberwise Anosov if for every word $\omega\in\Omega$, there exists a continuous splitting $TM_{\omega}=E^s_{\omega}\oplus E^u_{\omega}$ and constants $\gamma>0$, $C>1$, $0<\theta<1$, $L>1$ such that the following hold:
	\begin{enumerate}
		\item The splitting is $DF$-invariant: 
		$$
		DF\left(E^s_\omega\right)=Df_{\omega_0}\left(E^s_\omega\right)=E^s_{\sigma(\omega)} 
		\qquad \text{and} \qquad 
		DF\left(E^s_\omega\right)=Df_{\omega_0}\left(E^u_\omega\right)=E^u_{\sigma(\omega)}.
		$$
		\item For nonzero vectors $v^s\in E^s_\omega$ and $v^u\in E^u_\omega$, we have
		$$
		\|Df_\omega^nv^s\|<C\exp(-n\gamma)\|v^s\| \qquad \text{and} \qquad
		\|Df_\omega^{-n}v^u\|<C\exp(-n\gamma)\|v^u\|, \qquad \forall n\geq0.
		$$
		\item For all $(\omega,x)\in\Omega\times M$, $(\omega,x)\mapsto E^s_{\omega}(x)$ and $(\omega,x)\mapsto E^u_{\omega}(x)$ are $(L,\theta)$-H\"older continuous.
	\end{enumerate}
\end{Definition}

The classical Stable Manifold Theorem also holds for fiberwise Anosov skew products.

\begin{Lemma}\label{lem:stable-mfd}
	Let $\cA\subset\diff^2(M)$ and $\Omega=\cA^{\ZZ}$, if the skew product $F(\omega,x)=\left(\sigma(\omega),f_\omega(x)\right):\Omega\times M\to\Omega\times M$ is fiberwise Anosov, then there exist families of stable and unstable foliations 
	$$
	\left\{\cF^s_\omega %=\cF^s_{\omega,f_{\omega}}
	\right\}_{\omega\in\Omega} 
	\qquad \text{and} \qquad 
	\left\{\cF^u_\omega%=\cF^u_{\omega,f_{\omega}}
	\right\}_{\omega\in\Omega}
	$$
	such that the following hold:
	\begin{enumerate}
		\item For every $\omega\in\Omega$, $\cF^s_\omega$ and $\cF^u_\omega$ are two transversal foliations on the fiber $M_\omega=\{\omega\}\times M$ and for every $x\in M_\omega$,
		$$
		T_x\cF^s_\omega(x)=E^s_\omega(x), \qquad \text{and} \qquad
		T_u\cF^s_\omega(x)=E^u_\omega(x).
		$$
		\item Let $\gamma>0$ be the constant in Definition \ref{def:Anosov}. There exists $r>0$ such that for every $0<\epsilon<\gamma$, there exists $C=C_\epsilon>1$, such that for $y,z\in\cF^s_{\omega,{r}}(x)$ and $n>0$,
		$$
		d_M(f^n_\omega(y),f^n_\omega(z))\leq C\exp[n(-\gamma+\epsilon)]d_M(y,z).
		$$
		Above, $\cF^s_{\omega,{r}}(x)$ denotes the path-component containing $x$ of the intersection of $\cF^s_{\omega,}(x)$ with the ball of radius $r$ centered at $x$.  
		
		{The same holds for $y,z\in\cF^u_{\omega,r}(x)$ under negative iteration $f^{-n}_\omega$.}
	\end{enumerate}
\end{Lemma}

%\begin{Definition}\label{def:common-cone}\note{idea: ``cone hyperbolic'' as short hand defn}
%	Let $\cA\subseteq\diff^2(M)$ be a family diffeomorphisms. We say elements in $\cA$ share common stable and unstable cone-fields  \red{are cone hyperbolic?} if there exist continuous cone fields \red{are the cones open, transverse? otherwise, we only get partial hyperbolicity} $x\mapsto\cC^s_x$, $x\mapsto\cC^u_x$ for $x\in M$ and $\gamma>0$, such that for every $x\in M$, non-zero vectors $v^s\in\cC^s_x$, $v^u\in\cC^u_x$, and every $f\in\cA$,
%	\begin{itemize}
%		\item $Df^{-1}(\cC^s(x))\subset\cC^s_{f^{-1}(x)}$,  ~and ~$\|Df^{-1}(v_s)\|>\exp(\gamma)\cdot\|v_s\|$;
%		\item $Df(\cC^u(x))\subset\cC^u_{f(x)}$, ~and ~$\|Df(v_u)\|>\exp(\gamma)\cdot\|v_u\|$.
%	\end{itemize}
%\end{Definition} 

{Let $M$ be a compact Riemannian manifold.  
Given a subspace $V\subset T_xM$ and a vector $w\in T_xM$ write $w=w_V+ w^\perp$ where $w_V\in V$ and $w^\perp\in V^\perp$.  Define $C_ \kappa(V)$, the cone of radius $0< \kappa<1$ around $V$ by 
$$C_ \kappa(V) := \{w\in T_xM: |w^\perp|\le  \kappa |w_V|\}$$
A $k$-cone-field on $M$ is a continuous family $x\mapsto \cC(x)$ where each $\cC(x)$ is a $C_ \kappa(V)$ around a vector space of dimension $k$.  

\begin{Definition}\label{def:common-cone}  %\note{idea: ``cone hyperbolic'' as short hand defn}
%Given $x\in M$, and $1\le k\le n$, and $0<\gamma<1$, a $k$-cone at $x$ of width $\gamma$ is a subset $C_k(x)\subset T_xM$ of vectors $$C_k(x)
Let $M$ be a compact Riemannian manifold of dimension $n$.  
	Let $\cA\subseteq\diff^2(M)$ be a family diffeomorphisms.   
	We say that  $\cA$ is \emph{cone hyperbolic} if there exist $1\le k \le n-1$, a continuous $k$-cone-field $x\mapsto\cC^s_x$,  
	a continuous $(n-k)$-cone-field $x\mapsto\cC^u_x$, and $\gamma>0$, 
%	\red{are the cones open, transverse? otherwise, we only get partial hyperbolicity} $x\mapsto\cC^s_x$, $x\mapsto\cC^u_x$ for $x\in M$ and 
	such that for every $x\in M$, non-zero vectors $v^s\in\cC^s_x$, $v^u\in\cC^u_x$, and every $f\in\cA$,
	\begin{itemize}
		\item $Df^{-1}(\cC^s(x))\subset\cC^s_{f^{-1}(x)}$,  ~and ~$\|Df^{-1}(v_s)\|>\exp(\gamma)\cdot\|v_s\|$;
		\item $Df(\cC^u(x))\subset\cC^u_{f(x)}$, ~and ~$\|Df(v_u)\|>\exp(\gamma)\cdot\|v_u\|$.
	\end{itemize}
\end{Definition} 

}

\begin{Remark}\label{rem:cone hyperbolicity}
\begin{enumerate}
\item We remark that if   $\cA$ is {cone hyperbolic} then every element of $f\in \cA$ is an Anosov diffeomorphism.  Moreover, for every $x\in M$ the stable subspace $E^s_f(x)$ has dimension $k$ and  $E^s_f(x)\subset \cC^s(x).$

	\item It is clear cone hyperbolicity %that sharing common stable and unstable cone-fields 
	is a $C^1$-open property. That is, $\cA\subseteq\diff^1(M)$   is  cone hyperbolic where elements %share  common stable and unstable cone-fields, 
	then there exists a neighborhood $\cU\subseteq\diff^1(M)$ of $\cA$ such that $\cU$ is  cone hyperbolic.  
	%admits the same property. 
	In particular, this holds for $\cA=\{f\}$ where $f$ is an Anosov diffeomorphism.
	\end{enumerate}
\end{Remark}

For a {cone hyperbolic}  family of Anosov diffeomorphisms, % sharing common stable and unstable cone-fields, we can show that 
the related skew product system is a fiberwise Anosov system from the classical cone-fields argument.

\begin{Lemma}\label{lem:cone-Anosov}
	Let $\cA\subseteq\diff^2(M)$ be a {cone hyperbolic}  family of diffeomorphisms.  % share common stable and unstable cone-fields. 
	Denote $\Omega=\cA^\ZZ$.  Then the skew product $F:\Omega\times M\to\Omega\times M$ defined as $F(\omega,x)=\left(\sigma(\omega),f_\omega(x)\right)$ is fiberwise Anosov.
\end{Lemma}

\begin{proof}
	The proof is the classical cone-criterion. The stable and unstable bundles are defined as
	$$
	E^s_{\omega}=\bigcap_{n\geq0}Df^{-n}_{\sigma^n(\omega)}\left(\cC^s_{\sigma^n(\omega)}\right)
	\qquad \text{and} \qquad
	E^u_{\omega}=\bigcap_{n\geq0}Df^n_{\sigma^{-n}(\omega)}\left(\cC^u_{\sigma^{-n}(\omega)}\right).
	$$
	See \cite[Section 2.2]{CP} for details.
\end{proof}

\subsection{Stationary and fiberwise SRB measures}

Let $\cA\subseteq\diff^2(M)$ be a family of diffeomorphisms such that, with $\Omega=\cA^\ZZ$, the skew product $F:\Omega\times M\to\Omega\times M$ is fiberwise Anosov. Let $\{\cF^u_\omega(x)\}_{(\omega,x)\in\Omega\times M}$ be the fiberwise unstable foliations of $F$. 

For every $F$-invariant measure $\mu$, we consider a measurable partition $\cP$ of $\Omega\times M$ such that for $\mu$-a.e. $(\omega,x)\in\Omega\times M$, there is an $r=r(\omega,x)>0$ such that the $r$-size unstable manifold of $(\omega,x)$ satisfies
$$ 
{\cF^u_{\omega,r}(x)}\subset\cP(\omega,x)\subset\cF^u_\omega(x).
$$
Such a partition is said to be $\cF^u$-subordinate. Let $\{\tilde{\mu}^{\cP}_{(\omega,x)}\}_{(\omega,x)\in\Omega\times M}$ denote a family of conditional measures of $\mu$ with respect to the partition $\cP$.

\begin{Definition}\label{def:fiber-SRB}
	Let $F:\Omega\times M\to\Omega\times M$ be an fiberwise Anosov skew product. An $F$-invariant measure $\mu$ is fiberwise SRB  if for any $\cF^u$-subordinate measurable partition $\cP$ with a family of corresponding conditional measures $\{\tilde{\mu}^{\cP}_{(\omega,x)}\}_{(\omega,x)\in\Omega\times M}$, the conditional measure $\tilde{\mu}^{\cP}_{(\omega,x)}$ is absolutely continuous with respect to the Lebesgue measure on $\cF^u_\omega(x)$ for $\mu$-a.e. $(\omega,x)\in\Omega\times M$.
\end{Definition}
{We remark that fiberwise SRB are actually equivalent to the Lebesgue measure on $\cF^u_\omega(x)$ for $\mu$-a.e. $(\omega,x)\in\Omega\times M$.} 

Recall Proposition \ref{prop:measure}, for any probability $\nu$ on $\cA$ and any $\nu$-stationary measure $\hat\mu$, there exists a unique $F$-invariant measure $\mu$ on $\Omega\times M$ such that
$$
\pi_*(\mu)=\nu^\ZZ
\qquad \text{and} \qquad
\mu=\lim_{n\to+\infty}(F^n)_*(\nu^\ZZ\otimes\hat\mu),
$$
where $\pi:\Omega\times M\to\Omega$ is the projection. 
We can now define the SRB property $\nu$-stationary  measures. We restricted ourselves in the case that $F$ is fiberwise Anosov.

\begin{Definition}\label{def:SRB}
	Let $\cA\subseteq\diff^2(M)$ and $\Omega=\cA^\ZZ$ where the corresponding skew product $F:\Omega\times M\to\Omega\times M$ is fiberwise Anosov. Assume the probability $\nu$ is supported on $\cA$, and $\hat\mu$ is a $\nu$-stationary probability measure. We say $\hat\mu$ is SRB if the corresponding $F$-invariant measure $\mu$ given by Proposition  \ref{prop:measure} is fiberwise SRB.
\end{Definition}

The following theorem shows the existence and uniqueness of stationary SRB-measure for a family of transitive Anosov diffeomorphisms sharing common stable and unstable cone-fields, which was essentially proved in \cite[Theorem VII.1.1]{LQ}, see also \cite[Theorem 2.9]{BLOR}.

\begin{Theorem}\label{thm:SRB}
	Let $\cA\subseteq\diff^2(M)$ be a family of transitive Anosov diffeomorphisms share common stable and unstable cone-fields. For every probability $\nu$ supported on $\cA$ satisfying integrability condition (\ref{eq:integrable}), there exists a unique ergodic $\nu$-stationary SRB measure $\hat\mu$ such that the corresponding $F$-invariant measure $\mu$ given by Proposiotn  \ref{prop:measure} satisfies $\pi_*(\mu)=\nu^\ZZ$ and is fiberwise SRB:
	\begin{itemize}
		\item the disintegration of $\hat\mu$ along fiberwise unstable foliation is absolutely continuous;
		\item $\hat\mu$ satisfies the fiberwise entropy formula:
		   $$
%		 \red {  h_{\hat\mu}(\chi^+(M,\nu))=\sum_{k=1}^lm_k\cdot\lambda_k.}\quad\quad
 h_{\hat\mu}(\chi^+(M,\nu))=\sum_{\lambda_k>0}m_k\cdot\lambda_k.
		   $$
		   Here $h_{\hat\mu}(\chi^+(M,\nu))$ is the measure entropy of random dynamical systems generated by $\nu$ (\cite[Definition I.2.1]{LQ}), $\{(\lambda_k, m_k)\}_{k=1}^l$ are Lyapunov exponents with multiplicities of $\hat\mu$ defined in Proposition \ref{prop:Lyapunov}.
	\end{itemize}
\end{Theorem}

\subsection{Stationary SRB measure for random expanding maps on $\SS^1$}

Let $E^r(\SS^1)$ denote the  open subset of $C^r(\SS^1)$ consisting of $C^r$-expanding maps of $\SS^1$,
$$
E^r(\SS^1)=\left\{f\in C^r(\SS^1):~|f'(x)|>1,~\forall x\in\SS^1\right\}.
$$
 Recall that we always assume $r\ge2$ in this paper.

For every $f\in E^r(\SS^1)$, %denote by $\|Df\|=\sup_{x\in\SS^1}|f'(x)|$ 
Denote 
	$$
	\|f\|_{C^2}=\sup_{x\in\SS^1}\max\left\{|f'(x)|,|f''(x)|\right\}.
	$$ 
	and let $\deg(f)$ be the degree of $f$.  For every expanding map $f\in E^r(\SS^1)$, we have 
	$$
	1<|\deg(f)|=\int|f'(x)|d{\rm Leb}(x)\leq\|f\|_{C^2}.
	$$
   and
   $$
   \left|\log|f'(x)|-\log|f'(y)|\right|\leq\|f\|_{C^2}\cdot d(x,y).
   $$

Denote by $\Sigma_+=\left( E^r(\SS^1)\right)^{\NN}$ the set of sequence of expanding maps
$$
\omega=(\omega_i)_{i\in\NN}=(f_{\omega_0},f_{\omega_1},f_{\omega_2},\cdots)\in\Sigma_+,
\qquad \text{where} \qquad
f_{\omega_i}\in E^r(M). 
$$
Let $\sigma:\Sigma_+\to\Sigma_+$ be the shift map $(f_{\omega_i})\mapsto(f_{\omega_{i+1}})$, 
and $F_+:\Sigma_+\times\SS^1\to\Sigma_+\times\SS^1$ be the associated skew-product system
$$
F_+(\omega,x)=\left(\sigma(\omega),f_{\omega}(x)\right)=\left(\sigma(\omega),f_{\omega_0}(x)\right),
\qquad \omega\in\Sigma_+,~x\in\SS^1.
$$

Let $\nu$ be a Borel probability measure on $E^r(\SS^1)$. As in the case when $\nu$ is supported on diffeomorphisms, a Borel probability measure $\mu$ on $\SS^1$ is a $\nu$-stationary measure
$$
\mu=\int f_*(\mu)d\nu(f),
$$
if and only if $\nu^\NN\times\mu$ is an $F_+$-invariant measure on $\Sigma_+\times\SS^1$. (See \cite[Chapter 1]{LQ}).

The following theorem shows the existence and absolute continuity of $\nu$-stationary SRB measure, which was proved in \cite[Theorem B']{KK}, see also \cite{Kif1,Kif2}.

\begin{Theorem}\label{thm:expanding-SRB}
	Let $\nu$ be a probability measure on $E^2(\SS^1)$ which satisfies
	$$
	\int\log\|f\|_{C^2}d\nu(f)<\infty.
	$$
	There exists a unique $\nu$-stationary ergodic SRB-measure $\mu_{\rm SRB}$ which is equivalent to Lebesgue measure on $\SS^1$.  
\end{Theorem}

\begin{Remark}\label{rk:density}
	Since $\mu_{\rm SRB}$  is equivalent to Lebesgue measure on $\SS^1$, there is a non-negative function $q(x)\in L^1(\SS^1,{\rm Leb})$  such that
	$$
	\frac{1}{q(x)}\in L^1(\SS^1,\mu_{\rm SRB}) 
	\qquad \text{and} \qquad
	q(x)=\frac{d\mu_{\rm SRB}}{d{\rm Leb}}(x), \qquad {\rm Leb}~a.e.~x\in\SS^1.
	$$
\end{Remark}

\section{Statement of results}

\subsection{Spectrum rigidity of expanding maps on $\SS^1$}
We recall a theorem of  Baxendale from \cite{Ben} which says that for a random walk generated by diffeomorphisms and any ergodic, absolutely continuous stationary measure $\mu$, the sum of all Lyapunov exponents (with multiplicity) is non-positive.  Moreover, equality holds if and only if $\mu$ is invariant under almost every generator. 

In the setting of expanding maps on $\SS^1$, we have the following Baxendale-like theorem generalizing Theorem \ref{thm:S1}.  Indeed, the following gives provides an inequality on the Lyapunov exponent for the unique absolutely continuous stationary measure.  When equality holds, we have rigidity of our random dynamical system: the stationary measure is invariant under every generator and the generators are simultaneously smoothly conjugate to affine maps.  

%.  We view this as a version of Bexandale's theorem and 
% Lyapunov spectrum rigidity.  
%which implies 
%Our first result shows
Recall that $\|f\|_{C^2}=\sup_{x\in\SS^1}\max\left\{|f'(x)|,|f''(x)|\right\}$ for $f\in E^r(\SS^1)$ and $r\ge2$.

\begin{Theorem}\label{thm:S1-general}
	Let $\nu$ be a probability measure on $E^r(\SS^1)$ $(r\geq2)$ which satisfies
	$$
	\int\log\|f\|_{C^2}d\nu(f)<\infty.
	$$
	The Lyapunov exponent of the $\nu$-stationary SRB measure $\mu_{\rm SRB}$ satisfies
	$$
	\lambda(\mu_{\rm SRB})\leq\int\log|\deg(f)|d\nu(f).
	$$
	Moreover, equality holds if and only if there exists $h\in\diff^r(\SS^1)$ that simultaneously conjugates every $f$ in the support of $\nu$ %$\supp (\nu)$, %\in E^r(\SS^1)$, 
	%$\nu$-a.e. $f\in E^r(\SS^1)$, 
%	$h$ conjugates $f$ 
	to an affine expanding map:
	$$
	h\circ f\circ h^{-1}(x)=\deg(f)\cdot x+\rho_f,
	\qquad \text{where}\quad\rho_f\in\SS^1.
	$$
\end{Theorem}

\begin{Remark}
	Here we don't need to assume the support of $\nu$ is bounded or discrete in $E^r(\SS^1)$. Theorem \ref{thm:S1} follows from this theorem directly by taking $\nu$ be finitely supported with equal probability. The main idea of the proof follows from the classical argument of \cite{LJ} and \cite{Ben}. 
\end{Remark}

\subsection{Lyapunov exponents of random automorphisms}
Throughout, we consider a family $\{A_i\}_{i=1}^m\subseteq\GL(d,\ZZ)$.  We have the following version of Oseledec's multiplicative theorem for i.i.d.\ products of matrices.

\begin{Proposition}\label{prop:Lyapaffine}
Let $p=\{p_i\}_{1\le i\le m}$ be a probability vector on $\{A_i\}_{i=1}^m\subseteq{GL}(d,\ZZ)$ with $p_i(A_i)>0$ for every $1\le i\le m$.  
There exist $-\infty<\lambda_1<\lambda_2<\cdots<\lambda_l<\infty$ such that for $p^\NN$-a.e.\ word $\omega=(A_{i_0}, A_{i_1}, A_{i_2}, \dots)$ there is a flag 
	$$
	\{0\}=V_\omega^0\subsetneqq V_\omega^1\subsetneqq \cdots \subsetneqq V_\omega^l=T_xM,
	$$
such that for all $1\le k\le l$ and $ v\in V_\omega^k\setminus V_\omega^{k-1},$
$$
	\lim_{n\to+\infty}\frac{1}{n}\log\|A_{i_{n-1}}\circ \dots \circ A_{i_0}(v)\|=\lambda_k.
	$$
\end{Proposition}
We observe that if $\mu$ is a $p$-stationary probability measure on  $\TT^d$, then the Lyapunov exponents and flag given by  Proposition \ref{prop:Lyapunov} coincide with those in Proposition {prop:Lyapaffine} and, in particular, are independent of $x$.

%\subsection{Spectrum rigidity of commuting automorphisms}
\subsection{Spectrum rigidity of individual and commuting families of automorphisms}
We consider the Lyapunov exponents relative to the SRB stationary measure of random perturbation of a single Anosov diffeomorphism $f$ on 2-torus.  We obtain rigidity of the family of perturbations when the SRB exponents coincide with those of the linearization of $f$.

\begin{Theorem}\label{thm:2-dim}
	Let $f\in\diff^r(\TT^2)$ be an Anosov diffeomorphism ($r\geq2$) and $A=f_*\in\GL(2,\ZZ)$ be its linear part. There exists a $C^2$-neighborhood $\cU_f\subseteq\diff^r(\TT^2)$ of $f$, such that for every probability $\nu$ on $\diff^r(\TT^2)$ with $\nu(\cU_f)=1$, letting $\mu_{\rm SRB}$ be the unique $\nu$-stationary SRB-measure the following hold:
	\begin{itemize}
		\item $|\lambda^{\pm}(\mu_{\rm SRB})|\le |\lambda^\pm(A)|$.
		\item %If the positive Lyapunov exponent of $\mu$ is equal to $A$: 
%\note{need to argue somewhere that non-random stable bundles gives non-random foliations}	
	If $\lambda^+(\mu_{\rm SRB})=\lambda^+(A)$, there is a $C^r$ foliation $\cF^s$ such that $\cF^s=\cF^s_{g}$ for $\nu$-a.e.\ $g\in\diff^r(\TT^2)$.  In particular, the stable foliation $\{\cF^s_\omega\}$ is $\nu^\NN$-almost surely independent of $\omega$ and is $C^r$ smooth.

%	then for $\nu$-a.e.\ $g,g'\in\diff^r(\TT^2)$} stable foliations $\{\cF^s_{g}\}$ and $\{\cF^s_{g'}\}$ coincide and are $C^r$-smooth

%	 then the stable bundles \red{$E^s(g)$ coincide for $\nu$-a.e. $g\in\diff^r(\TT^2)$;} in particular, \red{for $\nu^\NN$-a.e.\ $\omega$, the stable bundles $E^s_\omega (x)$ are non-random and the corresponding stable foliation $\{\cF^s_\omega\}$ is $C^r$-smooth.}
		\item %If both Lyapunov exponents of $\mu$ are equal to $A$: 
		If $\lambda^{+}(\mu_{\rm SRB})=\lambda^{+}(A)$ and $\lambda^{-}(\mu_{\rm SRB})=\lambda^{-}(A)$, then there exists $h\in\diff^{r-}(\TT^2)$ that simultaneously conjugates ever  $g$ in the support of $\nu$ to an affine Anosov diffeomorphism:
		$$
		h\circ g\circ h^{-1}=A+v_{g}, \qquad \nu\text{-a.e.}~g\in\diff^r(\TT^2),
		$$ 
		where $v_g\in\TT^2$.
	\end{itemize}
\end{Theorem}

\begin{Remark}
	In this theorem, the non-randomness of stable bundles and simultaneous lineariation actually holds for every $g\in\supp(\nu)$. Moreover, we can take $\nu$ supported in a $C^1$-neighborhood of $f$ but satisfies the integrability condition (\ref{eq:integrable}), then the result of our theorem still holds. The same conclusion holds in the rest theorems of this paper.
\end{Remark}

In higher dimensions, de la Llave \cite{dlL92} constructed examples of volume-preserving Anosov diffeomorphisms with reducible linearization and whose Lyapunov spectrum coincides with the Lyapunov spectrum of its linearization but are  not Lipschitz conjugated to their linearization.  To rule out examples like    de la Llave's, we introduce the following generic assumption for automorphisms in $\GL(d,\ZZ)$.

\begin{Definition}\label{def:generic}
	We say $A\in\GL(d,\ZZ)$ is \emph{generic} if it satisfies
	\begin{enumerate}
		\item $A$ is hyperbolic, i.e., the spectrum of $A$ is disjoint from the unit circle in $\mathbb{C}$;
		\item $A$ and $A^4$ are irreducible, i.e., have characteristic polynomials that are irreducible over $\QQ$;
		\item no three eigenvalues of $A$ have the same absolute value, and if two eigenvalues of $A$ have the same absolute value then they are a pair of complex conjugate eigenvalues. %\red{ that are not roots of unity??}
	\end{enumerate}
\end{Definition}
It was proved in \cite{GKS} that generic automorphisms $A\in\GL(d,\ZZ)$ have full probability in $\GL(d,\ZZ)$: 
$$
\lim_{K\to+\infty}
\frac{\sharp\{A\in\GL(d,\ZZ)~\text{is~generic}:~\|A\|\leq K\}}{\sharp\{A\in\GL(d,\ZZ):~\|A\|\leq K\}}=1.
$$

Let $A\in\GL(d,\ZZ)$ be a generic (in particular irreducible over $\QQ$) hyperbolic automorphism admitting the finest dominated splitting
\begin{align}\label{fine-splitting}
	T\TT^d=L^s_l\oplus\cdots L^s_1\oplus L^u_1\oplus\cdots L^u_k.
\end{align}
That is, the splitting in \eqref{fine-splitting} is $A$-invariant, the restriction of $A$ to each invariant subspace, $A|_{L^{s/u}_j}$, is conformal.  Moreover, setting $\lambda(A,L^{s/u}_j):= \log \|A|_{L^{s/u}_j}\|$,
we assume the subspaces are ordered so that 
%the Lyapunov exponents of $A$ in one bundle $L^s_i$ or $L^u_j$ are all equal and
$$
\lambda(A,L^s_{i+1})<\lambda(A,L^s_i)<0, \qquad 0<\lambda(A,L^u_j)<\lambda(A,L^u_{j+1}), 
\qquad i=1,\cdots,l-1,~j=1,\cdots,k-1.
$$
Let $A'\in\GL(d,\ZZ)$ be commuting with $A$.  Then each  $ L^{s/u}_j $ is $A'$-invariant.  
We say $A'$ admits the same finest dominated splitting s $A$ if the splitting (\ref{fine-splitting}) is also the finest dominated splitting of $A'$; that is, if 
$$
\lambda(A',L^s_{i+1})<\lambda(A',L^s_i)<0, \qquad 0<\lambda(A',L^u_j)<\lambda(A',L^u_{j+1}), 
\qquad i=1,\cdots,l-1,~j=1,\cdots,k-1.
$$

%
%
%
%t is clear for any $B$ commuting with $A$, there exists $N\in\NN$, such that for every $n>N$, $A^nB$ admits the same finest dominated splitting with $A$.

Let $\{A_i\}_{i=1}^m\subseteq\GL(d,\ZZ)$ be a family of generic commuting automorphisms with the finest dominated splitting of $A_1$ (\ref{fine-splitting}).  Then the dimension of every bundle $E^{s/u}_j$ is smaller than or equal to $2$. From commuting property, the splitting (\ref{fine-splitting}) is $A_i$-invariant for $i=2,\cdots,m$. Let $\nu_0$ be a probability on $\{A_i\}_{i=1}^m$ with $\nu_0(A_i)=p_i\in[0,1]$ with $\sum p_i=1$. Then for any $\nu_0$-stationary measure $\mu$, the Lyapunov exponents of $\mu$ in $L^{s/u}_j$ is equal to the linear combination
$$
\lambda(\mu,L^{s/u}_j)=\sum_{i=1}^mp_i\cdot\lambda(A_i,L^{s/u}_j).
$$
For instance, we can take $\mu$ to be the Lebesgue measure on $\TT^d$ which is $\nu_0$-invariant thus $\nu_0$-stationary.

We prove the following theorem for perturbations of commuting automorphisms admitting the same {finest  dominated splitting}, which implies Theorem \ref{thm:d-dim}. It is clear that perturbations of commuting automorphisms with same splitting share the common stable and unstable cone-fields. Theorem \ref{thm:SRB} implies every probability supported on these perturbations admit a stationary SRB-measure.

\begin{Theorem}\label{thm:commuting}
	Let $\{A_i\}_{i=1}^m\subseteq\GL(d,\ZZ)$ be a family of commuting generic automorphisms  {admitting the same dominated splitting.} 
	There exists a family of neighborhoods $A_i\in\cU_i\subseteq\diff^2(\TT^d)$ for $i=1,\cdots,m$, such that for every probability $\nu$ supported on $\cup\cU_i\subseteq\diff^2(\TT^d)$ satisfying
	$$
	\nu(\cU_i)=p_i\in[0,1] \qquad \text{and} \qquad
	\sum_{i=1}^mp_i=1,	
	$$
	let $\mu_{\rm SRB}$ be the $\nu$-stationary SRB-measure, and denote $\nu_0$ be the probability on $\{A_i\}_{i=1}^m$ where $\nu_0(A_i)=p_i$. 
	\begin{itemize}
		\item If the positive Lyapunov exponents of $\mu_{\rm SRB}$ coincide with the positive Lyapunov exponents of any $\nu_0$-stationary measure (with multiplicity),
there is a $C^{1+}$ foliation $\cF^s$ such that $\cF^s=\cF^s_{g}$ for $\nu_0$-a.e.\ $g\in\diff^r(\TT^d)$.  In particular, the stable foliation $\{\cF^s_\omega\}$ is $\nu^\NN$-almost surely independent of $\omega$ and is $C^{1+}$ smooth.  		
		
%		 then the stable bundle $E^s(f)$ is non-random for $\nu$-a.e. $f\in\diff^2(\TT^d)$.
		
		\item 
		If  all Lyapunov exponents of $\mu_{\rm SRB}$ coincide with all Lyapunov exponents of any $\nu_0$-stationary measure (with multiplicity), %If all Lyapunov exponents of $\mu_{\rm SRB}$ are equal to Lyapunov exponents of any $\nu_0$-stationary measure, 
		 then there exists $h\in\diff^{1+}(\TT^2)$ that simultaneously conjugates ever $f$ in the support of $\nu$ to an affine Anosov diffeomorphism:
		$$
		h\circ f\circ h^{-1}=A_i+v_{f}, \qquad \nu\text{-a.e.}~f\in\cU_i,
		$$ 
		where $v_f\in\RR^d$ and for every $i=1.\cdots,m$.
	\end{itemize}
\end{Theorem} 

\begin{Remark}
	Here the simultaneously linearization result holds 
	If we take $\supp(\nu)\subseteq\diff^{\infty}(\TT^d)$, then the conjugacy $h$ is also $C^\infty$-smooth by applying the recent work of Kalinin-Sadovskaya-Wang \cite{KSW1,KSW2}.
\end{Remark}

Theorems \ref{thm:2-dim} and  \ref{thm:commuting} are random analogues of the results of Saghin and Yang \cite{SY} and of Gogolev, Kalinin, and Sadovskaya \cite{GKS}.  See also the work of DeWitt \cite{DeW1}.   Indeed, in \cite{SY,GKS}, volume-preserving perturbations of (Lyapunov simple or generic) Anosov automorphisms are considered.  If the Lyapunov exponents of the perturbation coincide with the Lyapunov exponents of the automorphism, the corresponding conjugacy is shown to be smooth.

% proved similar smoothness of the conjugacy given coincidence of Lyapunov exponents for a volume-preserving perturbation; we view 
%of determined Anosov diffeomorphisms by assuming Lyapunov spectrum rigidity.} 
%We consider Lyapunov spectrum rigidity of random perturbation of a single Anosov diffeomorphism on torus.\note{i don't understand this sentence}  
%\blue{remark about \cite{SY} and G-K-S}.

\subsection{Spectrum rigidity of non-commuting automorphisms of $\TT^2$}

Let $\{A_i\}_{i=1}^m\subseteq\GL(2,\ZZ)$ be a famiy of hyperbolic non-commuting automorphisms with hyperbolic splitting $T\TT^2=L^s_i\oplus L^u_i$ of $A_i$. The non-commuting property implies there exist $i\neq j$ satisfying $L^s_i\neq L^s_j$ and $L^u_i\neq L^u_j$. 

\begin{Example}
	For every $n_i\in\NN_{\geq2}$, $i=1,\cdots,m$ and $n_i\neq n_j$ if $i\neq j$, let
	\begin{displaymath}
		A_i=\begin{bmatrix}
			n_i & n_i-1 \\
			1 & 1
		\end{bmatrix}.
	\end{displaymath}
	The family $\{A_i\}_{i=1}^m$ induces a family of hyperbolic non-commuting automorphisms on $\TT^2$. Moreover, $\{A_i\}_{i=1}^m$ share common stable and unstable cone fields on $\TT^2$ defined in Definition \ref{def:common-cone}. By taking the trivialization $T\TT^2=\TT^2\times\RR^2$, we can take
	\begin{align}\label{eq:cones}
		\cC^s=\TT^2\times\{(x,y)\in\RR^2:~x\cdot y\leq 0\} \qquad \text{and} \qquad
		\cC^u=\TT^2\times\{(x,y)\in\RR^2:~x\cdot y\geq 0\}.
	\end{align}
\end{Example}

%\begin{Definition}
%	Let $\cA\subseteq\diff^2(M)$ be a family diffeomorphisms. We say elements in $\cA$ share common stable and unstable cone-fields if there exist continuous cone fields $x\mapsto\cC^s_x$, $x\mapsto\cC^u_x$ for $x\in M$ and $\lambda>0$, such that for every $x\in M$, non-zero vectors $v^s\in\cC^s_x$, $v^u\in\cC^u_x$, and every $f\in\cA$,
%	\begin{itemize}
%		\item $Df^{-1}(\cC^s(x))\subset\cC^s_{f^{-1}(x)}$,  and $\|Df^{-1}(v_s)\|>e^\lambda\cdot\|v_s\|$;
%		\item $Df(\cC^u(x))\subset\cC^u_{f(x)}$, and $\|Df(v_u)\|>e^\lambda\cdot\|v_u\|$..
%	\end{itemize}
%\end{Definition} 

The following theorem states simultaneously linearization for perturbations of automorphisms which are non-commuting and admit common stable and unstable cone-fields, which implies Theorem \ref{thm:positive-T2}.

Define $\Upsilon\colon \diff^r(\TT^2) \to \GL(2,\ZZ)$ by the induced map on homology.  That is, $f$ is homotopic to the  map of $\TT^2$ induced by $\Upsilon (f)$.  
Then given a probability measure $\nu$ on  $\diff^r(\TT^2)$, we view $\Upsilon_*\nu$ as a probability on $\GL(2,\ZZ)$.  We let $\lambda^{\pm}(\Upsilon_*\nu)$ be the top and bottom Lyapunov exponents, respectively, of the i.i.d.\ random walk given by  Proposition \ref{prop:Lyapunov}; in particular, they are  the a.s.\ limit
$$\lambda^{\pm}(\Upsilon_*\nu):= \lim_{n\to \infty} \frac  1 n \log \|(A_{n-1} \dots  A_0)^{\pm 1}\|$$ where $A_i$ are chosen i.i.d.\ relative to $\Upsilon_*\nu$.   

%
% \blue{Recall we define $\Upsilon\colon \diff^r(\TT^2) \to \GL(2,\ZZ)\simeq \aut{\TT^2}$ as the induced map on homology.} 

\begin{Theorem}\label{thm:non-commuting}
	Let $\{A_i\}_{i=1}^m\subset\GL(2,\ZZ)$ $(m\geq2)$ be a family of non-commuting, cone hyperbolic automorphisms.  %sharing common stable and unstable cone-fields. 
	There exists a family of  neighborhoods $\cU_i\subseteq\diff^r(\TT^2)$ of $A_i$ for $i=1,\cdots,m$ and $r\geq2$, such that for every probability $\nu$ supported on $\cup\cU_i\subseteq\diff^r(\TT^2)$ satisfying
	$$
	\nu(\cU_i)=p_i\in(0,1) \qquad \text{and} \qquad
	\sum_{i=1}^mp_i=1,	
	$$
	there exists a unique $\nu$-stationary SRB-measure $\mu$. 
	
%	Denote $\nu_0$ be the probability on $\{A_i\}_{i=1}^m$ where $\nu_0(A_i)=p_i$, and $\lambda^+(\nu_0)$ the positive Lyapunov exponent of a $\nu_0$-stationary measure. 
	
	Then \begin{equation}\label{eq:bexandaleish}\lambda^+(\mu)\le \lambda^+(\Upsilon_*\nu).\end{equation} Moreover, equality holds in \eqref{eq:bexandaleish} if and only if there exists $h\in\diff^{r-}(\TT^2)$ that simultaneously conjugates every $f$ in the support of $\nu$ to an affine Anosov diffeomorphism:  %$\nu$-a.e. $f$ to an affine Anosov diffeomorphism:
	$$
	h\circ f\circ h^{-1}=A_i+v_{f}, \qquad f \in \supp(\nu), %\nu\text{-a.e.}~f\in\cU_i,
	$$ 
	where $v_f\in\RR^2$ and for every $i=1,\cdots,m$.
\end{Theorem}

\begin{Remark}
	When we take a family of positive matrices $\{A_i\}\subseteq\GL(2,\ZZ)$, then the stable and unstable cone-fields $\cC^s,\cC^u$ defined in (\ref{eq:cones}) are common stable and unstable cone-fields of $\{A_i\}$. Moreover, the family $\{A_i\}$ is irreducible if and only if there exist $A_i\neq A_j$ that do not commute. Thus, any irreducible family of positive matrices satisfies the assumption of Theorem \ref{thm:non-commuting} and Theorem \ref{thm:non-commuting} implies Theorem \ref{thm:positive-T2}.
\end{Remark}

\begin{Remark} 
We state all our perturbative results (Theorems \ref{thm:commuting}, \ref{thm:2-dim}, and \ref{thm:non-commuting}) for $C^2$-small perturbations of a family of Anosov diffeomorphisms of  $\diff^2(\TT^d)$.   All results continue hold if we instead take our perturbations to be $C^2$-bounded and $C^1$-small in $\diff^2(\TT^d)$.  We state results for $C^2$-perturbations for simplicity.  
%The random walk can be driven by a probability $\nu$ supported in a $C^1$-neighborhood of $A$ in $\diff^2(\TT^d)$ with finite first moment or perturbations of a family of commuting generic automorphisms admitting the same dominated splitting, see Theorem \ref{thm:commuting}.
\end{Remark}

%\section{\red{Invariance of stationary SRB measures}}

\section{Simultaneous linearization of random expanding maps}
In this section, we prove Theorem \ref{thm:S1-general}, which shows that Lyapunov spectrum rigidity of the stationary SRB measure implies simultaneous linearization of random expanding maps on $\SS^1$. The key fact is to show Lyapunov spectrum rigidity implies the stationary SRB measure is actually an invariant measure. The idea originates from \cite{LJ} and \cite{Ben}.

\begin{proof}[Proof of Theorem \ref{thm:S1-general}]
	Let $\nu$ be a probability measure on $E^r(\SS^1)$ $(r\geq2)$ which satisfies
	\begin{align}\label{equ:integrable}
		\int\log\|f\|_{C^2}d\nu(f)<\infty.
	\end{align}
 Theorem \ref{thm:expanding-SRB} and Remark \ref{rk:density} then show that there exists a unique $\nu$-stationary SRB measure $\mu_{\rm SRB}$ equivalent to the Lebesgue measure and with densities satisfying 
    $$
    q(x)=\frac{d\mu_{\rm SRB}}{d{\rm Leb}}(x)\in L^1(\SS^1,{\rm Leb}),
    \qquad \frac{1}{q(x)}\in L^1(\SS^1,\mu_{\rm SRB}).
    $$
	
	We consider the skew-product system $F_+:\Sigma_+\times\SS^1\to\Sigma_+\times\SS^1$ where
	$$
	F_+(\omega,x)=\left(\sigma(\omega),f_{\omega}(x)\right)=\left(\sigma(\omega),f_{\omega_0}(x)\right),
	\qquad \omega\in\Sigma_+=\left( E^r(\SS^1)\right)^{\NN},~x\in\SS^1.
	$$
	The probability measure $\nu^{\NN}\times\mu_{\rm SRB}$ is an ergodic $F_+$-invariant measure.
	
	Since $1/q(x)\in L^1(\SS^1,\mu_{\rm SRB})$ and $\nu$ satisfies the integrability condition (\ref{equ:integrable}), we have
	$$
	Q(\omega,x)=
	\frac{q(f_{\omega}(x))}{q(x)}\cdot\frac{|f_{\omega}'(x)|}{|\deg(f_{\omega})|}
	~\in~L^1(\nu^{\NN}\times\mu_{\rm SRB}).
	$$
	Moreover, we have
	\begin{align*}
		\int_{\Sigma_+}\int_{\SS^1}Q(\omega,x)
		d\mu_{\rm SRB}(x)d\nu^{\NN}(\omega)
		=&\int_{\Sigma_+}\int_{\SS^1}\frac{q(f_{\omega(x)})}{q(x)}
		\cdot\frac{|f_{\omega}'(x)|}{|\deg(f_{\omega})|}
		d\mu_{\rm SRB}(x)d\nu^{\NN}(\omega) \\
		=&\int_{\Sigma_+}\frac{1}{|\deg(f_{\omega})|}
		\left[\int_{\SS^1}q(f_{\omega}(x))|f_{\omega}'(x)|
		d{\rm Leb}(x)\right]d\nu^{\NN}(\omega) \\
		=&\int_{\Sigma_+}\frac{|\deg(f_{\omega})|}{|\deg(f_{\omega})|}d\nu^{\NN}(\omega)\\
		=&1.
	\end{align*}
	
	Since $\nu^{\NN}\times\mu_{\rm SRB}$ is ergodic with respect to $F_+$, for $\nu^{\NN}\times\mu_{\rm SRB}$-a.e. $(\omega,x)\in\Sigma_+\times\SS^1$, denote
	$$
	\mu_n=\frac{1}{n}\sum_{i=0}^{n-1}\delta_{F_+^i(\omega,x)},
	\qquad \text{we have} \qquad
	\lim_{n\to\infty}\int Qd\mu_n=1.
	$$
    Fix the generic point $(\omega,x)$ of $\nu^{\NN}\times\mu_{\rm SRB}$, it satisfies
    \begin{align}\label{eq:S1}
    	0&=\lim_{n\to\infty}\log\left(\frac{1}{n}\sum_{i=0}^{n-1}Q\circ F_+^i(\omega,x)\right) 
    	 \geq\lim_{n\to\infty}\frac{1}{n}\sum_{i=0}^{n-1}\log Q\circ F_+^i(\omega,x) \notag\\
    	 &=\lim_{n\to\infty}\frac{1}{n}\left[\log q(f_{\omega}^n(x))-\log q(x)\right]+
    	 \lim_{n\to\infty}\frac{1}{n}\sum_{i=0}^{n-1}
    	 \log\frac{\left|f_{\sigma^i(\omega)}'
    	 	\left(f_{\omega}^i(x)\right)\right|}{\left|\deg(f_{\sigma^i(\omega)})\right|}
    \end{align}

    \begin{Claim-numbered}\label{clm:n_k}
    	There exists an infinite subsequence $\{n_k\}\subseteq\NN$, such that
    	\begin{align*}
    		\lim_{k\to\infty}\frac{1}{n_k}\left[\log q(f_{\omega}^{n_k}(x))-\log q(x)\right]=0.
    	\end{align*}
    \end{Claim-numbered}
    
    \begin{proof}[Proof of the Claim]
    	Define $n_{k+1}=\inf\left\{m>n_k:~q(f_{\omega}^m(x))<\sqrt{m}\right\}$ with $n_0=0$, then we have 
    	\begin{align*}
    		\lim_{k\to\infty}\frac{1}{n_k}\left[\log q(f_{\omega}^{n_k}(x))-\log q(x)\right]
    		\leq\lim_{k\to\infty}\frac{1}{n_k}\left[\log\sqrt{n_k}-\log q(x)\right]=0.
    	\end{align*}
    	We only need to show that $n_k\to\infty$ as $k\to\infty$. Otherwise, for every $m\in\NN$ large enough,
    	$$
    	\frac{1}{q(f_{\omega}^m(x))}\leq\frac{1}{\sqrt{m}}.
    	$$
    	
    	Since the measure $\sum_{m=0}^{n-1}\delta_{f_{\omega}^m(x)}$ converges to $\mu_{\rm SRB}$ on $\SS^1$, we have
    	$$
    	1=\int_{\SS^1}\frac{{\rm Leb}}{\mu_{\rm SRB}}d\mu_{\rm SRB}
    	=\int_{\SS^1}\frac{1}{q(x)}d\mu_{\rm SRB}
    	=\lim_{n\to\infty}\sum_{m=0}^{n-1}\frac{1}{q(f_{\omega}^m(x))}
    	\leq\lim_{n\to\infty}\sum_{m=0}^{n-1}\frac{1}{\sqrt{m}}=0.
    	$$
    	This is absurd.
    \end{proof}
    
    We continue the proof of Theorem \ref{thm:S1-general}.
    From equation (\ref{eq:S1}) and Claim \ref{clm:n_k}, we have
    \begin{align*}
    	0\geq \lim_{k\to\infty}\frac{1}{n_k}\sum_{i=0}^{n_k-1}
    	\log\frac{\left|f_{\sigma^i(\omega)}'
    		\left(f_{\omega}^i(x)\right)\right|}{\left|\deg(f_{\sigma^i(\omega)})\right|}.
    \end{align*}
    This is equivalent to
    \begin{align*}
    	\lim_{k\to\infty}\frac{1}{n_k}\sum_{i=0}^{n_k-1}
    	\log\left|f_{\sigma^i(\omega)}'\left(f_{\omega}^i(x)\right)\right|
    	~\leq~
    	\lim_{k\to\infty}\frac{1}{n_k}\sum_{i=0}^{n_k-1}
    	\left|\deg(f_{\sigma^i(\omega)})\right|.
    \end{align*}
    Since $(\omega,x)$ is a generic point of $\nu^\NN\times\mu_{\rm SRB}$, the Lyapunov exponent of $\mu_{\rm SRB}$ satisfies
    \begin{align*}
    	\lambda(\mu_{\rm SRB})&=\lim_{k\to\infty}\frac{1}{n_k}\sum_{i=0}^{n_k-1}
    	\log\left|f_{\sigma^i(\omega)}'\left(f_{\omega}^i(x)\right)\right| \\
    	~&\leq
    	\lim_{k\to\infty}\frac{1}{n_k}\sum_{i=0}^{n_k-1}
    	\left|\deg(f_{\sigma^i(\omega)})\right|
    	=\int_{E^r(\SS^1)}|\deg(f)|d\nu(f).
    \end{align*}
    This proves the inequality in Theorem \ref{thm:S1-general}.
    
    \vskip3mm
    
    Finally, assume $\lambda(\mu_{\rm SRB})=\int_{E^r(\SS^1)}|\deg(f)|d\nu(f)$, we have the following claim.
    \begin{Claim-numbered}
    	If $\lambda(\mu_{\rm SRB})=\int_{E^r(\SS^1)}|\deg(f)|d\nu(f)$, then $\mu_{\rm SRB}$ is an $f$-invariant measure for $\nu$-a.e. $f\in E^r(\SS^1)$ and
    	$$
    	q(f(x))\cdot |f'(x)|=|\deg(f)|\cdot q(x), 
    	\qquad {\rm Leb}~\text{- a.e.}~x\in\SS^1.
    	$$
    \end{Claim-numbered}
    
    \begin{proof}[Proof of the Claim]
        Denote the set 
       $$
       G=\left\{(\omega,x)\in\Sigma_+\times\SS^1:
       ~Q(\omega,x)=\frac{q(f_{\omega}(x))}{q(x)}\cdot\frac{|f_{\omega}'(x)|}{|\deg(f_{\omega})|}\neq1\right\}.
       $$
       If $\nu^\NN\times\mu_{\rm SRB}(G)>0$, then there exists $a<0$, such that
       $$
       \int_{\Sigma_+\times\SS^1}\log(G)~d\nu^\NN\times\mu_{\rm SRB}~<~a.
       $$
       
       For the generic point $(\omega,x)$ of $\nu^\NN\times\mu_{\rm SRB}$ and $\mu_n=1/n\cdot\sum_{i=0}^{n-1}\delta_{F_+^i(\omega,x)}$, there exists $N>0$, such that
       $$
       \int_{\Sigma_+\times\SS^1}\log(G)~d\mu_n<\frac{a}{2},
       \qquad \forall n\geq N.
       $$
       From equation (\ref{eq:S1}), for every $n\geq N$, we have
       $$
       \frac{1}{n}\left[\log q(f_{\omega}^n(x))-\log q(x)\right]+
       \frac{1}{n}\sum_{i=0}^{n-1}
       \log\frac{\left|f_{\sigma^i(\omega)}'
       	\left(f_{\omega}^i(x)\right)\right|}{\left|\deg(f_{\sigma^i(\omega)})\right|}
       	<\frac{a}{2}.
       $$
       
       Taking the subsequence $\{n_k\}$ from Claim \ref{clm:n_k} and let $k\to\infty$, we have
       \begin{align*}
       	\lambda(\mu_{\rm SRB})&=\lim_{k\to\infty}\frac{1}{n_k}\sum_{i=0}^{n_k-1}
       	\log\left|f_{\sigma^i(\omega)}'\left(f_{\omega}^i(x)\right)\right| \\
       	&\leq
       	\lim_{k\to\infty}\frac{1}{n_k}\sum_{i=0}^{n_k-1}
       	\left|\deg(f_{\sigma^i(\omega)})\right|+\frac{a}{2} \\
       	&=\int_{E^r(\SS^1)}|\deg(f)|d\nu(f)+\frac{a}{2}.
       \end{align*}
       
       This is a contradiction. So we have $\nu^\NN\times\mu_{\rm SRB}(G)=0$, that is
       $$
       Q(\omega,x)=
       \frac{q(f_{\omega}(x))}{q(x)}\cdot\frac{|f_{\omega}'(x)|}{|\deg(f_{\omega})|}=1,
       \qquad
       \nu^\NN\times\mu_{\rm SRB} \text{~-~a.e.}~(\omega,x)\in\Sigma_+\times\SS^1. 
       $$
       This implies for $\nu$-a.e. $f\in E^r(\SS^1)$, it satisfies
       $$
       q(f(x))\cdot |f'(x)|=|\deg(f)|\cdot q(x), 
       \qquad {\rm Leb}~\text{- a.e.}~x\in\SS^1.
       $$
       This proves the claim.
    \end{proof}
    
    For $\nu$-a.e. $f\in E^r(\SS^1)$, since $\mu_{\rm SRB}$ is equivalent to Lebesgue measure and $f$-invariant, it is the unique SRB measure for $f$. Thus the density $q(x)$ is $C^{r-1}$-smooth which is bounded away from zero on $\SS^1$, see \cite{Sack,Krz}. From the smoothness of both $f$ and $q$, we have
    $$
      q(f(x))\cdot |f'(x)|\equiv q(x)\cdot|\deg(f)|, 
    \qquad \forall x\in\SS^1.
    $$
    
    Define 
    $$
    h(x)=\int_0^x~q(t)~d{\rm Leb}(t)~\in~\diff^r(\SS^1),
    $$
    it satisfies
    $$
    |\left(h\circ f\circ h^{-1}\right)'(x)|\equiv|\deg(f)|,
    \qquad \forall x\in\SS^1,~\nu~\text{-~a.e.}~f\in E^r(\SS^1).
    $$
    
    This implies for $\nu$-a.e. $f\in E^r(\SS^1)$, there exists $\rho_f\in\SS^1$ such that
    $$
    h\circ f\circ h^{-1}(x)\equiv \deg(f)\cdot x+\rho_f,
    \qquad \forall x\in\SS^1.
    $$
    This proves Theorem \ref{thm:S1-general}.
\end{proof}

\section{Skew product and fiberwise leaf conjugacy}

In this section, we show the structural stability and stability of intermediate stable and unstable foliations of fiberwise Anosov skew product systems with $\TT^d$-fibers. %%%%%%Furthermore,  Lyapunov \note{rewrote} exponents for fiberwise SRB measures satisfy and inequality;  moreover, equality implies the fiberwise conjugacy intertwines fiberwise SRB measures.  %Gibbs states of corresponding expanding foliations provided some 

\subsection{Structural stability of fiberwise Anosov systems}
As observed in Remark \ref{rem:cone hyperbolicity}, cone hyperbolicity is $C^1$-open.  %We first observe that preserving common stable and unstable cone-fields is an open property for a family of Anosov diffeomorphisms. The proof is standard. 

%%%\note{do we even need to state this at all?}
\begin{Lemma}\label{lem:robust-Anosov}
	Let $\cA=\{A_i\}_{i=1}^m\subseteq\GL(d,\ZZ)$ be a cone hyperbolic family of automorphisms.   Let $\Omega_0=\cA^\ZZ$ and let 
	$$
	A:\Omega_0\times \TT^d\to\Omega_0\times \TT^d, \qquad
	A(\omega,x)=\left(\sigma(\omega),f_\omega(x)\right)
	$$ 
	be the fiberwise Anosov skew product.
	
	 Then there exists a   pairwise disjoint collection of neighborhoods $\cU_i\subseteq\diff^2(\TT^d)$ of $A_i$  such that, with $\cU=\cup\cU_i$ and $\Omega=\cU^\ZZ$, the corresponding skew product 
		$$
		F:\Omega\times \TT^d\to\Omega\times \TT^d, \qquad
		F(\omega,x)=\left(\sigma(\omega),f_\omega(x)\right)
		$$ 
	is also fiberwise Anosov. 
\end{Lemma}	

We introduce a number of skew systems associated with perturbations of $A$.  % and 

%ethe extension of $A_0:\Omega_0\times M\to\Omega_0\times M$ to $\Omega\times M$ which allowed us to state the structural stability of fiberwise Anosov skew product.

%\note{swap order of definition and lemma}
%Recall the map $\pi\colon \cU\subseteq\diff^2(\TT^d)\to \GL(2, \ZZ)$ given by $\pi
\begin{Definition}\label{def:extension}
Let $\cA=\{A_i\}_{i=1}^m\subset \GL(d,\ZZ)\subseteq\diff^2(\TT^d)$ and let $\Omega_0:=\cA^\ZZ$.  
Given $$\omega=(\dots, M_{-1}, M_0, M_1,  \dots) \in \Omega_0,$$ let $A_\omega = M_0$.
Let 
$A:\Omega_0\times \TT^d\to\Omega_0\times \TT^d$ be the induced skew-product give by $$A(\omega,x)=\left(\sigma(\omega),A_\omega(x)\right).$$

Given an open neighborhood $\cU\subseteq\diff^2(\TT^d)$ of $\cA$, recall the map $\Upsilon\colon \cU\subseteq\diff^2(\TT^d)\to \GL(2, \ZZ)$ given by the induced map on homology.  Write $\Omega=\cU^\ZZ$.  
Having taken $\cU$ sufficiently small, for $\omega=(\dots, f_{-1}, f_0, f_1,\dots) \dots \in\Omega$, also  write $\Upsilon\colon \Omega\to \Omega_0$ for the induced map, \begin{equation}
\label{eq:basemap} \Upsilon(\omega) := \bigl(\dots, \Upsilon(f_{-1}), \Upsilon(f_{0}), \Upsilon(f_{1}),\dots  \bigr).\end{equation}
%	For every $\omega\in\Omega$ and $j\in\ZZ$, if $f_{\omega_j}\in\cU_{i_j}$ for $i_j\in\{1,\cdots,m\}$, then we define $f_{\omega'_j}=f_{i_j}\in\cA$. The map 
%	$$
%	\omega\mapsto\omega'  \qquad   \text{maps} \qquad
%	\prod_j\cU_{i_j}\mapsto(f_{i_j})\in\Omega_0=\cA^\ZZ.
%	$$ 
%	Denote $\bar{A}:\Omega\times M\to\Omega\times M$ the extension of $A_0$ defined as
%	$$
%	\bar{A}(\omega,x)=
%	\left(\sigma(\omega), \bar{A}_{\omega}(x)\right):=
%	\left(\sigma(\omega),f_{\omega'}(x)\right),
%	\qquad \text{where} \qquad
%	f_{\omega'}=f_i\in\cA\quad \text{if}\quad f_{\omega_0}\in\cU_i.
%	$$

For $\cU$ as above and $\Omega=\cU^\ZZ$, let $\bar{A}:\Omega\times\TT^d\to\Omega\times\TT^d$ be the extension of $A$ over the shift $\Omega$:
  $$
  \bar{A}(\omega,x)=\left(\sigma(\omega),\bar{A}_{\omega}(x)\right)
  =\left(\sigma(\omega),\Upsilon (f_\omega)(x)\right).
  $$
  We also extend $\Upsilon \colon \Omega\to \Omega_0$ to $\Upsilon\colon \Omega\times \TT^d \to \Omega_0\times \TT^d $  by 
  $$\Upsilon (\omega, x) = (\Upsilon (\omega), x).$$
%\note{added Def of $\bar{A}$ here}
\end{Definition}		

Fiberwise Anosov systems satisfy fiberwise structural stability under sufficiently small perturbations.
For individual Anosov diffeomorphisms this is a classical standard result (see \cite[Theorems 18.2.1 and 19.1.2]{KH}). For structural stability of random perturbations of a single Anosov (or Axiom A) diffeomorphism, see \cite{Liu}.  
 The proof for perturbations of the fiberwise Anosov system $\bar{A}$ is a straightforward adaptation of the classical arguments as in \cite{Liu}.

\begin{Proposition}\label{prop:structure-stable}
	Let $\cA=\{A_i\}_{i=1}^m\subset \GL(d,\ZZ)\subseteq\diff^2(\TT^d)$ be a family of Anosov diffeomorphisms with common stable and unstable cone-fields and let $\Omega_0=\cA^\ZZ$.  The fiberwise Anosov skew product $A:\Omega_0\times M\to\Omega_0\times M$ %where $A_0(\omega,x)=\left(\sigma(\omega),f_\omega(x)\right)$ 
	is structurally stable in the following sense:

Fix the neighborhoods  $\cU$ of $\cA$ sufficiently small  so that the skew product $F:\Omega\times \TT^d\to\Omega\times \TT^d$ is fiberwise Anosov.
	Then there exists a  homeomorphism  fibering over the identity %$\Upsilon\colon \Omega\to \Omega_0$,
	$$
	H:\Omega\times \TT^d\to\Omega\times \TT^d, \qquad
	H(\omega,x)=\left(\omega, H_\omega(x)\right),
    $$
	which satisfies the following properties:
	\begin{enumerate}
		\item $H$ is a continuous conjugacy $H\circ F=\bar{A}\circ H:\Omega\times \TT^d\to\Omega\times \TT^d$ which descends to a continuous semiconjugacy
		$(\Upsilon \circ H)\circ F={A}\circ (\Upsilon \circ H):\Omega\times \TT^d\to\Omega_0\times \TT^d$. 
%%	\red{Need global Holderness???}
		
		\item For every $\omega\in \Omega$, $H_\omega\colon \TT^d\to  \TT^d $ is a bi-H\"older homeomorphism such that for every $x\in \TT^d$,	
%		topological conjugacy $H\circ F=\bar{A}\circ H:\Omega\times M\to\Omega\times M$, i.e.  
			$$
			H_{\sigma(\omega)}\left(f_{\omega}(x)\right)
			%%=A_{\Upsilon(\omega)}\circ H_{\omega}(x)
			=\bar{A}_{\omega}\left(H_{\omega}(x)\right)
			={A}_{\Upsilon(\omega)}\left(H_{\omega}(x)\right).
%			\qquad \forall (\omega,x)\in\Omega\times M,
			$$

		\item For every $\omega\in\Omega$ %and the fiber $M_\omega=\{\omega\}\times M$, 
		the map $H_\omega:\TT^d \to \TT^d$ is uniformly (in $\omega$) $C^0$ close to the identity map and intertwines the stable and unstable foliations:
			$$
			H_\omega\left(\cF^s_{\omega}\right)=\cF^s_{\Upsilon(\omega)},
			\qquad \text{and} \qquad
			H_\omega\left(\cF^u_{\omega,}\right)=\cF^u_{\Upsilon(\omega)}.%\omega,f_{\omega'}},	
			$$
	\end{enumerate}

%	Let $\bar{A}:\Omega\times M\to\Omega\times M$ be the extension of $A_0$. 
%	Then there exists a  continuous surj 
%	$$
%	H:\Omega\times M\to\Omega\times M, \qquad
%	H(\omega,x)=\left(\omega,H_\omega(x)\right),
%    $$
%	which satisfying the following properties:
%	\begin{enumerate}
%		\item $H$ is a topological conjugacy $H\circ F=\bar{A}\circ H:\Omega\times M\to\Omega\times M$, i.e.  
%			$$
%			H_{\sigma(\omega)}\left(f_{\omega}(x)\right)
%			=\bar{A}_{\omega}\circ H_{\omega}(x)
%			=f_{\omega'}\left(H_{\omega}(x)\right),
%			\qquad \forall (\omega,x)\in\Omega\times M,
%			$$
%			and H\"older continuous with respect to $(\omega,x)\in\Omega\times M$.
%		\item For every $\omega\in\Omega$ and the fiber $M_\omega=\{\omega\}\times M$, $H_\omega:M_\omega\to M_{\omega}$ is close to identity map on $M_\omega$, and preserves the stable and unstable foliations:
%			$$
%			H_\omega\left(\cF^s_{\omega,f_\omega}\right)=\cF^s_{\omega,f_{\omega'}},
%			\qquad \text{and} \qquad
%			H_\omega\left(\cF^u_{\omega,f_\omega}\right)=\cF^u_{\omega,f_{\omega'}},		    
%			$$
%	\end{enumerate}
\end{Proposition}
 Given $\omega= (\dots, \omega_{-1}, \omega_0,\omega_1,\dots), \xi=(\dots, \xi{-1}, \xi_0,\xi_1,\dots)\in \Omega,$ we define the Smale bracket 
	$$[\omega,\xi]:= (\dots, \omega_{-2},\omega_{-1},\xi_0,\xi_1,\dots).$$
We have \begin{equation}
\cF^s_{\omega}= \cF^s_{[\xi,\omega]}, \quad \quad \cF^u_{\omega}= \cF^u_{[\omega,\xi]}.
\end{equation}	
We remark that (the proof of) Proposition \ref{prop:structure-stable} gives the following: there is $R>0$ for any $(\omega,x),(\xi,y)\in \Omega\times \TT^d$, there is $$z\in  \cF^u_{\omega}(x) \cap \cF^s_{\xi}(y)= \cF^u_{[\omega,\xi]}(x) \cap \cF^s_{[\omega,\xi]}(y)$$
such that, relative to the induced Riemannian distances $d_{\cF^u_\omega}(\cdot,\cdot)$ and $d_{\cF^s_\omega}(\cdot,\cdot)$ on immersed submanifolds, 
$$d_{\cF^u_{\omega}}(x,z)\le R,\quad \quad d_{\cF^s_{\xi}}(y,z)\le R.$$

Since both the skew product systems $F:\Omega\times M\to\Omega\times M$ and extension ${A}:\Omega_0\times M\to\Omega_0\times M$ are fiberwise Anosov over a full shift whose alphabet is a subset of Polish  (complete, separable, second countable) space, % system $\sigma:\Omega\to\Omega$, 
the skew producs satisfy the Anosov Closing Lemma and Liv\v{s}ic Theorem. We first recall the definition of Anosov Closing Lemma.

%\begin{Definition}\label{def:Anosov-closing}
%	Let $(X,d)$ be a compact metric space and $f:X\to X$ be a homeomorphism. We say  $f:X\to X$ satisfies Anosov Closing Lemma if there exists constant $c,\delta_0,\gamma>0$ such that for any $x\in X$ and $k\in\NN$ with $d\left(x,f^k(x)\right)<\delta_0$, then there exists a periodic point $p\in X$ with $f^k(p)=p$ such that the orbit segments $x,f(x),\cdots,f^k(x)$ and $p,f(p),\cdots,f^k(p)$ remain exponential close
%	$$
%	d\left(f^i(x),f^i(p)\right)\leq 
%	c\cdot\exp\left(-\gamma\cdot\min\{i,k-i\}\right)\cdot d\left(x,f^k(x)\right),
%	\qquad \forall i=0,\cdots,k.
%	$$
%\end{Definition}

\begin{Lemma}\label{lem:livsic}\label{def:Anosov-closing}
	The fiberwise Anosov skew product system $F:\Omega\times M\to\Omega\times M$ satisfies the Anosov Closing Lemma and Liv\v{s}ic Theorem:
	\begin{enumerate}
\item There exist constants $c,\delta_0,\gamma>0$ such that for any $x\in X$ and $k\in\NN$ with $d\left(x,f^k(x)\right)<\delta_0$, then there exists a periodic point $p\in X$ with $f^k(p)=p$ such that the orbit segments $x,f(x),\cdots,f^k(x)$ and $p,f(p),\cdots,f^k(p)$ remain exponential close
	$$
	d\left(f^i(x),f^i(p)\right)\leq 
	c\cdot\exp\left(-\gamma\cdot\min\{i,k-i\}\right)\cdot d\left(x,f^k(x)\right),
	\qquad \forall i=0,\cdots,k.
	$$

	\item If $\varphi,\psi:\Omega\times M\to\RR$ are two bounded H\"older continuous function satisfying 
	$$
	\sum_{i=0}^{l-1}\varphi\circ F^i(\omega,x)=\sum_{i=0}^{l-1}\psi\circ F^i(\omega,x),
	\qquad \forall (\omega,x)~\text{with $F^l(\omega,x)=(\omega,x)$},
	$$
	then there exists a H\"older continuous function $u:\Omega\times M\to\RR$ such that
	$$
	\varphi=\psi+u\circ F-u.
	$$
	Moreover, $u$ is unique up to a constant.
	\end{enumerate}
	The extension system $\bar{A}:\Omega\times M\to\Omega\times M$ also satisfies both the Anosov Closing Lemma and the Liv\v{s}ic Theorem.
\end{Lemma}

%\begin{proof}
%	Recall that hyperbolic systems, including transitive Anosov diffeomorphisms, topologically mixing diffeomorphisms of locally maximal hyperbolic sets, and mixing subshifts of finite type, all satisfy Anosov Closing Lemma, see \cite[6.4.15–17]{KH}.
	
We remark that  the skew product systems $F:\Omega\times \TT^d\to\Omega\times \TT^d$ is fiberwise Anosov over a full shift system $\sigma:\Omega\to\Omega$ and thus has local product structure % exponentially shadowing property 
along both the fiber and base directions.  This  implies the Anosov Closing Lemma holds for $F$. We remark that since the alphabet of $\Omega$ is second countable, the full shift $\omega\colon \Omega\to \Omega$ is topologically transitive and thus the skew product systems $F:\Omega\times \TT^d\to\Omega\times \TT^d$  is topologically transitive.   
The  Liv\v{s}ic Theorem  then follows from  the Anosov Closing Lemma and the existence of a dense orbit.    %The proof of extension system $\bar{A}$ is the same.
%\end{proof}

%%%%\begin{Remark}\label{rk:livsic}
%%%%	The key point of the proof is that $F$ is fiberwise Anosov over a full shift system $\sigma:\Omega\to\Omega$.
%%%%	The same conclusion of Liv\v{s}ic Theorem also holds for a subsystem of $F$. That is for some subset $\cK\subseteq\cU$ and $\Omega_\cK=\cK^\ZZ$, then for the restriction system $F:\Omega_\cK\times M\to\Omega_\cK\times M$, if we have $\varphi,\psi:\Omega_\cK\times M\to\RR$ have equal sums along all periodic orbits in $\Omega_\cK\times M$, then there exists H\"older $u:\Omega_\cK\times M\to\RR$ such that $\varphi=\psi+u\circ F-u$.
%%%%\end{Remark}

\subsection{Linear foliations on $\TT^d$}

Note that $\RR^d$ acts transitively by addition on  $\TT^d$.  
By  linear foliation $\cL$ on $\TT^d$, we mean a foliation  
whose leaves $\cL(x)$ are of the form $$\cL(x):= x+ E_\cL$$ 
for a choice of vector sub-space $E_\cL\subset \RR^d$.   We naturally identify $T_x \cL(x) \simeq E_\cL$ for every $x\in \TT^d$.    
We state some lemmas that describe sufficient criteria that ensure a foliation on $\TT^d$ is linear.  
Firstly, we have the following lemma from \cite{GoS}.

By a $C^0$-foliation, we mean only a topological foliation.  In particular, leaves are only assumed to be continuous submanifolds.  
Given two  $C^0$-foliations $\cL$ and $\cF$, a sufficient criterion for $\cL$ and $\cF$ to be topologically transverse is that there exists local charts relative to which the preimages of $\cL$ and $\cF$ have $C^1$ leaves that are everywhere transverse (and of complementary dimension).
A foliation $\cF$ is minimal if every leaf is dense.

\begin{Lemma}[{\cite[Lemma 2.1]{GoS}}]\label{lem:joint-minimal}
	Let $\cF$ be a $C^0$-foliation on $\TT^d$ that is sub-foliated by a minimal linear foliation $\cL$. Then $\cF$ is a minimal linear foliation on $\TT^d$.
\end{Lemma}

Now we have the following lemma, which shows the parallel translation of holonomy maps implies linearity of foliations on $\TT^d$.

%\begin{Lemma}\label{lem:linear-Td}
%	Let $\cL$ be a $j$-dimensional minimal linear foliation on $\TT^d$ and let $\cF$ be an $l$-dimensional $C^0$-foliation on $\TT^d$ satisfying the following:
%	\begin{itemize}
%		%\item There exists a homoeomorphism $h:\TT^d\to\TT^d$ homotopic to identity $\id_{\TT^d}$, such that $h(\cF)$ is a linear foliation on $\TT^d$;
%		\item $\cF$ is jointly integrate with $\cL$ to a foliation $\cF\oplus\cL$ which is a minimal linear foliation on $\TT^d$ by Lemma \ref{lem:joint-minimal}, and $\cF$ is topologically transverse to $\cL$ in every leaf of $\cF\oplus\cL$;
%		\item for every $x\in\TT^d$ and $y\in\cF(x)$, the holonomy map $\Hol^{\cF}_{x,y}:\cL(x)\to\cL(y)$ induced by $\cF$ inside the leaf $\cG$ is $C^1$-smooth and coincides with the  identity under the identification $T_\bullet \cL(\bullet)=E_\cL$:
%		$$
%		D\Hol^{\cF}_{x,y}=\id:~T_x\cL(x)=E_\cL~\longrightarrow~T_y\cL(y)=E_\cL.  
%		$$
%		%\item For every leaf $\cF\oplus\cL(x)$, let $\tilde{\cF}\oplus\tilde{\cL}(\tilde{x})$ be a lift in $\RR^d$, then
%		%the corresponding lifting foliations $\tilde{\cF}$ and $\tilde{\cL}$ has global product structure, i.e.
%		%for every $\tilde{y},\tilde{z}\in\tilde{\cF}\oplus\tilde{\cL}(\tilde{x})$, $\tilde{\cF}(\tilde{y})$ intersects $\tilde{\cL}(\tilde{z})$ with a unique point.
%	\end{itemize}
%	Then $\cF$ is a linear foliation.
%\end{Lemma}

\begin{Lemma}\label{lem:linear-Td}
	Let $\cL$ be a  minimal linear foliation on $\TT^d$ and let $\cF$ be a $C^0$-foliation on $\TT^d$ satisfying the following:
	\begin{itemize}
		%\item There exists a homoeomorphism $h:\TT^d\to\TT^d$ homotopic to identity $\id_{\TT^d}$, such that $h(\cF)$ is a linear foliation on $\TT^d$;
		\item there is a minimal linear foliation $\cG$ fully saturated by leaves of $\cF$ and leaves of  $\cL$; moreover, restricted to every leaf of $\cG$ the foliations $\cF$ and $\cL$ are  topologically transverse; % and \red{of complementary dimension;}  %to $\cL$ in every leaf of $\cF\oplus\cL$;
		\item for every $x\in\TT^d$ and $y\in\cF(x)$, the holonomy map $\Hol^{\cF}_{x,y}:\cL(x)\to\cL(y)$ induced by $\cF$ inside the leaf $\cF\oplus\cL(x)$ is $C^1$-smooth and the derivative is identity 
		$$
		D\Hol^{\cF}_{x,y}=\id:~T_x\cL(x)=E_\cL~\longrightarrow~T_y\cL(y)=E_\cL.  
		$$
		%\item For every leaf $\cF\oplus\cL(x)$, let $\tilde{\cF}\oplus\tilde{\cL}(\tilde{x})$ be a lift in $\RR^d$, then
		%the corresponding lifting foliations $\tilde{\cF}$ and $\tilde{\cL}$ has global product structure, i.e.
		%for every $\tilde{y},\tilde{z}\in\tilde{\cF}\oplus\tilde{\cL}(\tilde{x})$, $\tilde{\cF}(\tilde{y})$ intersects $\tilde{\cL}(\tilde{z})$ with a unique point.
	\end{itemize}
	Then $\cF$ is a linear foliation.
\end{Lemma}

%%%%[AB: Think if there is a clearer way to say this]
%%%\begin{proof}
%%%	We only need to show that for every $x_0\in\TT^d$, the local leaf $\cF_{\loc}(x_0)$ is locally an affine subspace with the same linear part, i.e.
%%%	for every $y_0,z_0\in\cF_{\loc}(x)$, we have $y_0+z_0-x_0\in\cF_{\loc}(x_0)$ and for every $x_0,x_1\in\TT^d$
%%%	$$
%%%	\{y_0-x_0:~y_0\in\cF_{\loc}(x_0)\}=\{y_1-x_1:~y_1\in\cF_{\loc}(x_1)\}.
%%%	$$
%%%	
%%%	Since $D\Hol^{\cF}_{x,y}=\id$ for every $x\in\TT^d$, the holonomy map $\Hol^{\cF}_{x_0,y_0}:\cL(x_0)\to\cL(y_0)$ is a parallel translation. From $y_0\in\cF_{\loc}(x_0)$, then for every $x\in\cL(x_0)$ we have 
%%%	$$
%%%	\Hol^{\cF}_{x_0,y_0}(x)=x+(y_0-x_0)\in\cF_{\loc}(x)\cap\cL(y_0).
%%%	$$
%%%	From the density of $\cL(x_0)$ and continuity of $\cF$, we have 
%%%	$$
%%%	x+(y_0-x_0)\in\cF_{\loc}(x), \qquad \forall x\in\TT^d.
%%%	$$
%%%	
%%%	Similarly, we have $x+(z_0-x_0)\in\cF_{\loc}(x)$ for every $x\in\TT^d$. This implies 
%%%	$$
%%%	y_0+(z_0-x_0)\in\cF_{\loc}(y_0)=\cF_{\loc}(x_0).
%%%	$$
%%%	This proves $\cF_{\loc}(x_0)$ is locally a linear subspace for every $x_0\in\TT^d$. Moreover, our proves shows that if $y_0\in\cF_{\loc}(x_0)$, then $x+(y_0-x_0)\in\cF_{\loc}(x)$ for every $x\in\TT^d$. This implies 
%%%	$$
%%%	\{y_0-x_0:~y_0\in\cF_{\loc}(x_0)\}=\{y_1-x_1:~y_1\in\cF_{\loc}(x_1)\}.
%%%	$$
%%%	Thus $\cF$ is a linear foliation.
%%%\end{proof}

\begin{proof}
Fix $x_0\in \TT^d$ and let $\gamma\colon [0,1]\to \RR^d$ be a $C^0$ curve such that $x_0+ \gamma(t)\in \cF(x_0)$ for every $t\in [0,1]$.  
Since the holonomy map $\Hol^{\cF}_{x_0,x_0+ \gamma(t)}:\cL(x_0)\to\cL(x_0+ \gamma(t))$ is the identity, for every  $x\in\cL(x_0)$ we have 
	$$
	\Hol^{\cF}_{x_0,x_0+ \gamma(t)}(x)=x+\gamma(t) \in \cF(x)\cap\cL(x_0+ \gamma(t)).
	$$
Thus $x+\gamma(t) \in\cF_{}(x)$ for every $x\in \cL(x_0)$.  The family of maps $t\mapsto x+ \gamma(t)$ parameterized by $x\in \cL(x_0)$ is equicontinuous; by the density of $\cL(x_0)$ and continuity of $\cF$, 
	$$
	x+\gamma(t) \in\cF_{}(x), \qquad \forall x\in\TT^d, t\in [0,1].
	$$
It follows that if $\gamma,\gamma'\colon [0,1]\to \RR^d$ are two $C^0$ curves such that $x_0+ \gamma'(s)\in \cF_{}(x_0)$ and $x_0+ \gamma(t)\in \cF_{}(x_0)$  for every $s,t\in [0,1]$, then for every 
$$
	x+\gamma'(s)+ \gamma(t) \in\cF_{}(x), \qquad \forall x\in\TT^d,  s,t\in [0,1].
	$$
Let $V$ be the closed subgroup of $\RR^d$ containing all $\gamma(t)$ where $\gamma\colon [0,1]\to \RR^d$ is a $C^0$ curve such that $x_0+ \gamma(t)\in \cF_{}(x_0)$ for every $t\in [0,1]$.  Then $V$ is connected, hence linear, and $x+v\in \cF(x)$ for all $v\in V$ and $x\in \TT^d$.
\end{proof}

\subsection{Fiberwise leaf conjugacy on $\TT^d$} %-AWB version}
 In this subsection, we consider a family of commuting automprphisms $\{A_i\}_{i=1}^m\subseteq\GL(d,\ZZ)$ acting on $\TT^d$ admitting the same finest dominated splitting
  $$
  	T\TT^d=L^s_l\oplus\cdots L^s_1\oplus L^u_1\oplus\cdots L^u_k.  %\note{reverse s order}
  $$
  Denote $\cA=\{A_i\}_{i=1}^m$ and $\Omega_0=\cA^\ZZ$. Let 
  $$
  A:\Omega_0\times\TT^d\to\Omega_0\times\TT^d,
  \qquad
  A(\omega,x)=\left(\sigma(\omega),A_{\omega}(x)\right)
  $$ 
  be the corresponding fiberwise Anosov skew product.  
  
  Following Proposition \ref{prop:structure-stable}, let $\{\cU_i\}$ be a (pairwise disjoint) family of neighborhoods of $\{A_i\}$ respectively, such that for $\cU=\cup\cU_i\subseteq\diff^2(\TT^d)$ and with $\Omega=\cU^\ZZ$, the corresponding skew-product $F:\Omega\times\TT^d\to\Omega\times\TT^d$ is fiberwise Anosov.

Let $\bar{A}:\Omega\times\TT^d\to\Omega\times\TT^d$ denote the extension of $A$ in Definition \ref{def:extension}.
%over the shift $\Omega$:
%  $$
%  \bar{A}(\omega,x)=\left(\sigma(\omega),\bar{A}_{\omega}(x)\right)
%  =\left(\sigma(\omega),A_{\omega'}(x)\right),
%  \qquad \text{where} \qquad
%  A_{\omega'}=A_i\in\cA\quad \text{if}\quad f_{\omega}=f_{\omega_0}\in\cU_i.
%  $$
  Since the fiberwise dynamics of $\bar{A}$ are compositions of $\{A_i\}_{i=1}^m$, we have the following lemma.
%
%  In this subsection, we consider a family of commuting automorphisms $\{A_i\}_{i=1}^m\subseteq\GL(d,\ZZ)$ acting on $\TT^d$ admitting the same finest dominated splitting
%  $$
%  	T\TT^d=L^s_1\oplus\cdots L^s_l\oplus L^u_1\oplus\cdots L^u_k.
%  $$
%  Denote  by $\cA=\{A_i\}_{i=1}^m$ and $\Omega_0=\cA^\ZZ$. Let 
%  $$
%  A_0:\Omega_0\times\TT^d\to\Omega_0\times\TT^d,
%  \qquad
%  A_0(\omega,x)=\left(\sigma(\omega),A_{\omega}(x)\right)
%  $$ 
%  be the corresponding fiberwise Anosov skew product.  

%  Let $\bar{A}:\Omega\times\TT^d\to\Omega\times\TT^d$ be the extension of $A_0$:
%  $$
%  \bar{A}(\omega,x)=\left(\sigma(\omega),\bar{A}_{\omega}(x)\right)
%  =\left(\sigma(\omega),A_{\omega'}(x)\right),
%  \qquad \text{where} \qquad
%  A_{\omega'}=A_i\in\cA\quad \text{if}\quad f_{\omega}=f_{\omega_0}\in\cU_i.
%  $$
%  Since the fiberwise dynamics of $\bar{A}$ are compositions of $\{A_i\}_{i=1}^m$, we have the following lemma.
  
  \begin{Lemma}\label{lem:linear-foliation}
  	The  skew product $\bar{A}:\Omega\times\TT^d\to\Omega\times\TT^d$ admits the fiberwise finest dominated splitting by identifying
  	$\TT^d=\TT^d_\omega=\{\omega\}\times\TT^d$:
%\note{changed order index of stables}  
	$$
  	T\TT^d_\omega
  	=L^s_{l,\omega}\oplus\cdots L^s_{1\omega}\oplus L^u_{1,\omega}\oplus\cdots\oplus L^u_{k,\omega}
  	=L^s_l\oplus\cdots L^s_1\oplus L^u_1\oplus\cdots\oplus L^u_k.
  	$$
  	For any $1\leq i\leq j\leq k$, there exists a family of $\bar{A}$-invariant fiberwise linear foliations $\{\cL^u_{(i,j),\omega}\}_{\omega\in\Omega_0}$, such that 
%  	$\cL^u_{(i,j),\omega}$ coincides the linear foliation $\cL^u_{(i,j)}$ on $\TT^d_\omega$:
%  	$$
%  	\left\{\cL^s_{(i,j),\omega}=\cL^s_{(i,j)}\right\}_{\omega\in\Omega_0},
%  	$$
%  	and satisfies
  	$$
  	T\cL^u_{(i,j),\omega}=L^u_{(i,j),\omega}
  	=L^u_{i,\omega}\oplus\cdots\oplus L^u_{j,\omega} 
  	=L^u_i\oplus\cdots\oplus L^u_j .
  	$$ 
  	The same holds for the stable bundle and corresponding stable foliations $\{\cL^s_{(i,j),\omega}\}_{\omega\in\Omega}$ and any $1\leq i\leq j\leq l$.
  \end{Lemma}

%  The fiberwise dynamics of $\bar A$ are compositions of commuting $\{A_i\}_{i=1}^m$.
  For every $1\leq j<k$, the splitting $T\TT^d=(L^s\oplus L^u_{(1,j)})\oplus L^u_{(j+1,k)}$ is a dominated splitting for every $A_i$, $i=1,\cdots,m$. 
By taking an adapted Riemannian metric on $\TT^d$, there exist constants  $\{\lambda^u_{j,i}>0\}_{i=1}^m$, $\{\lambda^u_{j+1,i}>0\}_{i=1}^m$, and  $\epsilon>0$, such that
\begin{equation}\label{eq:unfhyp}
  \|A_i|_{L^u_{(1,j)}}\|<\exp(\lambda^u_{j,i}+\epsilon/2)<
  \exp(\lambda^u_{j+1,i}-\epsilon/2)<
  m(A_i|_{L^u_{(j+1,k)}}),
  \qquad \forall 1\leq i\leq m.
  \end{equation}
  Here $m(\cdot)$ denotes the conorm of linear operators, defined by $m(L)=\|L^{-1}\|^{-1}$.  We note $\epsilon>0$ can be taken  independent of $i,j$.

%  The following lemma comes from the definition of $\bar{A}$ and $\lambda^n_{j,\omega}$.

%  For every $\omega\in\Omega=(\cup\cU_i)^\ZZ$, we define
%  \begin{align}
%  	\lambda^u_{j,\omega}:&=\lambda^u_{j,i} \qquad
%  	\text{if} \qquad f_{\omega_0}\in\cU_i, 
%  	\qquad \text{and}   \notag\\
%  	\lambda^u_{j,\omega}(n):&=
%  	\lambda^u_{j,\sigma^{n-1}(\omega)}+\lambda^u_{j,\sigma^{n-2}(\omega)}
%  	+\cdots+
%  	\lambda^u_{j,\sigma(\omega)}+\lambda^u_{j,\omega}, \qquad \forall n\geq1.
%  \end{align}
%  The following lemma comes from the definition of $\bar{A}$ and $\lambda^n_{j,\omega}$.
  For every $\omega\in\Omega$,  and $n\ge 1$, define %=(\cup\cU_i)^\ZZ$, we define
  \begin{align}
%  	\lambda^u_{j,\omega}:&=\lambda^u_{j,i} \qquad
%  	\text{if} \qquad f_{\omega_0}\in\cU_i, 
%  	\qquad \text{and}   \notag\\
  	\lambda^u_{j,\omega}(n):&=
  	\lambda^u_{j,\sigma^{n-1}(\omega)}+\lambda^u_{j,\sigma^{n-2}(\omega)}
  	+\cdots+
  	\lambda^u_{j,\sigma(\omega)}+\lambda^u_{j,\omega}. %, \qquad \forall n\geq1.
  \end{align}
  By iterating \eqref{eq:unfhyp}, we have the following:
  \begin{Lemma}\label{lem:A-dominated}
  	  Relative to an adapted Riemannian metric on $\TT^d$,  %to every fiber $\TT^d_\omega$, 
	  there exists $\epsilon>0$, such that for every $(\omega,x)\in\Omega\times\TT^d$, $1\leq j<k$ and $n\geq1$,
  	  $$
  	  \left\| \bar{A}^n(\omega,x)|_{L^u_{(1,j),\omega}}\right\|<
  	  \exp\left[\lambda^u_{j,\omega}(n)+\frac{n\cdot\epsilon}{2}\right]<
  	  \exp\left[\lambda^u_{j+1,\omega}(n)-\frac{n\cdot\epsilon}{2}\right]<
  	  m\left(\bar{A}^n(\omega,x)|_{L^u_{(j+1,k),\omega}}\right).
  	  $$
  \end{Lemma}

%  From Proposition \ref{prop:structure-stable}, there exists a continuous 
%  $$
%  H:\Omega\times\TT^d\to\Omega_0\times\TT^d
%  \qquad \text{such~that} \qquad
%  H\circ F=A\circ H,
%  $$
%where the restriction of $H$ to each fiber $\TT^d_{\omega}$, $\omega\in \Omega$ is a bi-H\"older homeomorphism.
%  where $\bar{A}:\Omega\times\TT^d\to\Omega\times\TT^d$ is the extension skew product defined by Definition \ref{def:extension}.  

%%  Following Proposition \ref{prop:structure-stable}, let $\{\cU_i\}$ be a (pairwise disjoint) family of \red{ $C^1$} neighborhoods of $\{A_i\}$ respectively, such that for $\cU=\cup\cU_i\subseteq\diff^2(\TT^d)$.  With $\Omega=\cU^\ZZ$, the corresponding skew-product $F:\Omega\times\TT^d\to\Omega\times\TT^d$ is fiberwise Anosov.  
Having introduced the skew product $\bar{A}$ over the shift $\Omega$,   following Proposition \ref{prop:structure-stable}, we obtain  a homeomorphism  $  H\colon \Omega\times \TT^d\to  \Omega\times \TT^d $ with 
	$$  H\circ F= \bar A\circ   H.$$
  By shrinking the neighborhood $\cU_i$ if necessary, we have the following proposition.

  \begin{Proposition}\label{prop:leaf-conjugacy}
  	 There exists a family of  (pairwise disjoint) neighborhoods $\cU_i\subseteq\diff^2(\TT^d)$ of $A_i$, such that if $\cU=\cup\cU_i$ and $\Omega=\cU^\ZZ$, 
  	 the corresponding skew product 
  	 $$
  	 F:\Omega\times\TT^d\to\Omega\times\TT^d, \qquad
  	 F(\omega,x)=\left(\sigma(\omega),f_\omega(x)\right)
  	 $$ 
  	 satisfies the following properties:
  	 \begin{enumerate}
  	 	\item $F$ admits a fiberwise finest dominated splitting with the same dimensions as $\bar{A}$:
  	 	$$
  	 	T\TT^d_\omega=
%\note{reversed s order}
  	 	E^s_{l,\omega}\oplus\cdots E^s_{1,\omega}\oplus E^u_{1,\omega}\oplus\cdots E^u_{k,\omega}, 
  	 	\qquad \text{with} \qquad
  	 	{\rm dim}E^{s/u}_{j,\omega}={\rm dim}L^{s/u}_j,\quad \forall j.
  	 	$$
  	 	
  	 	\item For every $1\leq j<k$, $(\omega,x)\in\Omega\times\TT^d$ and $n\geq1$, we have
  	 	$$
  	 	\left\|Df_\omega^n(x)|_{E^u_{(1,j),\omega}}\right\|<
  	 	\exp\left[\lambda^u_{j,\omega}(n)+ {n\cdot\epsilon}\right]<
  	 	\exp\left[\lambda^u_{j+1,\omega}(n)-{n\cdot\epsilon}\right]<
  	 	m\left(Df_\omega^n(x)|_{E^u_{(j+1,k),\omega}}\right).
  	 	$$ 
  	 	
  	 	\item For every $1\leq i\leq j\leq k$, there exists a family of fiberwise $C^0$ foliations with $C^1$ leaves $\{\cF^u_{(i,j),\omega}\}_{\omega\in\Omega}$, such that for every $\omega\in\Omega$,
  	 	$$
  	 	F\left(\cF^u_{(i,j),\omega}\right)=\cF^u_{(i,j),\sigma(\omega)}
  	 	\qquad \text{and} \qquad
  	 	T\cF^u_{(i,j),\omega}=E^u_{(i,j),\omega}
  	 	=E^u_{i,\omega}\oplus\cdots\oplus E^u_{j,\omega}.
  	 	$$
  	 	Moreover, for $1\leq i\le i'\le j'\leq j\leq k$, the leaves of $\cF^u_{(i',j'),\omega}$ subfoliate the leaves of $\cF^u_{(i,j),\omega}$.  
  	 	\item For every $\omega\in\Omega$, the topological conjugacy $\bar H_{\omega}:\TT^d_\omega\to\TT^d_\omega$ satisfies
  	 	$$
  	 	  H_\omega\left(\cF^u_{(1,j),\omega}\right)=\cL^u_{(1,j),\omega}=\cL^u_{(1,j)},
  	 	\qquad \forall 1\leq j\leq k.
  	 	$$
  	 \end{enumerate}
Similar properties  hold for corresponding dominated splittings of stable bundle and stable foliation.  %(with appropriate change of indexing).
  \end{Proposition}
  
  \begin{proof}
  \noindent{\underline {Proof of first item.}}	 The skew product $F$ is a fiberwise perturbation of $\bar{A}$ on $\Omega\times\TT^d$. From the persistence of dominated splitting, $F$ admits the dominated splitting with the same dimensions to $\bar{A}$.   	 

 	 \vskip 3mm
  \noindent{\underline {Proof of second item.}}	  	 Shrinking the neighborhood $\cU=\cup\cU_i$ if necessary. Since $F$ is a fiberwise perturbation of $\bar{A}$ and the {$C^2$-norm} is close for fiber dynamics of $F$ and $\bar{A}$, the second item comes directly from Lemma \ref{lem:A-dominated}.
  	 
 	 \vskip 3mm
  \noindent{\underline {Proof of third item.}}	  	 
  	 For any $1\leq i\leq k$, the classical stable manifold theorem shows the existence of fiberwise strong unstable foliation $\{\cF^u_{(i,k),\omega}\}_{\omega\in\Omega}$ which is $F$-invariant and tangent to $E^u_{(i,k),\omega}$ everywhere.
  	 
  	 On the other hand, for every $i\leq j\leq k$, the foliation $\cL^u_{(1,j)}$ is fiberwise normally hyperbolic with respect to $\bar{A}$. Moreover, since the linear foliation $\cL^u_{(1,j)}$ is smooth, it is also fiberwise plaque expansive \cite[Theorem 7.2]{HPS} and \cite[Theorem 5.12]{Pes}. 
  	 
  	 Since $F$ is a fiberwise perturbation of $\bar{A}$ and the {$C^2$-norm} is close for fiber dynamics of $F$ and $\bar{A}$, the fiberwise structural stability of normally hyperbolic foliations \cite[Theorem 7.1]{HPS} shows the following claim.
  	 
  	 \begin{Claim-numbered}\label{clm:normal-hyperb}
  	 	There exists a fiberwise foliation $\{\cF^u_{(1,j),\omega}\}_{\omega\in\Omega}$ such that
  	 	$$
  	 	F\left(\cF^u_{(1,j),\omega}\right)=\cF^u_{(1,j),\sigma(\omega)}
  	 	\qquad \text{and} \qquad
  	 	T\cF^u_{(1,j),\omega}=E^u_{(1,j),\omega}.
  	 	$$  
  	 	Moreover, $\{\cF^u_{(1,j),\omega}\}_{\omega\in\Omega}$ is leaf conjugated to $\{\cL^u_{(1,j),\omega}\}_{\omega\in\Omega}$, i.e. there exists a homeomorphism 
  	 	$$
  	 	G:\Omega\times\TT^d\to\Omega\times\TT^d,
  	 	\qquad
  	 	G(\omega,x)=\left(\omega,G_\omega(x)\right)
  	 	\qquad \text{such~that} \qquad
  	 	G_\omega\left(\cF^u_\omega\right)=\cL^u_\omega,
  	 	$$
  	 	where $G_\omega$ is a homeomoprhism on $\TT^d_\omega$ and close to identity map and
  	 	\begin{align*}
  	 		G_{\sigma(\omega)}\circ F\left(\cF^u_{(1,j),\omega}(x)\right)&=
  	 		\bar{A}_\omega\circ G_\omega\left(\cF^u_{(1,j),\omega}(x) \right)  \\
  	 		&=
  	 		\bar{A}_\omega\left( \cL^u_{(1,j),\omega}(G_\omega(x)) \right)=
  	 		\cL^u_{(1,j),\sigma(\omega)}(\bar{A}_{\omega}\circ G_\omega(x)).
  	 	\end{align*}
  	 \end{Claim-numbered}

  	 Now we take
  	 $$
  	 \cF^u_{(i,j),\omega}= \cF^u_{(1,j),\omega}\cap \cF^u_{(i,k),\omega}
  	 \qquad \forall \omega\in\Omega.
  	 $$
  	 It is clear that $\cF^u_{(i,j),\omega}$ is $F$-invariant and tangent to $E^u_{(i,j),\omega}= E^u_{(1,j),\omega}\cap E^u_{(i,k),\omega}$.
  	 This proves the third item.
  	 
  	 \vskip 3mm

  	 %From Claim \ref{clm:normal-hyperb}, $\cF^u_{(1,j),\omega}$ satisfies $G_\omega\left(\cF^u_\omega\right)=\cL^u_\omega$ where $G_\omega$ is uniformly close to identity map in $\TT^d_\omega$, and $E^u_{(1,j),\omega}$ is uniformly close to $L^u_{(1,j),\omega}$ (independent to $\omega\in\Omega$), we have the following claim.
  	 
  	 %\begin{Claim-numbered}{\cite[Lemma 2.2]{GoS}}\label{clm:quasi-iso}
  	 %	There exist constant $a,b>0$ such that for every $\omega\in\Omega$, denote $\RR^d_\omega$ the universal cover of $\TT^d_\omega$ and $\tilde\cF^u_{(1,j),\omega}$ the lifting foliation of $\cF^u_{(1,j),\omega}$ in $\RR^d_\omega$, then for every $x\in\RR^d_\omega$ and $y\in\tilde\cF^u_{(1,j),\omega}(x)$, we have
  	 %	$$
  	 %	d_{\tilde\cF^u_{(1,j),\omega}}(x,y)\leq a\cdot d(x,y)+b.
  	 %	$$
  	 %	Here $d(\cdot,\cdot)$ denotes the distance on $\RR^d_\omega$ induced by the lifting adapted Riemannian metric of $\TT^d_\omega$, and 
  	 %	 $d_{\tilde\cF^u_{(1,j),\omega}}(\cdot,\cdot)$
  	 %	denotes the distance along the leaf of $\tilde\cL^u_{(1,j),\omega}$ induced by the restriction of the Riemannian metric.  	 	
  	 %\end{Claim-numbered}
  \noindent{\underline {Proof of foruth item.}}	  	 
%  	  Now we prove the last item.
  	 Denote $\tilde\cL^u_{(1,j),\omega}$ the lifting foliation of $\cL^u_{(1,j),\omega}$ on $\RR^d_\omega$, and for every $n\geq1$, denote
  	 \begin{itemize}
  	 	\item $\tilde{A}_\omega^n:\RR^d_\omega\to\RR^d_{\sigma^n(\omega)}$ the lifting linear map of $\bar{A}_\omega^n:\TT^d_\omega\to\TT^d_{\sigma^n(\omega)}$;
  	 	\item $\tilde{f}_\omega^n:\RR^d_\omega\to\RR^d_{\sigma^n(\omega)}$ the lifting map of $f_\omega^n:\TT^d_\omega\to\TT^d_{\sigma^n(\omega)}$ close to $\tilde{A}_\omega^n$;
  	 	\item $\tilde{H}_\omega:\RR^d_\omega\to\RR^d_\omega$ the lift of fiberwise conjugacy of $H_\omega:\TT^d_\omega\to\TT^d_\omega$ such that
  	 	    \begin{itemize}
  	 	    	\item $\tilde{H}_{\sigma^n(\omega)}\circ\tilde{f}_\omega^n=
  	 	    	\tilde{A}_\omega^n\circ\tilde{H}_\omega$;
  	 	    	\item there exists some constant $C>0$ such that
  	 	    	$$
  	 	    	\|\tilde{H}_\omega-{\rm Id}_{\RR^d_\omega}\|<C,
  	 	    	\qquad \forall \omega\in\Omega.
  	 	    	$$
  	 	    \end{itemize}
  	 \end{itemize} 
  	 
  	 Denote $d(\cdot,\cdot)$ the distance on $\RR^d_\omega$ induced by the lifting adapted Riemannian metric of $\TT^d_\omega$.
  	 By applying the estimation of Lemma \ref{lem:A-dominated} and the second item, we have the following claim.
  	 
  	 \begin{Claim-numbered}\label{clm:growth}
  	 	For every $\omega\in\Omega$ and every pair of points $x,y\in\RR^d_\omega$, we have
  	 	\begin{enumerate}
  	 		\item $y\in\tilde\cF^u_{(1,j),\omega}(x)$ if and only if~
  	 		$\lim_{n\to+\infty}\exp[-\lambda^u_{j,\omega}(n)]\cdot
  	 		d(\tilde{f}^n_\omega(x),\tilde{f}^n_\omega(y))=0$;
  	 		
  	 		\item $y\notin\tilde\cF^u_{(1,j),\omega}(x)$ if and only if~
  	 		$\lim_{n\to+\infty}\exp[-\lambda^u_{j,\omega}(n)]\cdot
  	 		d(\tilde{f}^n_\omega(x),\tilde{f}^n_\omega(y))=\infty$;
  	 		
  	 			\item $y\in\tilde\cL^u_{(1,j),\omega}(x)$ if and only if~
  	 		$\lim_{n\to+\infty}\exp[-\lambda^u_{j,\omega}(n)]\cdot
  	 		d(\tilde{A}^n_\omega(x),\tilde{A}^n_\omega(y))=0$;
  	 		
  	 		\item $y\notin\tilde\cL^u_{(1,j),\omega}(x)$ if and only if~
  	 		$\lim_{n\to+\infty}\exp[-\lambda^u_{j,\omega}(n)]\cdot
  	 		d(\tilde{A}^n_\omega(x),\tilde{A}^n_\omega(y))=\infty$.
  	 		
  	 	\end{enumerate}
  	 \end{Claim-numbered}
  	 
  	 \begin{proof}[Proof of Claim \ref{clm:growth}]
  	 	From the second item of the proposition, on the lifting bundle $\tilde{E}^u_{(1,j),\omega}$ of $E^u_{(1,j),\omega}$ on $\RR^d_\omega$, we have
  	 	$$
  	 	\|D\tilde{f}_\omega^n|_{\tilde{E}^u_{(1,j),\omega}}\|<\exp[\lambda^u_{j,\omega}(n)-n\cdot\epsilon/2].
  	 	$$
  	 	If $y\in\tilde\cF^u_{(1,j),\omega}(x)$, then
  	 	\begin{align*}
  	 		\lim_{n\to+\infty}
  	 		\frac{d(\tilde{f}^n_\omega(x),\tilde{f}^n_\omega(y))}{\exp[\lambda^u_{j,\omega}(n)]}
  	 		&\leq
  	 		\lim_{n\to+\infty}
  	 		\frac{d_{\tilde{\cF}^u_{(1,j),\sigma^n(\omega)}}(\tilde{f}^n_\omega(x),\tilde{f}^n_\omega(y)) }{\exp[\lambda^u_{j,\omega}(n)]} \\
  	 		&\leq
  	 		\lim_{n\to+\infty}
  	 		\frac{\exp[\lambda^u_{j,\omega}(n)-n\cdot\epsilon/2]\cdot d_{\tilde{\cF}^u_{(1,j),\omega}}(x,y) }{\exp[\lambda^u_{j,\omega}(n)]} \\
  	 		&=0
  	 	\end{align*}
  	 	
  	 	If $y\notin\tilde\cF^u_{(1,j),\omega}(x)$, 
  	 	then $\tilde\cF^u_{(1,j),\omega}(x)\neq\tilde\cF^u_{(1,j),\omega}(y)$. Denote
  	 	$$
  	 	d_{\min}\left(\tilde\cF^u_{(1,j),\omega}(x),\tilde\cF^u_{(1,j),\omega}(y)\right)
  	 	:=\inf\left\{ d(z,w):~z\in\tilde\cF^u_{(1,j),\omega}(x),~w\in\tilde\cF^u_{(1,j),\omega}(y)
  	 	 \right\},
  	 	$$
  	 	and $\tilde{G}_\omega:\RR^d_\omega\to\RR^d_\omega$ the leaf conjugacy of $G_\omega:\TT^d_\omega\to\TT^d_\omega$ in Claim \ref{clm:normal-hyperb} which is close to identity map on $\RR^d_\omega$ and satisfies
  	 	\begin{align*}
  	 		\tilde{G}_{\sigma(\omega)}\circ \tilde{f}_\omega\left(\cF^u_{(1,j),\omega}(x)\right)
  	 		&=
  	 		\tilde{A}_\omega\circ \tilde{G}_\omega\left(\cF^u_{(1,j),\omega}(x) \right)  \\
  	 		&=
  	 		\tilde{A}_\omega\left( \cL^u_{(1,j),\omega}(\tilde{G}_\omega(x)) \right)=
  	 		\cL^u_{(1,j),\sigma(\omega)}(\tilde{A}_{\omega}\circ \tilde{G}_\omega(x)).
  	 	\end{align*}
  	 	
  	 	Since $\tilde{G}_\omega$ is a homeomorphism, 
  	 	if $\tilde\cF^u_{(1,j),\omega}(x)\neq\tilde\cF^u_{(1,j),\omega}(y)$,
  	 	then 
  	 	$$
  	 	\tilde\cL^u_{(1,j),\omega}(\tilde{G}_\omega(x))=
  	 	\tilde{G}_\omega\left(\tilde\cF^u_{(1,j),\omega}(x)\right)
  	 	\neq
  	 	\tilde{G}_\omega\left(\tilde\cF^u_{(1,j),\omega}(y)\right)=
  	 	\tilde\cL^u_{(1,j),\omega}(\tilde{G}_\omega(y)).
  	 	$$
  	 	From the leaf conjugacy and $\tilde{G}_\omega$ is uniformly close to identity map on $\RR^d_\omega$, we have
  	 	\begin{align*}
  	 		\lim_{n\to+\infty}
  	 		\frac{d(\tilde{f}^n_\omega(x),\tilde{f}^n_\omega(y))}{\exp[\lambda^u_{j,\omega}(n)]}
  	 		&\geq 
  	 		\lim_{n\to+\infty}
  	 		\frac{d_{\min}\left(\tilde{f}^n_\omega\left(\tilde\cF^u_{(1,j),\omega}(x)\right), \tilde{f}^n_\omega\left(\tilde\cF^u_{(1,j),\omega}(y)\right)\right)}{
  	 			 \exp[\lambda^u_{j,\omega}(n)]} \\
  	 		&=
  	 		\lim_{n\to+\infty}
  	 		\frac{d_{\min}\left(\tilde{G}_{\sigma^n(\omega)}\circ \tilde{f}^n_\omega\left(\tilde\cF^u_{(1,j),\omega}(x)\right), 
  	 			\tilde{G}_{\sigma^n(\omega)}\circ 
  	 			\tilde{f}^n_\omega\left(\tilde\cF^u_{(1,j),\omega}(y)\right)\right)}{
  	 			 \exp[\lambda^u_{j,\omega}(n)]}  \\
  	 		&=
  	 		\lim_{n\to+\infty}
  	 		\frac{d_{\min}\left(\tilde{A}_{\omega}^n\circ
  	 			\tilde{G}_{\omega}\left(\tilde\cF^u_{(1,j),\omega}(x)\right), \tilde{A}_{\omega}^n\circ
  	 			\tilde{G}_{\omega}\left(\tilde\cF^u_{(1,j),\omega}(y)\right)\right)}{
  	 			 \exp[\lambda^u_{j,\omega}(n)]} \\
  	 		&=
  	 		\lim_{n\to+\infty}
  	 		\frac{d_{\min}\left(\tilde{A}_{\omega}^n\left(\tilde\cL^u_{(1,j),\omega}(\tilde{G}_\omega(x))\right), \tilde{A}_{\omega}^n\left(\tilde\cL^u_{(1,j),\omega}(\tilde{G}_\omega(x))\right)\right)}{
  	 			 \exp[\lambda^u_{j,\omega}(n)]}\\
  	 	    &\geq
  	 	    \lim_{n\to+\infty}
  	 	    \frac{ \exp\left[\lambda^u_{j,\omega}(n)+n\cdot\epsilon\right]\cdot
  	 	    	d_{\min}\left(\tilde\cL^u_{(1,j),\omega}(\tilde{G}_\omega(x)), \tilde\cL^u_{(1,j),\omega}(\tilde{G}_\omega(x))\right)}{
  	 	    	\exp[\lambda^u_{j,\omega}(n)]}  \\
  	 	    &=\infty.
  	 	\end{align*}
  	 	The last inequality comes from the Lemma \ref{lem:A-dominated} that growth rate of distance between two linear center leaves of $\tilde\cL^u_{(1,j),\omega}$ is larger than $\exp\left[\lambda^u_{j,\omega}(n)+n\cdot\epsilon\right]$.
  	 	This proves the first and second items of Claim \ref{clm:growth}.
  	 	
  	 	Similarly, since both foliations $\tilde\cL^u_{(1,j),\omega}$ and $\tilde\cL^u_{(j+1,k),\omega}$ are linear foliations on $\RR^d_\omega$ and the lifting adapted Riemannian metric are equivalent to the Euclidean metric on $\RR^d_\omega$, Lemma \ref{lem:A-dominated} implies the last two items of Claim \ref{clm:growth}.
  	 \end{proof}
  	 
  	 Since $\|\tilde{H}_\omega-{\rm Id}_{\RR^d_\omega}\|<C$ for every $\omega\in\Omega$,  Claim \ref{clm:growth} implies 
  	 \begin{align*}
  	 	y\in\tilde\cF^u_{(1,j),\omega}(x)  &\Longleftrightarrow
  	 	    \lim_{n\to+\infty}
  	 	    \frac{d(\tilde{f}^n_\omega(x),
  	 	    	\tilde{f}^n_\omega(y))}{\exp[\lambda^u_{j,\omega}(n)]}=0
  	 	    \\
  	 	    &\Longleftrightarrow
  	 	    \lim_{n\to+\infty}
  	 	    \frac{d(\tilde{H}_{\sigma^n(\omega)}\circ\tilde{f}^n_\omega(x),
  	 	    	\tilde{H}_{\sigma^n(\omega)}\circ\tilde{f}^n_\omega(y))}{ \exp[\lambda^u_{j,\omega}(n)]}=0
  	 	    \\
  	 	    &\Longleftrightarrow
  	 	    \lim_{n\to+\infty}
  	 	    \frac{d(\tilde{A}^n_{\omega}\circ\tilde{H}_{\omega}(x),
  	 	    	\tilde{A}^n_{\omega}\circ\tilde{H}_{\omega}(y))}{ \exp[\lambda^u_{j,\omega}(n)]}=0
  	 	    \\
  	 	    &\Longleftrightarrow 
  	 	    \tilde{H}_\omega(y)\in\tilde\cF^u_{(1,j),\omega}(\tilde{H}_\omega(x)).
  	 \end{align*}
  	 This proves that $\tilde{H}_\omega(\tilde\cF^u_{(1,j),\omega})=\tilde\cL^u_{(1,j),\omega}$ and consequently $H_\omega(\cF^u_{(1,j),\omega})=\cL^u_{(1,j),\omega}$. 
  	 This finishes the proof of Proposition \ref{prop:leaf-conjugacy}.
  \end{proof}

\subsection{Fiberwise $u$-Gibbs states of expanding foliations}

Let $X$ be a metric space and $T:X\to X$ be a homeomorphism of $X$. Let $M$ be a compact manifold and let $f:X\to\diff^2(M)$ be a continuous map with bounded $C^2$-norms, i.e. there exists $K>0$ such that for every $\omega\in X$,
$$
\omega\mapsto f_\omega\in\diff^2(M) \qquad\text{satisfies} \qquad
\|f_{\omega}\|_{C^2}+\|f_{\omega}^{-1}\|_{C^2}<K.
$$
The map $f$ defines a skew product system over $T$:
$$
F:X\times M\to X\times M,
\qquad
F(\omega,x)=\left(T(\omega),f_{\omega}(x)\right),
\quad \forall (\omega,x)\in X\times M.
$$
For $n\in\NN$, we have the notion 
$$
f_\omega^n:=f_{T^{n-1}\omega}\circ\cdots\circ f_{T\omega}\circ f_\omega,
\qquad \text{and} \qquad
f_\omega^{-n}:=(f_{T^{-n}\omega})^{-1}\circ\cdots\circ(f_{T^{-1}\omega})^{-1}.
$$

\begin{Definition}\label{def:expanding}
	Let $F:X\times M\to X\times M$ be the above skew product. Let  $\cF=\{\cF_{\omega}\}_{\omega\in X}$  be a parameterized family of $C^0$ foliations of $M$ with $C^1$-leaves.  
	
	We say  $\cF=\{\cF_{\omega}\}_{\omega\in X}$ is  \emph{fiberwise expanding}, if there exist $C>1$, $\gamma>0$, and $L>1$, $0<\theta<1$ such that 
	\begin{enumerate}
		\item For all $(\omega,x)\in\Omega\times M$, $x\mapsto T\cF_{\omega}(x)$ is $(L,\theta)$-H\"older continuous on $M_\omega= \{\omega\}\times M$.
		\item The foliation $\cF$ is $F$-invariant: 
		$F\left(\cF_{\omega}(x)\right)=f_{\omega}\left(\cF_{\omega}(x)\right)=\cF_{T(\omega)}(f_{\omega}(x))$.
		\item For every nonzero vector $v\in T\cF_{\omega}$,
		$\|Df_\omega^{-n}v\|< C\exp(-n\gamma)\|v\|$ for every $n\geq0$.
		
	\end{enumerate}
\end{Definition}
Since the bundle $T\cF$ is $(L,\theta)$-H\"older continuous, the leaves  of $\cF$ are uniformly $C^{1+\theta}$ continuous. 

Given an  $F$-invariant probability measure $\mu$, we can may consider a measurable partition $\cP$ subordinated to $\cF$, i.e. for $\mu$-a.e. $(\omega,x)\in\Omega\times M$, there exists $r>0$ such that
$$
\cF_{\omega,r}(x)\subset\cP(\omega,x)\subset\cF_\omega(x).
$$
We can generalize the notion of fiberwise SRB measures to $u$-Gibbs states associated to a fiberwise expanding foliation of $F$.

\begin{Definition}\label{def:gibbs}
	Let $\cF$ be a fiberwise expanding foliation of $F$. An $F$-invariant probability measure $\mu$ is a $u$-Gibbs state along $\cF$ if for any $\cF$-subordinate measurable partition $\cP$ with a family of corresponding conditional measures $\{\tilde{\mu}^{\cP}_{(\omega,x)}\}_{(\omega,x)\in\Omega\times M}$, the conditional measure $\tilde{\mu}^{\cP}_{(\omega,x)}$ is absolutely continuous with respect to the Lebesgue measure on $\cF_\omega(x)$ for $\mu$-a.e. $(\omega,x)\in X\times M$.
\end{Definition}

%%\red{should we state their existence?}

The following proposition follows from Theorem G of \cite{SY}.

\begin{Proposition}\label{prop:u-gibbs}
	
	Let $F,G:X\times M\to X\times M$ be two skew product system over the homeomorphism $T:X\to X$ defined by continuous maps $f,g:X\to\diff^2(M)$ with bounded $C^2$-norms respectively:
	$$
	F(\omega,x)=\left(T(\omega),f_{\omega}(x)\right),
	\qquad
	G(\omega,x)=\left(T(\omega),g_{\omega}(x)\right),
	\qquad \forall (\omega,x)\in X\times M.
	$$
	
	Let $\cF$ and $\cG$ be two fiberwise expanding foliations of $F$ and $G$ respectively, and let $\mu_F$ be an $F$-invariant $u$-Gibbs measure along $\cF$.  Assume there exists a homeomorphism %\red{how regular?}
	$$
	H:X\times M\to X\times M,
	\qquad
	H(\omega,x)=\left(\omega,H_\omega(x)\right)
	\quad
	\forall(\omega,x)\in X\times M.
	$$
	such that  $H_\omega:M_\omega\to M_\omega$ is a homeomorphism for every $\omega\in X$.   If $H$ satisfies 
	$$
	H\circ F=G\circ H
	\qquad \text{and} \qquad
	H(\cF)=\cG,
	$$  
	then the followings are equivalent for $\mu_G=H_*(\mu_F)$:
	\begin{enumerate}
		\item $\mu_G$ is a $u$-Gibbs measure of $G$ along $\cG$.
		\item $H$ is absolutely continuous along the leaves of $\cF$ on the support of $\mu_F$, with the Jacobian continuous on $\supp(\mu_F)$ and bounded away form zero and infinity.
		\item There exists $K>1$ such that for every $(\omega,x)\in\Omega\times M$ and $n\geq 0$,
		$$
		\frac{1}{K}<
		\frac{\det(DF^n(T\cF_\omega(x)))}{\det(DG^n(T\cG_\omega(H_\omega(x))))}
		<K.
		$$
		\item The sums of fiberwise Lyapunov exponents along expending foliations are the same:
		$$
	\sum\lambda(F,\mu_F,\cF)=\int\log(\det(DF|_{T\cF}))d\mu=
		\int\log(\det(DG|_{T\cG}))d\mu_G=\sum\lambda(G,\mu_G,\cL).
		$$
		\item If we further assume $\dim\cF=1$, then $H$ is uniformly $C^{1+}$-smooth along the leaves of $\cF$.
	\end{enumerate}
	
\end{Proposition}
In item 4 of Proposition \ref{prop:u-gibbs}, 
$$\sum\lambda(F,\mu_F,\cF)$$ denotes the sum of all Lyapunov exponents, counted with multiplicity, for the restriction of the cocycle $DF$ to the $DF$-invariant subbundle $(\omega,x)\mapsto T_x\cF_\omega(x)$.

We will frequently apply Proposition \ref{prop:u-gibbs} to the following situation: 
Let $\cA=\{f_i\}_{i=1}^m\subseteq\diff^2(M)$ be a cone hyperbolic family of (Anosov) diffeomorphisms.  Then the skew product $A:\cA^\ZZ\times M\to\cA^Z\times M$ is fiberwise Anosov. Denote by $\cU=\cup\cU_i$ be the neighborhood of $\cA$ and $\Omega=\cU^\ZZ$ as Lemma \ref{lem:robust-Anosov} such that the skew product $F:\Omega\times M\to\Omega\times M$ is also fiberwise Anosov. Let $\bar{A}:\Omega\times M\to\Omega\times M$ be the extension of $A$ in Definition \ref{def:extension} and let $H:\Omega\times M\to\Omega\times M$ be the topological conjugacy $H\circ F=\bar{A}\circ H$.  We will then apply Proposition \ref{prop:u-gibbs} to the fiberwise expanding foliations $\cF^u_{j,\omega}$ and $\cL^u_{j,\omega}$.  

%\red{
%\note{Just state the corollary?  But need fibered version that might not be stationary.  So just keep the prop.}
%\begin{Corollary}\label{coro:u-gibbs}
%	Let $\cF$ and $\cL$ be  expanding foliations for $F$ and $\bar{A}$, respectively, and let $H(\cL)=\cF$. Let $\nu$ be a probability supported on $\cU$ and $\hat\mu$ be a $\nu$-stationary probability measure. Assume the corresponding $F$-invariant measure $\mu$ is a $u$-Gibbs state along $\cF$. Then for $\mu_0=H_*(\mu)$, the following are equivalent:
%	\begin{enumerate}
%		\item $\mu_0$ is a $u$-Gibbs measure of $G$ along $\cG$.
%		\item $H$ is absolutely continuous along the leaves of $\cF$ on the support of $\mu$, with the Jacobian continuous on $\supp(\mu)$ and bounded away form zero and infinity.
%		\item There exists $K>1$ such that for every $(\omega,x)\in\Omega\times M$ and $n\geq 0$,
%		$$
%		\frac{1}{K}<
%		\frac{\det(DF^n(T\cF_\omega(x)))}{\det(D\bar{A}^n(T\cL_\omega(H_\omega(x))))}
%		<K.
%		$$
%		\item The sums of fiberwise Lyapunov exponents along expending foliations are the same:
%		$$
%		\red{\sum\lambda(F,\mu,\cF)}=\int\log|\det(DF|_{T\cF})|d\mu=
%		\int\log|\det(D\bar{A}|_{T\cL})|d\mu_0=\sum\lambda(\bar{A},\mu_0,\cL).
%		$$
%		\item If we further assume $\dim\cF=1$, then $H$ is uniformly $C^{1+}$-smooth along the leaves of $\cF$.
%	\end{enumerate}
%\end{Corollary}
%}

\section{Positive spectrum rigidity and smooth conjugacy}
%that if the corresponding SRB-measures of
%
In this section, we show for  fiberwise Anosov perturbations of linear maps, coincidence of Lyapunov exponents relative to the fiberwise SRB-measure with the Lyapunov exponents of the linear system 
%diffeomorphisms admit positive Lyapunov spectrum rigidity, then the
implies  topological conjugacy in Proposition \ref{prop:structure-stable} is $C^{1+}$-smooth along fiberwise expanding foliations. 
We note that we do not assume our fiberwise SRB measures  are a measure as in Proposition \ref{prop:measure} obtained from a stationary measure. 
Note that the converse holds automatically. We will first prove this on $\TT^2$, for both perturbations of a single Anosov diffeomorphism and perturbations of a cone hyperbolic  non-commuting family of automorphisms.

Recall we consider  $\cA=\{A_i\}_{i=1}^m\subseteq{GL}(d,\ZZ)$ be a cone hyperbolic family of Anosov automorphisms.  Following Lemma \ref{lem:robust-Anosov}, there exist a (pairwise disjoint) family of  neighborhoods $\cU_i\subseteq\diff^2(\TT^d)$ of $A_i$ such that for $\cU=\cup\cU_i$ and $\Omega=\cU^\ZZ$, the corresponding skew product
$$
F:\Omega\times\TT^d\to\Omega\times\TT^d, \qquad
 F(\omega,x)=\left(\sigma(\omega),f_\omega(x)\right)
$$
is fiberwise Anosov. 

%%%%Moreover, let $\bar{A}:\Omega\times\TT^d\to\Omega\times\TT^d$ be the extension of skew product system $A_0:\cA^\ZZ\times\TT^d\to\cA^\ZZ\times\TT^d$ defined by Definition \ref{def:extension}:
%%%%$$
%%%%\red{\bar{A}(\omega,x)=\left(\sigma(\omega),A_{\omega'}(x)\right),
%%%%\qquad \text{where} \qquad
%%%%A_{\omega'}=A_i\in\cA\quad \text{if}\quad f_{\omega}=f_{\omega_0}\in\cU_i.}
%%%%$$
%%%%Here the map $\omega\mapsto\omega'$ is defined as $A_{\omega'_j}=A_{i_j}\in\cA$ if $f_{\omega_j}\in\cU_{i_j}$ for $i_j\in\{1,\cdots,m\}$.

Proposition \ref{prop:structure-stable} shows that there exists a topological conjugacy
$H:\Omega\times\TT^d\to\Omega\times\TT^d$ with
$$
H(\omega,x)=\left(\omega, H_{\omega}(x)\right),
\qquad \text{such~that} \qquad
H\circ F=\bar{A}\circ H.
$$
If we denote $\cL^u$ and $\cF^u$ the fiberwise expanding foliations of $\bar{A}$ and $F$ on $\Omega\times\TT^d$ respectively, then $H(\cF^u)=\cL^u$. 

We note that since fiberwise dynamics of $\bar{A}:\Omega\times\TT^d\to\Omega\times\TT^d$ are compositions of affine maps, $\{A_i\}_{i=1}^m\subseteq{GL}(d,\ZZ)$ for every $(\omega,x)\in\Omega\times\TT^d$, the fiberwise derivative is $A_{\omega}$ is independent of the choice of $x\in\TT^d$. %We have the following observation.

%\blue{
This in particular justifies the following.  
\begin{Lemma}\label{lem:nu-exponent}
 Let $p=\{p_i\}_{1\le i\le m}$ be a probability measure on $\{A_i\}_{i=1}^m\subseteq{GL}(d,\ZZ)$ with $p(A_i)=p_i>0$ for every $1\le i\le m$.  
	Let $\nu$ be a probability on $\cU=\cup\cU_i$ with $\nu(\cU_i)=p_i$, $i=1,\cdots,m$.  % and $\sum p_i=1$. 
Then for any  probability measure $\bar\mu$ invariant under $\bar{A}:\Omega\times\TT^d\to\Omega\times\TT^d$ and which satisfies $\pi_*(\bar\mu)=\nu^\ZZ$, the fiberwise Lyapunov exponents of $\bar\mu$ coincide with the Lyapunov exponents of the i.i.d. (rel.\ $p$) product of matrices $\{A_i\}_{i=1}^m$ given in Proposition \ref{prop:Lyapaffine}.\end{Lemma}
We note that the measure $\bar \mu$ need not be a measure obtained from a stationary measure as in Proposition \ref{prop:measure}. 
%}

%\red{this needs a bit of justification since its not necessarily stationary anymore}
%only depends on $\{p_i\}_{i=1}^m$, which are independent on $\nu|_{\cU_i}$ and $\bar\mu$ along $\TT^d$-fibers.

\subsection{Positive spectrum rigidity on $\TT^2$}\label{sect:u-smooth-T2}

In this subsection, we consider the case that
$\cA=\{A_i\}_{i=1}^m\subseteq{GL}(2,\ZZ)$ is a cone hyperbolic family of (Anosov) automorphisms. %% share common stable and unstable cone fields. 
There are two possibilities:
\begin{Lemma}\label{claim:ZdenseorZ} Let $\cA=\{A_i\}_{i=1}^m\subseteq{GL}(2,\ZZ)$ be a cone hyperbolic family of automorphisms.  Then either
\begin{itemize}
	\item there exist $i\neq j$ such that $A_i$ and $A_j$ do not commute, or 
	\item  there exists a hyperbolic automorphism $A\in{GL}(2,\ZZ)$ such that $A_i=A^{n_i}$ with $n_i\in\NN$ for every $i=1,\cdots,m$. 
\end{itemize}
\end{Lemma}
\begin{proof}
By definition, every $A\in \cA$ is hyperbolic.  Fix $A_1\in \cA$ and let $E^u,E^s$ denote be the distinct eigenspaces of $A_1$.  For $A_j$, either $A_jE^{s} = E^{s}$ and $A_jE^{u} = E^{u}$ in which case $A_j$ and $A_1 $ commute, or either $A_jE^{s} \neq E^{s}$ or $A_jE^{u} \neq E^{u}$.  In the latter case we actually have  $A_jE^{s} \neq E^{s}$ and $A_jE^{u} \neq E^{u}$ and $A_1$ and $A_j$ do not commute.  If $A_1,\dots, A_m$ are  hyperbolic and all commute, they are necessarily products of a single matrix. 
\end{proof}

We have the following proposition, which shows that coincidence of  positive Lyapunov exponents % rigidity 
implies the topological conjugacy is smooth along fiberwise expanding foliation $\cF^u$.
%%%\red{Fix   a probability vector  $p=\{p_i\}_{1\le i\le m}$  on $\{A_i\}_{i=1}^m\subseteq{GL}(d,\ZZ)$ with $p_i(A_i)>0$ for every $1\le i\le m$.  }

\begin{Proposition}\label{prop:2-H-smooth}
Let $\nu$ be a probability measure on $\cU=\cup\cU_i$ with $\nu(\cU_i)=p_i$. Let $F:\Omega\times\TT^2\to\Omega\times\TT^2$ be the corresponding skew product system and let $\mu$ be the fiberwise SRB-measure of $F$ satisfying $\pi_*(\mu)=\nu^\ZZ$. 
Let $p$  be the probability measure  on  $\{A_i\}_{i=1}^m\subseteq{GL}(d,\ZZ)$ with $p(A_i)=p_i$ for every $1\le i\le m$.
Denote by
	\begin{itemize}
	\item $\lambda^+(\mu)$ the positive fiberwise Lyapunov exponent of $\mu$;
	\item $\lambda^+(p)$ the positive Lyapunov exponent for any ergodic measure of invariant under $\bar{A}$ and  projecting to $\nu^\ZZ$ as in Lemma \ref{lem:nu-exponent}. 
	\end{itemize}
Then  $\lambda^+(\mu)\le \lambda^+(p)$.  

 Moreover, if $\lambda^+(\mu)=\lambda^+(p)$, then the topological conjugacy $H$ is $C^{1+}$ smooth along leaves of fiberwise unstable foliation $\cF^u$ on the support of $\mu$:
	$$
	\supp(\mu)=\supp(\nu)^\ZZ\times\TT^2\subseteq\Omega\times\TT^2 \qquad \text{and} \qquad
	H_*(\mu)=\nu^\ZZ\times\Leb_{\TT^2}.
	$$
\end{Proposition}

\begin{proof}
	Let  $\mu_0:=H_*(\mu)$.  Then $\mu_0$ is $\bar{A}$-invariant, ergodic, and satisfies $\pi_*(\mu_0)=\nu^\ZZ$.  % on $\Omega$.
	The inequality $\lambda^+(\mu)\le \lambda^+(p)$ follows immediately since the conjugacy $H$ preserves fiberwise metric entropy.  Indeed, letting $h_\mu^{\mathrm{fib}}$ denote the fiber entropy of a skew system, by a fiberwise Ledrappier-Young formula and Ruelle inequality, we have 
    \begin{align*}
	    \lambda^+(\mu)&=h_\mu^{\mathrm{fib}}(F)	=h_{\mu_0}^{\mathrm{fib}}(\bar A)\le  \lambda^+(p).
    \end{align*}

	 Since, $\lambda^+(\mu)=\lambda^+(p)$, Lemma \ref{lem:nu-exponent} implies 	$$
	\lambda^+(\mu_0)=\lambda^+(p)=\lambda^+(\mu).
	$$
	We apply Proposition \ref{prop:u-gibbs}. Since the unstable foliation $\cF^u$ for $F$ has dimension one, $\lambda^+(\mu_0)=\lambda^+(\mu)$ implies $H$ is uniformly $C^{1+}$-smooth along leaves of $\cF^u$ and $\mu_0$ is a $u$-Gibbs measure of $\bar{A}$ along the linear expanding foliation $\cL^u$ for $\bar A$.  The unique fiberwise $u$-Gibbs state for $\bar A$ has conditional measures along leaves of $\cL^u_{\omega}$ that coincide (up to choice of normalization) with the Lebesgue measure of the leaf. Since on each fiber $\TT^2_{\omega}$, the unstable foliation $\cL^u_\omega$ is a linear irrational foliation, the conditional $u$-Gibbs measure of $\bar{A}$ along $\cL^u_{\omega}$ is the Lebesgue measure $\Leb_{\TT^2}$ on $\TT^2_{\omega}$. This proves $\mu_0=\nu^\ZZ\times\Leb_{\TT^2}$.
\end{proof}

We now consider the case of a random perturbations and arbitrary Anosov diffeomorphism $f\in\diff^2(\TT^2)$. Denote by $A=f_*\in{GL}(2,\ZZ)$ the linearization of $f$.  Since $f$ is topologically conjugate to $A$, by the structural stability of $f$, the perturbed skew system $F$ will be fiberwise conjugate to a direct product system $\bar A =(\sigma\times A)\colon \Omega\times\TT^2\to\Omega\times\TT^2$.

%From the continuity of stable and unstable bundles with respect to Anosov diffeomorphisms, we have the following lemma.

\begin{Lemma}
	There exists a neighborhood $\cU\subseteq\diff^2(\TT^2)$ of $f$ such that for $\Omega=\cU^\ZZ$, the corresponding skew product system 
	$$
	F:\Omega\times\TT^2\to\Omega\times\TT^2, \qquad
	F(\omega,x)=\left(\sigma(x),f_\omega(x)\right)
	$$ 
	is fiberwise Anosov. Moreover, if we denote $\bar{A}=\sigma\times A:\Omega\times\TT^2\to\Omega\times\TT^2$ the product system,  there exists a topological conjugacy $H:\Omega\times\TT^2\to\Omega\times\TT^2$ with
	$$
	H(\omega,x)=\left(\sigma(\omega),H_\omega(x)\right), 
	\qquad \text{such~that} \qquad
	H\circ F=\bar{A}\circ H,
	$$
    and $H$ maps the fiberwise expanding foliation $\cF^u$ of $F$ to the fiberwise linear expanding foliation $\cL^u$ of $A$.
\end{Lemma}

\begin{proof}
	Since $\{f\}$ is cone hyperbolic, the skew product $F$ is fiberwise Anosov.  As in the proof of Proposition \ref{prop:structure-stable}, fiberwise Anosov systems are fiberwise structurally stable.  Thus 
there are  topological conjugacies fibering over the identity between $F$ and $(\sigma\times f)\colon \Omega \times \TT^2\to \Omega \times \TT^2$ and between  $(\sigma\times f)\colon \Omega \times \TT^2\to \Omega \times \TT^2$ and $\bar{A}=(\sigma\times A)\colon \Omega \times \TT^2\to \Omega \times \TT^2$.	 Composing these gives the topological conjugacy $H$ in the lemma.
%	
%	 is deduced from robustness of Anosov diffeomorphisms and continuity of stable and unstable bundles. \red{this doesn't really follow from 5.3.}
%	From Proposition \ref{prop:structure-stable}, $F$ is topological conjugated to $A_0=\sigma\times f:\Omega\times\TT^2\to\Omega\times\TT^2$. Since $f$ is topological conjugated to $A$, we have $F$ conjugates to $\bar{A}=\sigma\times A$.
\end{proof}

In this setting, we also have the following proposition, whose proof is  identical to the proof of Proposition \ref{prop:2-H-smooth}.  % the positive Lyapunov exponent rigidity implies the topological conjugacy is smooth along  fiberwise expanding foliation $\cF^u$.
%\red{Not really a corollary... just the same proof}
\begin{Proposition}\label{cor:2-H-smooth}
	Let $\nu$ be a probability supported on the neighborhood $\cU$ of $f$ in $\diff^2(\TT^2)$. Let $\mu$ be the fiberwise SRB-measure of the skew product $F$ satisfying $\pi_*(\mu)=\nu^\ZZ$. 
	
	If the positive Lyapunov exponent $\lambda^+(\mu)$ of $\mu$ is equal to the positive Lyapunov exponent $\lambda^+(A)$ of $A$, then $H$ is $C^{1+}$ smooth along leaves of fiberwise unstable foliation $\cF^u$ on the support of $\mu$:
	$$
	\supp(\mu)=\supp(\nu)^\ZZ\times\TT^2\subseteq\Omega\times\TT^2 \qquad \text{and} \qquad
	H_*(\mu)=\nu^\ZZ\times\Leb_{\TT^2}.
	$$
\end{Proposition}

%\begin{Remark}
%	The proof of Corollary \ref{cor:2-H-smooth} is exactly the same to Proposition \ref{prop:2-H-smooth}.
We remark that by taking $\nu$ to be an atomic measure $\nu=\delta_g$ for some $g\in\cU$ conjugate $f$, then Proposition \ref{cor:2-H-smooth} is precisely \cite[Theorem G]{SY}. %% for a single Anosov diffeomorphism.
%\end{Remark}

\subsection{Positive spectrum rigidity for generic automorphisms on $\TT^d$}

In this subsection, we consider the case that $\cA=\{A_i\}_{i=1}^m\subseteq{GL}(d,\ZZ)$ be a family of commuting Anosov automorphisms admitting the same finest dominated splitting
$$
T\TT^d=L^s_1\oplus\cdots\oplus L^s_l\oplus L^u_1\oplus\cdots\oplus L^u_k.
$$ 
Moreover, we require that at least one of $A_i$ satisfies the generic assumption, i.e. $A_i$ is hyperbolic, irreducible and the dimension of each $L^{s/u}_j$ is 1 or 2, where the eigenvalues along each 2-dimensional subspace $L^{s/u}_j$ correspond  to a pair of complex conjugate eigenvalues of $A_i$ that are not  real multiples of $i$.  Following the classification of centralizers of such $A_i$ \cite[Proposition 3.7]{KKS}, every $A_{i'}\in\cA$ also satisfies the generic assumption.

Now we take a family of neighborhoods $\cU_i$ of $A_i$ in $\diff^2(\TT^d)$, denote $\cU=\cup\cU_i$ and $\Omega=\cU^\ZZ$, such that the corresponding skew product
$$
F:\Omega\times\TT^d\to\Omega\times\TT^d, \qquad
F(\omega,x)=\left(\sigma(\omega),f_\omega(x)\right)
$$
is fiberwise Anosov, and satisfies the conclusion of Proposition \ref{prop:structure-stable} and Proposition \ref{prop:leaf-conjugacy}, then $F$ admits the fiberwise finest dominated splitting 
$$
T\TT^d_\omega=
E^s_{1,\omega}\oplus\cdots E^s_{l,\omega}\oplus E^u_{1,\omega}\oplus\cdots E^u_{k,\omega}, 
\qquad \text{with} \qquad
{\rm dim}E^{s/u}_{j,\omega}={\rm dim}L^{s/u}_j,\quad \forall j.
$$

Let $H:\Omega\times\TT^d\to\Omega\times\TT^d$ be the topological conjugacy defined in  Proposition \ref{prop:structure-stable} and satisfying Proposition \ref{prop:leaf-conjugacy}. %which satisfies $H\circ F=\bar{A}\circ H$ where $\bar{A}$ is the extension skew product defined in Definition \ref{def:extension}.

The following proposition is the main result of this subsection.
%\red{Again, fix a probability vector  $p=\{p_i\}_{1\le i\le m}$  on $\{A_i\}_{i=1}^m\subseteq{GL}(d,\ZZ)$ with $p_i(A_i)>0$ for every $1\le i\le m$.  }

\begin{Proposition}\label{prop:d-H-smooth}
	Let $\nu$ be a probability on $\cU=\cup\cU_i$ and define a probability measure  $p=\{p_i\}_{1\le i\le m}$  on $\{A_i\}_{i=1}^m\subseteq{GL}(d,\ZZ)$ so that  $$p(A_i)=p_i := \nu(\cU_i).$$	
%	with $\nu(\cU_i)=p_i$.
	 Let $\mu$ be the fiberwise SRB-measure of the skew product $F:\Omega\times\TT^d\to\Omega\times\TT^d$ satisfying $\pi_*(\mu)=\nu^\ZZ$. For every $1\leq j\leq k$, denote 
	\begin{itemize}
%\note{$\{p_i\}$ or p?}		
\item $\lambda(\{p_i\},L^u_j)$ the fiberwise Lyapunov exponent of any ergodic meausre of $\bar{A}$ projected to $\nu^\ZZ$ as Lemma \ref{lem:nu-exponent}, where
		$$
		\lambda(\{p_i\},L^u_j)=\sum_{i=1}^m p_i\cdot\lambda(A_i,L^u_j).
		$$
		\item $\lambda(\mu,E^u_j)$ (or $\lambda_{\min}(\mu,E^u_j)\leq\lambda_{\max}(\mu,E^u_j)$) the fiberwise Lyapunov exponents of $\mu$ along $E^u_j$ if ~${\rm dim}E^u_j=1$ (or ~${\rm dim}E^u_j=2$).
	\end{itemize}

	Suppose for every $1\leq j\leq k$, the corresponding Lyapunov exponents of $F$ and $\bar{A}$ are all equal: 
	\begin{itemize}
		\item either ~${\rm dim}E^u_j=1$ and ~$\lambda(\{p_i\},L^u_j)=\lambda(\mu,E^u_j)$;
		\item or ~${\rm dim}E^u_j=2$ and ~$\lambda(\{p_i\},L^u_j)=\lambda_{\min}(\mu,E^u_j)=\lambda_{\max}(\mu,E^u_j)
		    :=\lambda(\mu,E^u_j)$.
	\end{itemize}
	Then $H_*(\mu)=\nu^\ZZ\times\Leb_{\TT^d}$ and on the support of $\mu$
	$$
	\supp(\mu)=\supp(\nu)^\ZZ\times\TT^d\subseteq\Omega\times\TT^d,
	$$
	 the topological conjugacy $H$ satisfies
	 \begin{enumerate}
	 	\item $H$ intertwines intermediate unstable invariant foliations: 
	 	$$
	 	H_{\omega}\left(\cF^u_{j,\omega}\right)=\cL^u_{j,\omega},
	 	\qquad 
	 	\forall j=1,\cdots,k,\quad \forall \omega\in\supp(\nu)^\ZZ.
	 	$$ 
	 	In particular, this implies the bundle $E^s_{\omega}\oplus E^u_{{k},\omega}$ integrates to a foliation $\cF^s_{\omega}\oplus\cF^u_{{k},\omega}$ on $\TT^d_{\omega}$ for every $\omega\in\supp(\nu)^\ZZ$. %\note{$k=j$??}
	    \item For every $j=1,\cdots,k$, $H$ is $C^{1+}$-smooth along leaves of fiberwise invariant foliation $\cF^u_j$, and the derivative $DH|_{E^u_j}$ is H\"older continuous with  H\"older constants uniform over $(\omega,x)\in\supp(\mu)$.
	 \end{enumerate}
	 In particular, $H$ is $C^{1+}$-smooth along leaves of the fiberwise unstable foliation $\cF^u$ on $\supp(\mu)$.
\end{Proposition}

\begin{Remark}
	When all $L^u_j$ are 1-dimensional, this proposition is basically the skew product version of \cite[Theorem B]{SY}.
	For generic automorphisms which has some $L^u_j$ is 2-dimensional, this was proved by \cite[Theorem 1.1]{GKS1}. The key point to handle the 2-dimensional bundles is applying the ``Two-dimensional Continuous Amenable Reduction'', see \cite[Proposition 3.1]{GKS1}.
\end{Remark}

We prove Proposition \ref{prop:d-H-smooth} at the end of this subsection. The proof consists of following lemmas and induction steps.
 First, we have the following observation.

\begin{Lemma}\label{lem:mu=leb}
	Assume the skew product system $F$ and fiberwise SRB-measure $\mu$ satisfies the assumption of Proposition \ref{prop:d-H-smooth}, then
	$$
	H_*(\mu)=\nu^\ZZ\times\Leb_{\TT^d}
	\qquad \text{and} \qquad
	\supp(\mu)=\supp(\nu)^\ZZ\times\TT^d.
	$$
\end{Lemma}

\begin{proof}
	From the assumption, the corresponding fiberwise Lyapunov exponents of $(F,\mu,E^u_j)$ and $(\bar{A},H_*(\mu),L^u_j)$ are all equal for $j=1,\cdots,k$. By taking the sum, we have
	$$
	\sum_{j=1}^k{\rm dim}E^u_j\cdot\lambda(\mu,E^u_j)~=~
	\sum_{j=1}^k{\rm dim}L^u_j\cdot\lambda(\{p_i\},L^u_j)~=~
	\sum_{j=1}^k{\rm dim}L^u_j\cdot\lambda(H_*(\mu),L^u_j).
	$$
	Since $\mu$ is fiberwise SRB-measure, which is also a $u$-Gibbs measure along $\cF^u$, Proposition \ref{prop:u-gibbs} shows that $H_*(\mu)$ is a $u$-Gibbs measure of $\bar{A}$ along $\cL^u$. Since the fiberwise dynamics of $\bar{A}$ are compositions of $\{A_i\}_{i=1}^m$, so the conditional measures of a $u$-Gibbs measure of $\bar{A}$ along $\cL^u$-leaves must be Lebesgue measure on $\cL^u$-leaves. The fiberwise conditional measure of $H_*(\mu)$ along $\TT^d$-fiber is Lebesgue measure on $\TT^d$. From the equidistribution of $\cL^u$ in $\TT^d$, this implies $H_*(\mu)=\nu^\ZZ\times\Leb_{\TT^d}$ and $\supp(\mu)=\supp(\nu)^\ZZ\times\TT^d$.
\end{proof}

The following three lemmas allow us to prove Proposition \ref{prop:d-H-smooth} inductively. These steps mainly follow \cite[Section 2]{GKS} and \cite[Section 3]{GKS1} for smooth conjugacy of the perturbation of a single generic Anosov automorphism. The key fact is that in our setting, even the space $\Omega\times\TT^d$ is not compact, but the skew product system $F:\Omega\times\TT^d\to\Omega\times\TT^d$ still satisfies the following properties:
\begin{itemize}
	\item the fiberwise invariant bundles of $F$ are uniformly H\"older continuous and their angles are uniformly bounded;
	\item the fiberwise dynamics of $F$ has bounded norm and the derivatives along invariant bundles are H\"older continuous with uniform H\"older constants;
	\item $F$ admits uniformly contracting and expanding rate on both the base dynamics $\sigma:\Omega\to\Omega$ and the fiberwise dynamics on $\TT^d$-fibers.
\end{itemize} 
These properties allowed us to replay the proof in \cite{GKS,GKS1}.

\begin{Lemma}\label{lem:j-smooth}
	For $1\leq j\leq k$, if $H(\cF^u_j)=\cL^u_j$ on $\supp(\mu)$ and fiberwsie Lyapunov exponents of $(F,\mu)$ along $E^u_j$ are all equal to exponents of $(\bar{A},H_*(\mu))$, then 
	\begin{itemize}
		\item $\mu$ is a $u$-Gibbs state of $F$ along $\cF^u_j$; ~and
		\item $H$ is $C^{1+}$ smooth along leaves of $\cF^u_j$ on $\supp(\mu)$. Moreover, the derivatives $(DH|_{E^u_j})$ is H\"older continuous with respect to both $\omega$ and $x$ for every $(\omega,x)\in\supp(\mu)$ with  H\"older constants  uniformly bounded.
	\end{itemize}
\end{Lemma}

\begin{Remark}
	When ${\rm dim}E^u_{j,\omega}=1$, Lemma \ref{lem:j-smooth} is a direct consequence of Proposition \ref{prop:u-gibbs}.  %and Corollary \ref{coro:u-gibbs}. 
	When ${\rm dim}E^u_{j,\omega}=2$, we need to apply the conformal structure and two dimensional continuous amenable reduction following \cite{GKS1}.
\end{Remark}

\begin{Lemma}\label{lem:j+1}
	For every $1\leq j<k$, if the topological conjugacy $H$ satisfies
	\begin{itemize}
		\item $H(\cF^u_i)=\cL^u_i$ for $1\leq i\leq j$ on $\supp(\mu)$;
		\item $H$ is $C^{1+}$-smooth along $\cF^u_{(1,j)}$ on $\supp(\mu)$,
	\end{itemize}
	then we have $H(\cF^u_{j+1})=\cL^u_{j+1}$ on $\supp(\mu)$.
\end{Lemma}

 %to induct smoothness.
Since $\Omega=\cU^\ZZ$ where all $f\in\cU_i\subseteq\cU$ are close to $A_i$, all invariant bundles $E^u_{(1,j)}$ and $E^u_{j+1}$ are H\"older continuous with uniform H\"older constants and bounded angles. We need the following Journ\'e's theorem \cite{J} to obtain smoothness as we induct.  %We have the following Journ\'e's theorem.

\begin{Lemma}[\cite{J}]\label{lem:Journe}
	For every $1\leq j<k$, if the topological conjugacy $H$ satisfies $H(\cF^u_{j+1})=\cL^u_{j+1}$, and $H$ is $C^{1+}$-smooth along both $\cF^u_{(1,j)}$ and $\cF^u_{j+1}$ on $\supp(\mu)$, then $H$ is $C^{1+}$-smooth along $\cF^u_{(1,j+1)}$ on $\supp(\mu)$, and the derivative $(DH|_{E^u_{(1,j+1)}})$ is H\"older continuous with respect to both $\omega$ and $x$ for every $(\omega,x)\in\supp(\mu)$, where H\"older contants are uniformly bounded.
\end{Lemma}

We prove Lemma \ref{lem:j-smooth} and Lemma \ref{lem:j+1} in the next two subsections.
Now we can prove Proposition \ref{prop:d-H-smooth}, assuming  Lemmas \ref{lem:j-smooth} and  \ref{lem:j+1}.

\begin{proof}[Proof of Proposition \ref{prop:d-H-smooth}]
    From the fourth item of Proposition \ref{prop:leaf-conjugacy}, we know that $H(\cF^u_1)=\cL^u_1$. By applying Lemma \ref{lem:j-smooth}, we know that $H$ is $C^{1+}$-smooth along $\cF^u_1$. 
    
    Assume $H(\cF^u_i)=\cL^u_i$ for $1\leq i\leq j$ and $H$ is $C^{1+}$-smooth on $\cF^u_{(1,j)}$ with uniformly bounded derivatives on $\supp(\mu)$, then Lemma \ref{lem:j+1} implies $H(\cF^u_{j+1})=\cL^u_{j+1}$ on $\supp(\mu)$. 
    
    As we assumed fiberwsie Lyapunov exponents of $(F,\mu)$ along $E^u_{j+1}$ are all equal to Lyapunov exponents of $(\bar{A},H_*(\mu))$ along $L^u_{j+1}$, Lemma \ref{lem:j-smooth} implies $H$ is $C^{1+}$-smooth along $\cF^u_{j+1}$. 
    Applying Lemma \ref{lem:Journe}, we have $H$ is $C^{1+}$-smooth along $\cF^u_{(1,j+1)}$, and the derivative $(DH|_{E^u_{(1,j+1)}})$ is H\"older continuous with respect to both $\omega$ and $x$ for every $(\omega,x)\in\supp(\mu)$, where H\"older contants are uniformly bounded.

    By induction and Lemma \ref{lem:mu=leb} which shows that $H_*(\mu)=\nu^\ZZ\times\Leb_{\TT^d}$, $H(\cF^u_j)=\cL^u_j$ and $H$ is $C^{1+}$-smooth along $\cF^u_j$ for every $j=1,\cdots,k$, where the derivative $(DH|_{E^u_j})$ is H\"older continuous with respect to both $\omega$ and $x$ for every $(\omega,x)\in\supp(\mu)$, where H\"older contants are uniformly bounded. 
    The same holds for $H$ on the whole expanding foliation $\cF^u$ on $\supp(\mu)$.
    This finishes the proof of Proposition \ref{prop:d-H-smooth}.
\end{proof}

\subsection{Smoothness along 2-dimensional conformal bundles}

In this subsection, we prove Lemma \ref{lem:j-smooth}. The proof mainly follows \cite[Proposition 2.4]{GKS1} but in the skew product systems. In this subsection, $F$ and $\bar{A}$ satisfy all assumptions in Lemma \ref{lem:j-smooth}.
We first prove the first item.

\begin{Lemma}\label{lem:mu=u-gibbs}
		For $1\leq j\leq k$, if $H(\cF^u_j)=\cL^u_j$ on $\supp(\mu)$ and of the fiberwsie Lyapunov exponents of $(F,\mu)$ along $E^u_j$ are all equal to  the Lyapunov exponents of $(\bar{A},H_0)$ along $\cL^u_j$ (i.e. $\lambda(\{p_i\},L^u_j)$), then $\mu$ is a $u$-Gibbs state of $F$ along $\cF^u_j$.
\end{Lemma}

\begin{proof}
	From Lemma \ref{lem:mu=leb}, we know that $H_*(\mu)=\nu^\ZZ\times\Leb_{\TT^d}$. Since $\cL^u_j$ is the linear foliation on $\TT^d$, the disintegration of Lebesgue measure $\Leb_{\TT^d}$ along $\cL^u_j$ is also leafwise Lebesgue measure along $\cL^u_j$-leaves. So $H_*(\mu)$ is a $u$-Gibbs of $\bar{A}$ along $\cL^u_j$. 
	
	From the assumption that fiberwise Lyapunov exponents of $(F,\mu)$ along $E^u_j$ are all equal to fiberwise Lyapunov exponents of $(\bar{A},H_*(\mu))$ along $L^u_j$, Proposition \ref{prop:u-gibbs} shows that $\mu=(H^{-1})_*\circ H_*(\mu)$ is a $u$-Gibbs state of $F$ along $\cF^u_j$. This proves the lemma.
\end{proof}

The proof of second item of Lemma \ref{lem:j-smooth} basically follows \cite[Proposition 2.4]{GKS1}, but in skew product systems.
If ${\rm dim}E^u_j=1$, then it is a direct consequence of Proposition \ref{prop:u-gibbs}.
For the case that ${\rm dim}E^u_j=2$, we need the following two dimensional continuous amenable reduction theorem which has been proved in \cite[Proposition 3.1]{GKS1} and \cite[Theorem 7.4 \& Corollary 7.6]{Sa}.

\begin{Proposition}\label{prop:reduction}
	Let $f:X\to X$ be a hyperbolic system which satisfies Anosov Closing Lemma. Let $E$ be a vector bundle over $X$ with two-dimensional fibers, and $\mu$ be an ergodic $f$-invariant probability measure with full support and local product structure.
	Let $\Phi:E\to E$ be a H\"older continuous fiber bunched cocycle over $f$ with one
	Lyapunov exponent with respect to $\mu$. Then at least one of the following holds:
	\begin{enumerate}
		\item $\Phi$ is conformal with respect to a H\"older continuous Riemannian norm
		on $E$;
		\item $\Phi$ preserves a H\"older continuous one dimensional sub-bundle;
		\item $\Phi$ preserves a H\"older continuous field of two transverse lines.
	\end{enumerate}
\end{Proposition}

\begin{Corollary}\label{coro:conformal}
	 Assume $F$ satisfies Proposition \ref{prop:d-H-smooth}.  For every subbundle with  ${\rm dim}E^u_j=2$ with both Lyapunov exponents of $(F,\mu)$ are equal, the cocycle 
	 $DF$ is conformal when restricted on $\supp(\mu)=\supp(\nu)^\ZZ\times\TT^d$ with respect to a H\"older continuous Riemannian norm on $E^u_j$.
\end{Corollary}

%%%%\blue{Note: conjugacy preserves product structure!!}
\begin{proof}
By Lemma \ref{lem:mu=leb}, we know that $\supp(\mu)=\supp(\nu)^\ZZ\times\TT^d$ and $$H_*\mu= \nu^\ZZ\times\Leb_{\TT^d}.$$ 
In particular, $H_*\mu$ has local product structure.  Since the topological conjugacy $H$ intertwines total stable and unstable sets (including combinatorial components) over $\supp(\mu)$ for the skew systems $F$ and $\bar A$, it follows that $\mu$ also has local  product structure.  Thus we can apply Proposition \ref{prop:reduction} to  the restriction of $DF$ to the subbundle  $E^u_j$ over  $\supp(\mu)$. We need to eliminate possibilities 2 and 3.
	
	Recall we assume the collection $\cA=\{A_i\}_{i=1}^m$ is generic.  This implies when we take the neighborhood $\cU_i$ of each $A_i$ sufficiently small enough, every $f_i\in\cU_i$ has fixed point such that the derivative $DF_{E^u_j}$ has a pair of complex conjugate eigenvalues which are not real multiples $i$.  This implies the cocycle $DF_{E^u_j}$  has neither a continuous invariant one-dimensional sub-bundle nor a continuous invariant field of two transverse lines.  This eliminates possibilities 2 and 3 in Proposition \ref{prop:reduction} for the cocycle  $DF_{E^u_j}$.
	\end{proof}

 From the conformal structure, we can show that the conjugacy $H$ is Lipschitz along $\cF^u_j$. Recall that for the extension system $\bar{A}:\Omega\times\TT^d\to\Omega\times\TT^d$, the fiberwise dynamics are compositions of $\{A_i\}_{i=1}^m$ which are all linear, which implies for any $x\in\TT^d$
 $$
 D\bar{A}_{\omega}(x)=\bar A_{\omega}%:=D\bar{A}_\omega, 
 \qquad
 \text{where}~A_{\omega}=A_i
 \quad \text{if} \quad \Upsilon(f_{\omega_0}) = A_i. %%\in\cU_i.
 $$
 
\begin{Lemma}\label{lem:H-Lip}
	Assume ${\rm dim}E^u_j=2$ and both Lyapunov exponents of $(F,\mu)$ along $E^u_j$ are equal to the Lyapunov exponent of $(\bar{A},H_*(\mu))$ along $L^u_j$.  Then:
	\begin{itemize}
		\item The Jacobian of $F$ along $\cF^u_j$ is cohomologous to the Jacobian of $\bar{A}$ along $\cL^u_j$ on $\supp(\mu)$, i.e. there exists a H\"older continuous function $u:\supp(\mu)=\supp(\nu)^\ZZ\times\TT^d\to\RR$, which is uniformly bounded on $\supp(\mu)$, such that for every $(\omega,x)\in\supp(\mu)$,
		\begin{align*}
			\log\left|\det\left(DF|_{E^u_{j,\omega}(x)}\right)\right| &=
			\log\left|\det\left(D\bar{A}|_{L^u_{j,\omega}(H_{\omega}(x))}\right)\right|
			+u\circ\bar{A}(\omega,x)-u(\omega,x) \\
			&=\log\left|\det\left(D\bar{A}_\omega|_{L^u_j}\right)\right|
			+u\circ\bar{A}(\omega,x)-u(\omega,x).
		\end{align*}
		\item The topological conjugacy $H$ is Lipschitz along $\cF^u_j$ on $\supp(\mu)=\supp(\nu)^\ZZ$ and $DH|_{\cF^u_j}$ exists and is invertible almost everywhere with respect to $\mu$.
	\end{itemize}
\end{Lemma}

\begin{proof}
	From the assumption that both Lyapunov exponents of $(F,\mu)$ are equal to the Lyapunov exponent of $(\bar{A},H_*(\mu))$ along $L^u_j$, the third item of Proposition \ref{prop:u-gibbs} implies for every periodic point $(\omega,x)=F^l(\omega,x)\in\supp(\mu)$, %\red{how do we get to periodic points? third item?} 
	we have
	$$
	\sum_{i=0}^{l-1}\log\left|\det\left(DF|_{E^u_{j,\sigma^i(\omega)}(f_\omega^ix)}\right)\right|
	=
	\sum_{i=0}^{l-1}\log\left|\det\left(D\bar{A}|_{L^u_{j,\sigma^i(\omega)}(H_{\sigma^i(\omega)}\circ f_\omega^ix)}\right)\right|.
	$$
	Since $\supp(\mu)=\supp(\nu)^\ZZ\times\TT^d$, following Lemma \ref{lem:livsic}, %and \red{Remark \ref{rk:livsic}, } 
	there exists a H\"older continuous function $u:\supp(\mu)\to\RR$ such that for every $(\omega,x)\in\supp(\mu)$,
	\begin{align*}
		\log\left|\det\left(DF|_{E^u_{j,\omega}(x)}\right)\right| &=
		\log\left|\det\left(D\bar{A}|_{L^u_{j,\omega}(H_{\omega}(x))}\right)\right|
		+u\circ\bar{A}(\omega,x)-u(\omega,x) \\
		&=\log\left|\det\left(D\bar{A}_\omega|_{L^u_j}\right)\right|
		+u\circ\bar{A}(\omega,x)-u(\omega,x).
	\end{align*}
	This proves the first item.
	
	\vskip3mm
	
	The proof of second item mainly follows \cite[Section 3.3]{GKS1}. We just sketch the proof. Notice our conjugacy $H$ maps the nonlinear foliation $\cF^u_j$ to linear one $\cL^u_j$, which corresponds to $h^{-1}$ in \cite[Section 3.3]{GKS1}. Let $\bar{\cL}$ be the linear integral foliation of
	$$
	L^s_1\oplus\cdots\oplus L^s_l\oplus L^u_1\oplus\cdots\oplus L^u_{j-1}\oplus L^u_{j+1}\oplus\cdots\oplus L^u_k.
	$$
	
	For every $(\omega,x)\in\supp(\mu)=\supp(\nu)^\ZZ\times\TT^d$, define $H^{-1}_{0,\omega}(x)$ as the intersection of local leaves
	$$
	H^{-1}_{0,\omega}(x)=\cF^u_{j,\omega,loc}(H^{-1}_{\omega}(x))\cap\bar{\cL}_{\omega,loc}(x).
	$$
	then $H^{-1}_0:\supp(\nu)^\ZZ\times\TT^d\to\supp(\nu)^\ZZ\times\TT^d$ is well defined, close to $H^{-1}$ and satisfies:
	\begin{enumerate}
		\item $H^{-1}_{0,\omega}(\cL^u_{j,\omega})=\cF^u_{j,\omega}$ and 
		$H^{-1}_{0,\omega}(\cL^u_{j,\omega}(x))=\cF^u_{j,\omega}(H^{-1}_{\omega}(x))$ for every $x\in\TT^d_{\omega}$;
		\item $\sup_{(\omega,x)\in\supp(\mu)}d_{\cF^u_j}(H^{-1}_{0,\omega}(x),H^{-1}_{\omega}(x))<+\infty$, where $d_{\cF^u_j}(\cdot,\cdot)$ is the distance along the leaf of $\cF^u_j$;
		\item $H^{-1}_{0,\omega}$ is a $C^{1+}$ diffeomorphisms along the leaves of $\cL^u_j$ with uniformly bounded derivatives on $\supp(\mu)$;
		\item $H^{-1}=\lim_{n\to\infty}H^{-1}_n$ uniformly on $\supp(\nu)^\ZZ\times\TT^d$ 
		  where $H^{-1}_n=F^{-n}\circ H^{-1}_0\circ\bar{A}^n$.
	\end{enumerate}
	
	Now we show that the derivatives $DH^{-1}_n$ along $\cL^u_j$ are uniformly bounded.
	\begin{align}\label{equ:H_n}
		&\quad\left\|DH^{-1}_n|_{L^u_{j,\omega}(x)}\right\| \notag \\
		&\leq
		\left\|\left(DF^n|_{E^u_{j,\omega}(F^{-n}_{\sigma^n(\omega)} \circ H^{-1}_{0,\sigma^n(\omega)}\circ\bar{A}^n_{\omega}(x))}\right)^{-1}\right\|
		\cdot
		\left\|DH^{-1}_{0,\sigma^n(\omega)}
		   |_{L^u_{j,\sigma^n(\omega)}(\bar{A}^n_{\omega}(x))}\right\|
		\cdot\left\|D\bar{A}^n_{\omega}|_{L^u_{j,\omega}(x)}\right\| \notag \\
		&\leq
		\left\|\left(DF^n|_{E^u_{j,\omega}(H^{-1}_{n,\omega}(x))}\right)^{-1}\right\|
		\cdot \left\|D\bar{A}^n_{\omega}|_{L^u_{j,\omega}(x)}\right\|
		\cdot
		\sup\left\{~\left\|DH^{-1}_{0,\omega}|_{L^u_{j,\omega}(x)}\right\|:
		~(\omega,x)\in\supp(\mu)~\right\}
	\end{align}
	Since the last term is finite, we only need to show the product of first and second terms are uniformly bounded.
	
	We can take the norm $\|\cdot\|$ such that $D\bar{A}|_{L^u_j}$ is conformal, then we have
	$$
	\|D\bar{A}|_{L^u_{j,\omega}(x)}\|
	=\sqrt{\det\left(D\bar{A}|_{L^u_{j,\omega}}(x)\right)}
	=\sqrt{\det\left(A_{\omega}|_{L^u_j}\right)}:=b(\omega).
	$$
	From Corollary \ref{coro:conformal}, there exists a H\"older continuous Riemannian norm $\|\cdot\|^F$ on $E^u_j|_{\supp(\mu)}$, such that 
	$$
	\|DF(v)\|^F=a(\omega,x)\cdot\|v\|^F,
	\qquad \forall(\omega,x)\in\supp(\mu),\quad
	\forall v\in E^u_{j,\omega}(x),
	$$
	where
	$$
	a(\omega,x)=\sqrt{\det\left(DF|_{E^u_{j,\omega}}(x)\right)}.
	$$
	This implies $\|(DF|_{E^u_{j,\omega}(x)})^{-1}\|^F=a(\omega,x)^{-1}$.
	
	From the first item that the Jacobian of $DF$ along $E^u_j$ is H\"older cohomologous to the Jacobian of $\bar{A}$ along $L^u_j$, there exists some H\"older continuous function 
	$$
	\phi(\omega,x)=\exp\left(\frac{1}{2}\cdot u(\omega,x)\right):~\supp(\mu)\to\RR^+,
	$$ 
	where $u(\omega,x)$ is defined by the first item of this lemma, such that
	$$
	\frac{b(\omega)}{a(\omega,x)}=\frac{\phi(\omega,x)}{\phi\circ\bar{A}(\omega,x)}.
	$$
	
	Since $u(\omega,x)$ is uniformly bounded on $\supp(\mu)$, $\phi(\omega,x)$ is uniformly bounded from $0$ and $+\infty$ on $\supp(\mu)$.
	This implies
	\begin{align*}
		\left\|\left(DF^n|_{E^u_{j,\omega}(H^{-1}_{n,\omega}(x))}\right)^{-1}\right\|^F
		\cdot \left\|D\bar{A}^n_{\omega}|_{L^u_{j,\omega}(x)}\right\| &=
		\frac{b(\sigma^{n-1}(\omega))\cdot\cdots\cdot b(\sigma(\omega))\cdot b(\omega)}{
			 a\circ F^{n-1}(\omega,x)\cdot\cdots\cdot a\circ F(\omega,x)\cdot a(\omega,x)} \\
		&=\frac{\phi(\omega,x)}{\phi\circ\bar{A}^n(\omega,x)},
	\end{align*}
	which is uniformly bounded on $\supp(\mu)$.
	
	Since the norm $\|\cdot\|^F$ is equivalent to $\|\cdot\|$, from the equation \ref{equ:H_n}, we have $\left\|DH^{-1}_n|_{L^u_{j,\omega}(x)}\right\|$ is uniformly bounded in $(\omega,x)$ and $n$, this implies $H^{-1}=\lim_{n\to\infty}H_n^{-1}$ is Lipschitz along $\cL^u_j$ on $\supp(\mu)$.
	
	A simliar argument shows that the topological conjugacy $H$ is Lipschitz along $\cF^u_j$ on $\supp(\mu)$. In particular, since $\mu$ is a $u$-Gibbs state along $\cF^u_j$, $DH|_{\cF^u_j}$ exists and invertible almost everywhere with respect to $\mu$.
\end{proof}

Now we can finish the proof of Lemma \ref{lem:j-smooth}.

\begin{proof}[Proof of Lemma \ref{lem:j-smooth}]
	The first item was proved in Lemma \ref{lem:mu=u-gibbs}. Now we prove the second item which says that the conjugacy $H$ is smooth along $\cF^u_j$ with uniformly H\"older continous derivative along $\cF^u_j$ on $\supp(\mu)$.

	Firstly, we assume ${\rm dim}E^u_j=1$, then the fiberwise Lyapunov exponent of $(F,\mu)$ along $E^u_j$ is equal to $\lambda(\{p_i\},L^u_j)$ implies the topological conjugacy $H$ is $C^{1+}$-smooth along $\cF^u_j$ by Proposition \ref{prop:u-gibbs} again. Moverover, taking the derivative of $H\circ F=\bar{A}\circ H$ along $\cF^u_j$, we have
	$$
	\log\|D\bar{A}|_{L^u_j}\|=\log\|DF|_{E^u_j}\|+\log\|DH|_{E^u_j}\|\circ F-\log\|DH|_{E^u_j}\|.
	$$
	That is $\log\|DH|_{E^u_j}\|$ is the coboundary between $\log\|D\bar{A}|_{L^u_j}\|$ and $\log\|DF|_{E^u_j}\|$.
	Lemma \ref{lem:livsic} shows that $\log\|DH|_{E^u_j}\|$ is H\"older continuous.
	
	\vskip2mm
	
	This proves Lemma \ref{lem:j-smooth} when ${\rm dim}E^u_j=1$.
	Now we assume ${\rm dim}E^u_j=2$.

	From second item of Lemma \ref{lem:H-Lip}, we can differentiate $H\circ F=\bar{A}\circ H$ along $\cF^u_j$ on a set of full Lebesgue and obtain
	$$
	\left(DH|_{E^u_{j,\sigma(\omega)}(f_\omega(x))}\right)\circ
	\left(DF|_{E^u_{j,\omega}(x)}\right) =
	\left(D\bar{A}|_{L^u_{j,\omega}(H_\omega(x))}\right)\circ
	\left(DH|_{E^u_{j,\omega}(x)}\right).
	$$
	This implies $DF|_{E^u_j}$ and $D\bar{A}|_{L^u_j}$ are two conformal cocycles  which are cohomologous by a measurable cocycle $DH|_{E^u_j}$. Since both cocylces $DF|_{E^u_j}$ and $D\bar{A}|_{L^u_j}$ are bounded and H\"older continuous, then \cite[Theorem 2.7]{Sa15} shows that $DH|_{E^u_j}$ is H\"older continuous, see also \cite[Theorem 4.14]{Sa}.
	%Moreover,  and every pair of points $(\omega,x),(\omega',x')\in\supp(\mu)$ can be connected by stable and unstable leaves of $F$ and $\bar{A}$ with uniformly bounded length, so $DH|_{E^u_j}$ is also uniformly bounded on $\supp(\mu)$, 
	
	This finishes the proof of Lemma \ref{lem:j-smooth}.
\end{proof}

\subsection{Matching the strong unstable foliations}

In this subsection, we prove Lemma \ref{lem:j+1}. The proof mainly follows \cite[Proposition 2.2]{GKS} but in the skew product systems. In this subsection, $F$ and $\bar{A}$ satisfy all assumptions in Lemma \ref{lem:j+1}.

The proof of Lemma \ref{lem:j+1} is based on the following lemma.

\begin{Lemma}\label{lem:j+1'}
	For every $1\leq j<k$, assume that $H(\cF^u_{(j,k)})=\cL^u_{(j,k)}$, $H(\cF^u_j)=\cL^u_j$ and $H$ is $C^{1+}$-smooth along $\cF^u_j$ on $\supp(\mu)$, then $H(\cF^u_{(j+1,k)})=\cL^u_{(j+1,k)}$.
\end{Lemma}

We can prove Lemma \ref{lem:j+1} by assuming Lemma \ref{lem:j+1'}.

\begin{proof}[Proof of Lemma \ref{lem:j+1}]
	For every fixed $1\leq j<k$, we prove Lemma \ref{lem:j+1} by induction. Since $H(\cF^u_{(1,k)})=\cL^u_{(1,k)}$ and $H$ is $C^{1+}$-smooth along $\cF^u_1$ with uniformly bounded derivatives, then Lemma \ref{lem:j+1'} implies $H(\cF^u_{(2,k)})=\cL^u_{(2,k)}$. The fourth item of Proposition \ref{prop:leaf-conjugacy} shows that 
	$$
	H(\cF^u_{(1,2)})=\cL^u_{(1,2)}.
	$$ 
	By taking the intersection, we have
	$$
	H\left(\cF^u_2\right)=H\left(\cF^u_{(1,2)}\cap\cF^u_{(2,k)}\right)
	=\cL^u_{(1,2)}\cap\cL^u_{(2,k)}=\cL^u_2
	\quad \text{on}\quad \supp(\mu).
	$$ 
	From the assumption that $H$ is $C^{1+}$-smooth along $\cF^u_{(1,j)}$, it is also $C^{1+}$-smooth along $\cF^u_2$. 
	Then we can apply Lemma\ref{lem:j+1'} again, which shows that $H(\cF^u_{(3,k)})=\cL^u_{(3,k)}$ and consequently $H(\cF^u_3)=\cL^u_3$. Repeat this argument, we have $H(\cF^u_{j+1,k})=\cL^u_{(j+1,k)}$ and thus
	$$
	H\left(\cF^u_{j+1}\right)=H\left(\cF^u_{(1,j+1)}\cap\cF^u_{(j+1,k)}\right)
	=\cL^u_{(1,j+1)}\cap\cL^u_{(j+1,k)}=\cL^u_{j+1}
	\quad \text{on}\quad \supp(\mu).
	$$
\end{proof}

The proof of Lemma \ref{lem:j+1'} mainly follows \cite[Proposition 2.3]{GKS}. Since we have $\cF^u_j,~\cF^u_{(j+1,k)}$ are two uniformly transversal foliations in each leaf of $\cF^u_{(j,k)}$, and
$$
H\left(\cF^u_{(j,k)}\right)=\cL^u_{(j,k)},
\qquad
H\left(\cF^u_j\right)=\cL^u_j,
$$
the foliation $H(\cF^u_{(j+1,k)})$ is an $\bar{A}$-invariant foliation which sub-foliates $\cL^u_{(j,k)}$ and transverses to $\cL^u_{(j+1,k)}$ in each leaf of $\cL^u_{(j,k)}$.

 \begin{Lemma}\label{lem:linear-j+1-smooth}
	Denote by $\cL:=H(\cF^u_{(j+1,k)})$.  Then, the holonomy maps along leaves of $\cL$ between leaves of $\cL^u_j$ are uniformly $C^{1+}$-smooth; that is,  for every $\omega\in\supp(\nu)^\ZZ$ and $y\in\cL(x)\subseteq\TT^d_\omega$, if we  denote by 
	$$
	\Hol^{\cL}_{x,y}:~\cL^u_j(x)\to\cL^u_j(y)
	$$ 
	the holonomy map induced by $\cL$ between leaves of $\cL^u_j$, then $\Hol^{\cL}_{x,y}$ is uniformly $C^{1+}$-smooth and the derivatives $D\Hol^{\cL}_{x,y}$ uniformly converges to identity map $\Id$ on $L^u_j$ as $y\to x$.
\end{Lemma}

\begin{proof}
	Since $\cF^u_{(j+1,k)}$ are strong foliations sub-foliating $\cF^u_{(j,k)}$, it is $C^2$-smooth inside each leaf of $\cF^u_{(j,k)}$, see \cite[Proposition 3.9]{KS07}).
	The rest of the lemma comes from the following conjugacy:
	$$
	\Hol^{\cL}_{x,y}=
	\left(H_{\omega}|_{\cF^u_j(y)}\right)
	\circ\left(\Hol^{\cF^u_{(j+1,k)}}_{H_{\omega}^{-1}(x),H_{\omega}^{-1}(y)}\right)
	\circ\left(H_{\omega}^{-1}|_{\cL^u_j(x)}\right).
	$$
	Actually, $y\to x$ implies $H_{\omega}^{-1}(y)\to H_{\omega}^{-1}(x)$. Thus the derivatives of holonomy map induced by $\cF^u_{(j+1,k)}$ uniformly converges to identity map. From the continuity of derivatives of $H$ along $\cF^u_j$, we have $D\Hol^{\cL}_{x,y}$ uniformly converges to identity maps on $\cL^u_j(x)$.
\end{proof}

Using the smoothness of holonomy maps of $\cL$, we show that $\cL$ is a parallel translation inside each leaf of $\cL^u_{(j,k)}$, which we will used to show  that $\cL=\cL^u_{(j+1,k)}$.

\begin{Lemma}\label{lem:j+1-linear}
	For every $\omega\in\supp(\nu)^\ZZ$ and $y\in\cL(x)\subseteq\TT^d_\omega$, the holonomy map
	$$
	\Hol^{\cL}_{x,y}:~\cL^u_j(x)\to\cL^u_j(y)
	$$ 
	is a parallel translation inside $\cL^u_{(j,k),\omega}(x)$.
\end{Lemma}

\begin{proof}
	To prove that $\Hol^{\cL}_{x,y}$ is a parallel translation, we only need to show that  $D\Hol^{\cL}_{x,y}=\Id$. 
	Denote $x_n=\bar{A}^n_\omega(x)$ and $y_n=\bar{A}^n_\omega(y)$ with $y_n\in\cL(x_n)$ and $\lim_{n\to+\infty}d(x_n,y_n)=0$. Then we have
	$$
	\Hol^{\cL}_{x,y}=\bar{A}^n_{\sigma^{-n}(\omega)}\circ \Hol^{\cL}_{x_n,y_n}\circ \bar{A}^{-n}_\omega,
	$$
	where $\Hol^{\cL}_{x_n,y_n}:\cL^u_j(x_n)\to\cL^u_j(y_n)$ is also the holonomy map induced by $\cL$. 
	
	Differentiating and denoting $D\Hol^{\cL}_{x_n,y_n}=\Id+\Delta_n$, we obtain
	\begin{align*}
		D\Hol^{\cL}_{x,y} &=
		\left(\bar{A}^n_{\sigma^{-n}(\omega)}|_{L^u_j}\right) \circ
		D\Hol^{\cL}_{x_n,y_n} \circ
		\left(\bar{A}^{-n}_\omega|_{L^u_j}\right) \\
		&=\left(\bar{A}^n_{\sigma^{-n}(\omega)}|_{L^u_j}\right) \circ
		\left(\Id+\Delta_n\right) \circ
		\left(\bar{A}^{-n}_\omega|_{L^u_j}\right) \\
		&=\Id+\left(\bar{A}^n_{\sigma^{-n}(\omega)}|_{L^u_j}\right) \circ
		\Delta_n \circ
		\left(\bar{A}^{-n}_\omega|_{L^u_j}\right).
	\end{align*}
	From Lemma \ref{lem:linear-j+1-smooth} we know that $D\Hol^{\cL}_{x_n,y_n}\to\Id$ as $n\to+\infty$, i.e.  $\Delta_n\to0$ as $n\to+\infty$. 
	
	Since $\bar{A}$ is conformal along $\cL^u_j$, there exists some constant $C>0$ such that
	$$
	\left\|\left(\bar{A}^n_{\sigma^{-n}(\omega)}|_{L^u_j}\right) \circ
	\Delta_n \circ
	\left(\bar{A}^{-n}_\omega|_{L^u_j}\right)\right\|\leq C\cdot\|\Delta_n\|\to0
	\qquad \text{as} \quad n\to+\infty.
	$$
	This implies $D\Hol^{\cL}_{x,y}=\Id$ and $\Hol^{\cL}_{x,y}$ is a parallel translation.
\end{proof}

Now we can prove Lemma \ref{lem:j+1'}. The key point is that if $\cL\neq\cL^u_{(j+1,k)}$ on $\supp(\mu)$, then it will  have linear deviations with $\cL^u_{(j+1,k)}$ which leads to a contradiction. The proof   follows \cite[Proposition 2.2]{GKS}; we include it for completeness.

\begin{proof}[Proof of Lemma \ref{lem:j+1'}]
	Recall $\cL_\omega(H_\omega(x))=H_\omega(F^u_{(j+1,k),\omega})(x)).$
		Since periodic points of $\bar{A}$ are dense in $\supp(\mu)=\supp(\nu)^\ZZ\times\TT^d$ and since $\cL_\omega(x)$ is continuous in $\omega$ and $x$, we only need to show
	that $\cL_\omega(x)=\cL^u_{(j+1,k),\omega}(x)$ for every $(\omega,x)\in\Per(\bar{A})\cap\supp(\mu)$.
	
 Fix $(\omega,x)\in\Per(\bar{A})\cap\supp(H_*(\mu))$ with period $\kappa\in\NN$.
	Denote by $B$ the unit ball centered at $x$ in $\cL^u_{(j+1,k),\omega}(x)$.  If $B\subset\cL_\omega(x)$, then the invariance of $\cL$ by $\bar{A}$ implies $\cL_\omega(x)=\cL^u_{(j+1,k),\omega}(x)$.
	Suppose there is $z_1\in B$ with $z_1\notin\cL_\omega(x)$. Denote 
	$$
	x_1=\cL_\omega(x)\cap\cL^u_{j,\omega}(z_1).
	$$
	
	Since $\cL^u_{j,\omega}(x)$ has dense leaves in $\TT^d_\omega$, there exists 
	$$
	\{b_n:~n\geq0\}\subseteq \cL^u_{j,\omega}(x)
	\qquad \text{such that} \qquad
	b_n\to x_1~~\text{as}~~n\to\infty.
	$$
	Let 
	$$
y_n=\Hol^{\cL}_{x,x_1}(b_n).  
%	\qquad \text{then} \squad
%	y_n\to x_2\in\cL(x) ~~\text{as}~~n\to\infty.
	$$
By, Lemma \ref{lem:j+1-linear}, $y_n = x_1 + b_n$; since $x_1\in \cL(x)$, there is $x_2\in\cL(x) $ such that $y_n\to x_2$ {as} $n\to\infty.$
%	Moreover, Lemma \ref{lem:j+1-linear} implies $\{x_1,x_2\}$ is a parallel translation of $\{x,x_1\}$.
%By {\red {This isn't really what 6.19 says? + affineness?}}
%	
%By, Lemma \ref{lem:j+1-linear},l 
Moreover, $\{x_1,x_2\}$ is a parallel translation of $\{x,x_1\}$ inside $\cL(x) $.
%By {\red {This isn't really what 6.19 says? + affineness?}}

	We continue this procedure recursively and obtain $\{x_n:~n\geq 1\}\subseteq\cL_\omega(x)$. Let
	$$
	z_n=\cL^u_{j,\omega}(x_n)\cap\cL^u_{(j+1,k),\omega}(x).
	$$
Since all foliations are linear, we have
	\begin{align}\label{equ:linear}
		d_{\cL^u_{(j+1,k)}}(z_n,x)=n\cdot d_{\cL^u_{(j+1,k)}}(z_1,x)
		\qquad \text{and} \qquad
		d_{\cL^u_j}(x_n,z_n)=n\cdot d_{\cL^u_j}(x_1,z_1).
	\end{align}
	
	For every $n$, denote $N(n)$ the smallest integer such that $\bar{A}_\omega^{-\kappa\cdot N(n)}(z_n)\in B$. Since $\bar{A}^{-\kappa}$ contracts $\cL^u_{(j+1,k)}$ exponentially faster than $\cL^u_j$, equation (\ref{equ:linear}) implies
	$$
	d_{\cL^u_j}\left(\bar{A}_\omega^{-\kappa\cdot N(n)}(z_n),~\bar{A}_\omega^{-\kappa\cdot N(n)}(x_n)\right)
	\to\infty
	\qquad \text{as} \qquad
	n\to\infty.
	$$
	This contradicts to the compactness of $B$ where
	$$
	\max_{z\in B}d_{\cL^u_j}\left(z, \cL^u_{j,\omega}(z)\cap\cL_\omega(x)\right)<\infty.
	$$
	Thus we have $H(\cF^u_{(j+1,k)})=\cL=\cL^u_{(j+1,k)}$ in $\supp(\mu)$.
\end{proof}

\section{Non-randomness of stable bundles}

In this section, we show that if the topological conjugacy is $C^{1+}$-smooth along fiberwise unstable foliations, then stable bundles of perturbing systems are non-random.

Let $\cU=\cU_i$ be the neighborhood of $\cA$ in $\diff^2(\TT^d)$ and $\Omega=\cU^\ZZ$ as Proposition \ref{prop:structure-stable} such that the associated skew product 
$$
F:\Omega\times\TT^d\to\Omega\times\TT^d,
\qquad
F(\omega,x)=\left(\sigma(\omega),~f_\omega(x)\right)
$$  
is fiberwise Anosov. Moreover, for the extension system $\bar{A}:\Omega\times\TT^d\to\Omega\times\TT^d$ in Definition \ref{def:extension} %with
%$$
%\bar{A}(\omega,x)=\left(\sigma(\omega),A_{\omega'}(x)\right),
%\qquad \text{where} \qquad
%A_{\omega'}=A_i\in\cA\quad \text{if}\quad f_{\omega}=f_{\omega_0}\in\cU_i,
%$$
there exists a topological conjugacy
$H:\Omega\times\TT^d\to\Omega\times\TT^d$ with
$$
H(\omega,x)=\left(\omega, H_{\omega}(x)\right),
\qquad \text{such~that} \qquad
H\circ F=\bar{A}\circ H.
$$
If we denote $\cL^u$ and $\cF^u$ the fiberwise expanding foliations of $\bar{A}$ and $F$ on $\Omega\times\TT^d$ respectively, then $H(\cF^u)=\cL^u$.

\subsection{Non-randomness of stable bundles on $\TT^2$}

In this subsection, we consider the case that
$\cA=\{A_i\}_{i=1}^m\subseteq{GL}(2,\ZZ)$ is a  cone hyperbolic family of Anosov automorphisms.  Recall the two possibilities in Lemma \ref{claim:ZdenseorZ}.  
%, there are two possibilities:
%\begin{itemize}
%	\item either there exist $A_i$ and $A_j$ are non-commuting which implies the subgroup generated by $\{A_i\}_{i=1}^m$ is Zariski dense in ${GL}(2,\RR)$; 
%	\item or there exists some hyperbolic automorphism $A\in{GL}(2,\ZZ)$ such that $A_i=A^{n_i}$ with $n_i\in\NN$ for $i=1,\cdots,m$. \red{cite}
%\end{itemize}

Let $\cU$ be a sufficiently small neighborhood of $\cA$ in $\diff^2(\TT^2)$.   
Let $\nu$ be a probability supported on $\cU$.   % $\supp(\nu)\subseteq\cU$. %In Section \ref{sect:u-smooth-T2}, 
Recall from Proposition \ref{prop:2-H-smooth}, if the positive Lyapunov exponent of the unique $\nu$-stationary SRB measure $\mu$ of $F$ coincides with the positive exponent of the linear random walk (or the skew system $\bar{A}$), then the conjugacy $H$ is $C^{1+}$ along fiberwise unstable foliation $\cF^u$ over $\supp(\mu)=\supp(\nu)^\ZZ\times\TT^2$.

We fix some notation for the remainder. 
\begin{Notation}\
	\begin{itemize}
		\item Let $\cK\subseteq\cU=\cup\cU_i$ be a subset and $\Sigma=\cK^\ZZ$. Let $f\in\cK\cap\cU_i$ which is $C^2$ close to $A_i\in\cU_i$ thus Anosov. Denote $\cF^s_f$ and $\cF^u_f$ the stable and unstable foliations of $f$ respectively. Let $A_f= A_i$ if $f\in \cU_i$.  
		\item 
		   Denote $A_f=\Upsilon (f) $ which is the linearization of $f$ in ${GL}(2,\ZZ)$. Denote $\cL^s_f$ and $\cL^u_f$ the linear stable and unstable foliations of $A_f$ respectively.
		\item Let $H_f:\TT^2\to\TT^2$ be the topological conjugacy such that $H_f\circ f=A_f\circ H_f$ and $C^0$-close to identity $\id_{\TT^2}$, then
		$$
		H_f(\cF^s_f)=\cL^s_f
		\qquad \text{and} \qquad
		H_f(\cF^u_f)=\cL^u_f.
		$$
	\end{itemize}
\end{Notation}

\begin{Lemma}\label{lem:fixed-point}
	Let  $\omega=(\cdots,f,f,f,\cdots)\in\Sigma$ be a fixed point of the shift $\sigma:\sigma(\omega)=\omega$. Then we have
	$$
	\Upsilon(\omega)=(\cdots,A_f,A_f,A_f,\cdots)
	\qquad \text{and} \qquad 
	\bar{A}_{\omega}=A_f:\TT^2\to\TT^2.
	$$
	Moreover, for the topological conjugacy $H_\omega:\TT^2_\omega\to\TT^2_\omega$ in Proposition \ref{prop:structure-stable}, we have %satisfying $H_{\sigma(\omega)}\circ F_\omega=\bar{A}_\omega\circ H_\omega$, it satisfies 
	$$
	H_\omega=H_f.
	$$
\end{Lemma}

\begin{proof}
	This lemma is a direct consequence derived from the definitions of $\bar{A}$ and topological conjugacy $H$.
\end{proof}

We will take $\cK=\supp(\nu)$. In this subsection, all our results are independent of $\nu$. It works for any subset $\cK\subseteq\cU$.  
The following proposition is the main theorem in this subsection. %which implies Theorem \ref{thm:2-dim} and \ref{thm:non-commuting}.

\begin{Proposition}\label{prop:s-nonrandom-T2}
	Assume the topological conjugacy $H:\Sigma\times\TT^2\to\Sigma\times\TT^2$ is $C^{1+}$-smooth along fiberwise unstable foliations with bounded derivatives. For every pair $f,g\in\cK$, let $H_f,H_g:\TT^2\to\TT^2$ be the corresponding topoplogical conjugacies with
	$H_f\circ f=A_f\circ f$ and 
	$H_g\circ g=A_g\circ g$. 
	Then they satisfy
	$$
	H_f(\cF^s_g)=\cL^s_g
	\qquad \text{and} \qquad
	H_g(\cF^s_f)=\cL^s_f.
	$$
	Here $\cF^s_f$ and $\cF^s_g$ are stable foliations of $f$ and $g$ respectively, $\cL^s_f$ and $\cL^s_g$ are linear stable foliations of $A_f$ and $A_g$ respectively.
\end{Proposition}

\begin{proof}
	We prove $H_f(\cF^s_g)=\cL^s_g$ as the proof that $H_g(\cF^s_f)=\cL^s_f$ is  symmetric.  
	Since $H_g(\cF^s_g)=\cL^s_g$ is a irrational linear foliation on $\TT^2$ and since both $H_g,H_f$ are $C^0$-close to $\id_{\TT^2}$, by lifting all these foliations in the universal cover $\RR^2$, we have
	\begin{itemize}
		\item each leaf of $\cF^s_g$ is contained in a  {bounded} neighborhood of a leaf of $\cL^s_g$ in $\RR^2$;
		\item each leaf of $H_f(\cF^s_g)$ is contained in a  {bounded} neighborhood of a leaf of $\cF^s_g$ in $\RR^2$.
	\end{itemize}
	This implies $H_f(\cF^s_g)$ is contained in a  {bounded}  neighborhood of a leaf of $\cL^s_g$ in $\RR^2$.
	It suffices to  prove that $H_f(\cF^s_g)$ is a linear foliation. Indeed, if $H_f(\cF^s_g)$ is a linear foliation with each leaf bounded distance from a leaf of $\cL^s_g$, then $H_f(\cF^s_g)$  and $\cL^s_g$ coincide. 
	
%	 is 
%	
%	since \blue(by  \red{explain} it then follows that $H_f(\cF^s_g)=\cL^s_g$.
	
	\vskip3mm
	
	Denote $\omega=(\cdots,f,f,f,\cdots)$, $\xi=(\cdots,g,g,g,\cdots)$, and {$\tau=[\omega,\xi]=(\cdots,f,g,g,\cdots)$. } That is, $\tau_i=f,$ for all $ i<0$ and 	$\tau_j=g,$  for all $ j\geq0.$
%
%	$$
%	\tau=(\cdots,f,g,g,\cdots), 
%	\qquad \text{where }\quad
%	\tau_i=f, ~~\forall i<0; 
%	\quad \text{and} \quad
%	\tau_j=g, ~~\forall j\geq0.
%	$$
	Since $f,g\in\cK$, we have $\omega,\xi,\tau\in\Sigma=\cK^\ZZ$, and $\tau$ is a heteroclinic orbit from $\omega$ to $\tau$ for the shift map $\sigma:\Sigma\to\Sigma$.
	
	For the fiberwise dynamics of the skew product  $F^n(\tau,x)=\left(\sigma^n(\tau),f^n_{\tau}(x)\right)$, we have
	\begin{align*}
	 	f^n_{\tau}(x)&=g^n(x):
		  ~\TT^2_{\tau}\to\TT^2_{\sigma^n(\tau)}, \qquad \forall n\geq 0;\\
		f^n_\tau(x)&=f^n(x):
		  ~\TT^2_{\tau}\to\TT^2_{\sigma^n(\tau)}, \qquad \forall n<0.
	\end{align*}
	This implies the fiberwise stable and unstable foliations of $F$ on the fiber $\TT^2_{\tau}$ are
	$$
	\cF^s_{\tau}=\cF^s_g
	\qquad \text{and} \qquad
	\cF^u_{\tau}=\cF^u_f.
	$$
	
	Similarly, the fiberwise dynamics of skew product system 
	$\bar{A}^n(\tau,x)=\left(\sigma^n(\tau),\bar{A}^n_{\tau}(x)\right)$ are
	\begin{align*}
		\bar{A}^n_{\tau}(x)&=A_g^n(x):
		~\TT^2_{\tau}\to\TT^2_{\sigma^n(\tau)}, \qquad \forall n\geq 0;\\
		\bar{A}^n_\tau(x)&=A_f^n(x):
		~\TT^2_{\tau}\to\TT^2_{\sigma^n(\tau)}, \qquad \forall n<0.
	\end{align*}
	 The fiberwise stable and unstable  linear foliations of $\bar{A}$ on the fiber $\TT^2_{\tau}$ are
	$$
	\cL^s_{\tau}=\cL^s_g
	\qquad \text{and} \qquad
	\cL^u_{\tau}=\cL^u_f.
	$$
	
	Moreover, the fiberwise topological conjugacy $H_{\tau}:\TT^2_{\tau}\to\TT^2_{\tau}$ satisfies
	$$
	H_{\tau}(\cF^s_g)=H_{\tau}(\cF^s_{\tau})=\cL^s_{\tau}=\cL^s_g
	\qquad \text{and} \qquad
	H_{\tau}(\cF^u_f)=H_{\tau}(\cF^u_{\tau})=\cL^u_{\tau}=\cL^u_f.
	$$
	Since we have $H_{\omega}=H_f:\TT^2\to\TT^2$ satisfies 
	$$
	H_{\omega}\circ f=H_{\omega}\circ f_{\omega}=A_f\circ H_{\omega},
	$$
	we only need to show $H_f(\cF^s_g)=H_{\omega}(\cF^s_{\tau})=H_f(\cF^s_{\tau})$ is a linear foliation, which will follow from the following claim.  	
%	The following claim implies $H_{\omega}(\cF^s_{\tau})$ is a linear foliation.
	
	\begin{Claim}
		 Let $\cF=H_{\omega}(\cF^s_{\tau})$.  For every $x\in\TT^2$ and every $y\in\cF(x)$, the holonomy map of $\cF$ between leaves of $H_{\omega}(\cF^u_{\tau})=H_{\omega}(\cF^u_f)=\cL^u_f$:
		 $$
		 \Hol^{\cF}_{x,y}:~\cL^u_f(x)\to\cL^u_f(y)
		 $$ 
		 is $C^{1+}$-smooth and the derivative satisfies $D\Hol^{\cF}_{x,y}=\id$. 
	\end{Claim}
	
	\begin{proof}[Proof of Claim]
		For every $x\in\TT^2$, $y\in\cF(x)$ and $n>0$, we denote 
		$$
		x_n=\bar{A}_{\tau}^{-n}(x)=A_f^{-n}(x)\in\TT^2_{\sigma^{-n}(\tau)}
		\qquad \text{and} \qquad
		y_n=\bar{A}_{\tau}^{-n}(y)=A_f^{-n}(y)\in\TT^2_{\sigma^{-n}(\tau)}.
		$$
		Denote $\cF_n=\bar{A}_{\tau}^{-n}(\cF)=A_f^{-n}(\cF)$, then $y_n\in\cF_n(x)$.
		
		Since $\bar{A}^{-n}_{\tau}(\cL^u_f)=A_f^{-n}(\cL^u_f)=\cL^u_f$, denoting by 
		$\Hol^{\cF_n}_{x_n,y_n}:\cL^u_f(x_n)\to\cL^u_f(y_n)$ the holonomy map induced by $\cF_n$ between $\cL^u_f(x_n)$ and $\cL^u_f(y_n)$, we have
		\begin{align}\label{eq:Hol}
			\Hol^{\cF}_{x,y}=
			\bar{A}_{\sigma^n(\tau)}\circ
			\Hol^{\cF_n}_{x_n,y_n} \circ
			\bar{A}_{\tau}^{-n} 
			=A_f^n\circ \Hol^{\cF_n}_{x_n,y_n}\circ A_f^{-n}.
		\end{align}
		From properties of the topological conjugacy, we have
		\begin{align*}
			\cF_n=\bar{A}^{-n}_{\tau}(\cF)
			&=\bar{A}^{-n}_{\tau}\circ H_{\omega}(\cF^s_{\tau})
			=A_f^{-n}\circ H_{\omega}(\cF^s_{\tau})\\
			&=H_{\omega}\circ f^{-n}(\cF^s_{\tau}) 
			=H_{\omega}\circ f_{\tau}^{-n}(\cF^s_{\tau}) \\
			&=H_{\omega}(\cF^s_{\sigma^{-n}(\tau)}).
		\end{align*}

		Since $\cF^u_{\sigma^{-n}(\tau)}=\cF^u_{\tau}=\cF^u_f$, both conjugacies $H_{\omega}$ and $H_{\sigma^{-n}(\tau)}$ satisfy 
		$$
		\cL^u_f=H_{\omega}(\cF^u_{\sigma^{-n}(\tau)})=H_{\sigma^{-n}(\tau)}(\cF^u_{\sigma^{-n}(\tau)}).
		$$
		Denote
		\begin{itemize}
			\item $x_n'=H_{\omega}^{-1}(x_n)$ and $y_n'=H_{\omega}^{-1}(y_n)\in\cF^s_{\sigma^{-n}(\tau)}(x_n')$;
			\item $\Hol^{\cF^s_{\sigma^{-n}(\tau)}}_{x_n',y_n'}:\cF^u_f(x_n')\to\cF^u_f(y_n')$ the holonomy map induced by $\cF^s_{\sigma^{-n}(\tau)}=H_{\omega}^{-1}(\cF_n)$;
			\item $x_n''=H_{\sigma^{-n}(\tau)}(x_n')$ and $y_n''=H_{\sigma^{-n}(\tau)}(y_n')\in H_{\sigma^{-n}(\tau)}(\cF^s_{\sigma^{-n}(\tau)}(x_n'))=\cL^s_{\sigma^{-n}(\tau)}(x_n'')$;
			\item $\Hol^{\cL^s_{\sigma^{-n}(\tau)}}_{x_n'',y_n''}:\cL^u_f(x_n'')\to\cL^u_f(y_n'')$ the holonomy map induced by $\cL^s_{\sigma^{-n}(\tau)}$ which is a parallel translation between $\cL^u_f(x_n'')$ and $\cL^u_f(y_n'')$.
		\end{itemize}
		
		The holonomy map
		$\Hol^{\cF_n}_{x_n,y_n}:\cL^u_f(x_n)\to\cL^u_f(y_n)$ satisfies
		\begin{align}
			\Hol^{\cF_n}_{x_n,y_n} &=
			\left(H_{\omega}|_{\cF^u_f(y_n')}\right)
			\circ\left(\Hol^{\cF^s_{\sigma^{-n}(\tau)}}_{x_n',y_n'}\right)\circ 
			\left(H_{\omega}^{-1}|_{\cL^u_f(x_n)}\right)  \notag \\
			&=
			\left(H_{\omega}|_{\cF^u_f(y_n')}\right)
			\circ\left(H_{\sigma^{-n}(\tau)}^{-1}|_{\cL^u_f(y_n'')} \right)
			\circ\left(\Hol^{\cL^s_{\sigma^{-n}(\tau)}}_{x_n'',y_n''}\right)
			\circ\left(H_{\sigma^{-n}(\tau)}|_{\cF^u_f(x_n')} \right)
			\circ\left(H_{\omega}^{-1}|_{\cL^u_f(x_n)}\right).  \label{eq:convergence}
		\end{align}
		Here $\Hol^{\cL^s_{\sigma^{-n}(\tau)}}_{x_n'',y_n''}$ is a parallel translation whose derivative is the identity $\id$.
		
		Since both conjugacies $H_{\omega}$ and $H_{\sigma^{-n}(\tau)}$ are $C^{1+}$-smooth along $\cF^u_{\sigma^{-n}(\tau)}=\cF^u_f$ with H\"older continuous and bounded derivatives, the holonomy map $\Hol^{\cF_n}_{x_n,y_n}$ is differentiable. 
		
		Moreover, since $\sigma^{-n}(\tau)\to\omega$ as $n\to+\infty$, the continuity of $DH|_{\cF^u}$ implies
	    \begin{align*}
	    	\left(DH_{\sigma^{-n}(\tau)}|_{\cF^u_f(x_n')} \right)
	    	\circ\left(DH_{\omega}^{-1}|_{\cL^u_f(x_n)}\right)
	    	=\id+\Delta_n \to \id 
	    	\qquad \text{as}\quad n\to+\infty; \\
	    	\left(DH_{\omega}|_{\cF^u_f(y_n')}\right)
	    	\circ\left(DH_{\sigma^{-n}(\tau)}^{-1}|_{\cL^u_f(y_n'')} \right)
	    	=\id+\Delta_n' \to \id 
	    	\qquad \text{as}\quad n\to+\infty.
	    \end{align*}
{Above, $\Delta_n\to 0$ uniformly (over $x$) as a linear operator from $\cL^u_f(x_n)$ to $\cL^u_f(x_n'')$ (viewed as linear spaces with origins at $x_n$ and $x_n''$, respectively)}. 
	   From Equation \ref{eq:Hol} and \ref{eq:convergence}, the holonomy map $\Hol^{\cF}_{x,y}:\cL^u_f(x_n)\to\cL^u_f(y_n)$ is differentiable and satisfies
	   \begin{align*}
	   	  D\Hol^{\cF}_{x,y}
	   	      &=A_f^n\circ\left(D\Hol^{\cF_n}_{x_n,y_n}\right)\circ A_f^{-n} \\
	   	      &=A_f^n\circ(\id+\Delta_n')\circ(\id+\Delta_n)\circ A_f^{-n} .%\qquad (n\to+\infty)\\
%	   	      &=\id.
	   \end{align*}
	  {Since the restriction of $A_f^n$ to leaves of $ \cL^u_f$ is conformal with derivative independent of base point, we have $$	   	      A_f^n\circ(\id+\Delta_n')\circ(\id+\Delta_n)\circ A_f^{-n} \to \id$$ as $n\to \infty$.  }
	    This proves the claim.
	\end{proof}
	
    Since $H_{\omega}(\cF^u_{\tau})=H_{\omega}(\cF^u_f)=\cL^u_f$ is a minimal linear foliation on $\TT^2$ and transverse to $\cF$,  Lemma \ref{lem:linear-Td} shows that $\cF$ is a linear foliation. Thus we have 
    $$
    H_f(\cF^s_g)=H_{\omega}(\cF^s_{\tau})=\cF=\cL^s_g.
    $$
    Symmetrically, we have $H_g(\cF^s_f)=\cL^s_f$. 
    This finishes the proof of the proposition.
\end{proof}

\subsection{Non-randomness of stable bundles on $\TT^d$}

In this subsection, we consider the case that $\cA=\{A_i\}_{i=1}^m\subseteq{GL}(d,\ZZ)$ be a family of commuting Anosov automorphisms admitting the same finest dominated splitting
$$
T\TT^d=L^s_1\oplus\cdots\oplus L^s_l\oplus L^u_1\oplus\cdots\oplus L^u_k.
$$ 
Moreover, each $A_i$ satisfies the generic assumption, i.e. $A_i$ is hyperbolic, irreducible and $A_i$ is conformal in every $L^{s/u}_j$ with $\dim L^{s/u}_j\leq 2$. Since all $\{A_i\}_{i=1}^m$ have the same invariant foliations, we denote $\cL^s,\cL^u,\cL^u_j$ the corresponding linear foliations tangent to $L^s,L^u,L^u_j$ for $j=1,\cdots,k$ respectively.

Let $\cU_i$ be a neighborhood of $A_i$ in $\diff^2(\TT^d)$, denote $\cU=\cup\cU_i$ and $\Omega=\cU^\ZZ$, such that the corresponding skew product
$$
F:\Omega\times\TT^d\to\Omega\times\TT^d, \qquad
F(\omega,x)=\left(\sigma(\omega),f_\omega(x)\right)
$$
is fiberwise Anosov, and satisfies the conclusion of Proposition \ref{prop:structure-stable} and Proposition \ref{prop:leaf-conjugacy}, then $F$ admits the fiberwise finest dominated splitting 
$$
T\TT^d_\omega=
E^s_{1,\omega}\oplus\cdots E^s_{l,\omega}\oplus E^u_{1,\omega}\oplus\cdots E^u_{k,\omega}, 
\qquad \text{with} \qquad
{\rm dim}E^{s/u}_{j,\omega}={\rm dim}L^{s/u}_j,\quad \forall j.
$$

Let $H:\Omega\times\TT^d\to\Omega\times\TT^d$ be the topological conjugacy defined in  Proposition \ref{prop:structure-stable} and satisfies Proposition \ref{prop:leaf-conjugacy}, which satisfies $H\circ F=\bar{A}\circ H$ where $\bar{A}$ is the extension skew product defined in Definition \ref{def:extension}.

Let $\nu$ be a probability with $\supp(\nu)\subseteq\cU$ where the corresponding fiberwise SRB-measure of $F$ with $\pi_*(\mu)=\nu^\ZZ$ admits positive Lyapunov exponents rigidity, i.e. $\nu$ satisfies the assumption of Proposition \ref{prop:d-H-smooth}.

Let $\cK=\supp(\nu)$ and $\Sigma=\cK^\ZZ$. For the restriction skew product systems $F:\Sigma\times\TT^d\to\Sigma\times\TT^d$ and $\bar{A}:\Sigma\times\TT^d\to\Sigma\times\TT^d$, Proposition \ref{prop:d-H-smooth} shows that the corresponding conjugacy $H:\Sigma\times\TT^d\to\Sigma\times\TT^d$ satisfies
\begin{enumerate}
	\item $H$ intertwines all intermediate unstable  foliations: 
	   $$
	   H_{\omega}\left(\cF^u_{j,\omega}\right)=\cL^u_{j,\omega},
	   \qquad 
	   \forall j=1,\cdots,k,\quad \forall \omega\in\Sigma.
	   $$
	   In particular, the bundle $E^s_{\omega}\oplus E^u_{k,\omega}$ is jointly integrable to a foliation $\cF^s_{\omega}\oplus\cF^u_{k,\omega}$.
	\item $H$ is $C^{1+}$-smooth along leaves of $\cF^u_j$ with bounded derivatives for $j=1,\cdots,k$ on $\Sigma\times\TT^d$.  
\end{enumerate}

For every $f\in\cK\cap\cU_i$, denote $\cF^s_f, \cF^u_f$ the stable, unstable foliations of $f$, and $\cF^u_{j,f}$ the $f$-invariant foliations tangent to $E^u_{j,f}$ for $j=1,\cdots,k$.
Denote $A_f=A_i\in{GL}(d,\ZZ)$ be the linearization of $f$. 
Let $H_f:\TT^d\to\TT^d$ be the topological conjugacy $H_f\circ f=A_f\circ H_f$, then Proposition \ref{prop:d-H-smooth} shows that $H_f(\cF^u_{j,f})=\cL^u_j$ and $H_f$ is $C^{1+}$-smooth along $\cF^u_j$ for $j=1,\cdots,k$. 

Denote $\omega=(\cdots,f,f,f,\cdots)\in\Sigma$, then we have 
$$
\bar{A}_{\omega}=A_f:~\TT^d_{\omega}\to\TT^d_{\sigma(\omega)}=\TT^d_{\omega},
\qquad \text{and} \qquad
H_{\omega}=H_f:~\TT^d_{\omega}\to\TT^d_{\omega}.
$$
This following proposition is the main theorem in this subsection, which is a high dimensional version of Proposition \ref{prop:s-nonrandom-T2}.

\begin{Proposition}\label{prop:s-nonrandom-Td} %\note{do we really need to write the exact same proof twice?}
	Assume the topological conjugacy $H:\Sigma\times\TT^d\to\Sigma\times\TT^d$ intertwines the strongest expanding foliation $H(\cF^u_k)=\cL^u_k$ and  is $C^{1+}$-smooth along $\cF^u_k$ with bounded derivatives. For every pair $f,g\in\cK$, let $H_f,H_g:\TT^d\to\TT^d$ be the corresponding topological conjugacies.  
	Then we have 
	$$
	H_f(\cF^s_g)=\cL^s
	\qquad \text{and} \qquad
	H_g(\cF^s_f)=\cL^s.
	$$
\end{Proposition}
The proof is similar to Proposition \ref{prop:s-nonrandom-T2}.  We include it for completeness.  
\begin{proof}
We only need to show that $H_f(\cF^s_g)$ is a linear foliation.  Then, as in the proof of Proposition \ref{prop:s-nonrandom-T2}, we have $H_g(\cF^s_g)=\cL^s$. %%%%\red{explain as above}
	
	For $f,g\in\cK$, denote $\omega=(\cdots,f,f,f,\cdots)$, $\xi=(\cdots,g,g,g,\cdots)$, and
	$$
	\tau=(\cdots,f,g,g,\cdots), 
	\qquad \text{that~is}\quad
	\tau_i=f, ~~\forall i<0; 
	\quad \text{and} \quad
	\tau_j=g, ~~\forall j\geq0.
	$$
	Then $\omega,\xi,\tau\in\Sigma$ and $\lim_{n\to+\infty}\sigma^{-n}(\tau)=\omega$.
	
	For the fiberwise dynamics of skew product system $F^n(\tau,x)=\left(\sigma^n(\tau),f^n_{\tau}(x)\right)$, we have
	\begin{align*}
		f^n_{\tau}(x)&=g^n(x):
		~\TT^d_{\tau}\to\TT^d_{\sigma^n(\tau)}, \qquad \forall n\geq 0;\\
		f^n_\tau(x)&=f^n(x):
		~\TT^d_{\tau}\to\TT^d_{\sigma^n(\tau)}, \qquad \forall n<0.
	\end{align*}
	This implies the fiberwise invariant foliations of $F$ on the fiber $\TT^d_{\tau}$ are
	$$
	\cF^s_{\tau}=\cF^s_g
	\qquad \text{and} \qquad
	\cF^u_{\tau}=\cF^u_f.
	$$
	In particular, the strongest expanding foliation $\cF^u_{k,\tau}$ depends on negative iterations, thus we have 
	$$
	\cF^u_{k,\tau}=\cF^u_{k,f}=\cF^u_{k,\omega}.
	$$
	
    For the fiberwise dynamics of skew product system 
	$\bar{A}^n(\tau,x)=\left(\sigma^n(\tau),\bar{A}^n_{\tau}(x)\right)$ are
	\begin{align*}
		\bar{A}^n_{\tau}(x)&=A_g^n(x):
		~\TT^d_{\tau}\to\TT^d_{\sigma^n(\tau)}, \qquad \forall n\geq 0;\\
		\bar{A}^n_\tau(x)&=A_f^n(x):
		~\TT^d_{\tau}\to\TT^d_{\sigma^n(\tau)}, \qquad \forall n<0.
	\end{align*}
	The corresponding linear foliations on $\TT^d_{\tau}$ are 
	$$
	\cL^s_{\tau}=\cL^s, \qquad
	\cL^u_{\tau}=\cL^u, \qquad
	\text{and} \qquad
	\cL^u_{k,\tau}=\cL^u_k.
	$$
	
	Moreover, the fiberwise topological conjugacy $H_{\tau}:\TT^d_{\tau}\to\TT^d_{\tau}$ satisfies
	$$
	H_{\tau}(\cF^s_g)=H_{\tau}(\cF^s_{\tau})=\cL^s
	\qquad \text{and} \qquad
	H_{\tau}(\cF^u_f)=H_{\tau}(\cF^u_{\tau})=\cL^u.
	$$
	Since we have $H_{\omega}=H_f:\TT^2\to\TT^2$ satisfies 
	$$
	H_{\omega}\circ f=H_{\omega}\circ f_{\omega}=A_f\circ H_{\omega},
	$$
	we only need to show $H_{\omega}(\cF^s_{\tau})=H_f(\cF^s_{\tau})$ is a linear foliation. 
	
	Since $H_{\tau}(\cF^u_{k,\tau})=\cL^u_k$ which is jointly integrable with $H_{\tau}(\cF^s_{\tau})=\cL^s$ and $\cF^u_{k,\tau}=\cF^u_{k,\omega}$, we have
	\begin{itemize}
		\item the foliation $\cF^u_{k,\tau}$ is jointly integrable with $\cF^s_{\tau}$;
		\item $H_{\omega}(\cF^u_{k,\tau})=\cL^u_k$.
	\end{itemize} 
	Denote $\cF=H_{\omega}(\cF^s_\tau)$ which is jointly integrable with $\cL^u_k$. Since $A_i$ is irreducible, we have that  $\cL^u_k$ is minimal and Lemma \ref{lem:joint-minimal} shows that $\cF\oplus\cL^u_k$ is a minimal linear foliation on $\TT^d$. Since both $H_{\omega}$ and $H_{\tau}$ are homotopic to $\id_{\TT^d}$, the fact that 
	$$
	H_{\tau}\circ H_{\omega}^{-1}\left(\cF\oplus\cL^u_k\right)
	=H_{\tau}\left(\cF^s_{\tau}\oplus\cF^u_{k,\tau}\right)
	=\cL^s\oplus\cL^u_k
	$$
	implies 
	$$
	\cF\oplus\cL^u_k=\cL^s\oplus\cL^u_k.
	$$
	
	The following claim implies $\cF$ is a linear foliation.
	
	\begin{Claim}
		For every $x\in\TT^2$ and every $y\in\cF(x)$, the holonomy map of $\cF$ between leaves of $\cL^u_k$ inside the leaf $\cF\oplus\cL^u_k(x)$:
		$$
		\Hol^{\cF}_{x,y}:~\cL^u_k(x)\to\cL^u_k(y)
		$$ 
		is $C^{1+}$-smooth and the derivative $D\Hol^{\cF}_{x,y}=\id$. 
	\end{Claim}
	
	\begin{proof}[Proof of Claim]
		For every $x\in\TT^d$, $y\in\cF(x)$ and $n>0$, we denote 
		$$
		x_n=\bar{A}_{\tau}^{-n}(x)=A_f^{-n}(x)\in\TT^d_{\sigma^{-n}(\tau)}
		\qquad \text{and} \qquad
		y_n=\bar{A}_{\tau}^{-n}(y)=A_f^{-n}(y)\in\TT^d_{\sigma^{-n}(\tau)}.
		$$
		Denote $\cF_n=\bar{A}_{\tau}^{-n}(\cF)=A_f^{-n}(\cF)$, then $y_n\in\cF_n(x)$.
		
		Since $\bar{A}^{-n}_{\tau}(\cL^u_k)=A_f^{-n}(\cL^u_k)=\cL^u_k$, denote 
		$\Hol^{\cF_n}_{x_n,y_n}:\cL^u_k(x_n)\to\cL^u_k(y_n)$ the holonomy map induced by $\cF_n$ between $\cL^u_k(x_n)$ and $\cL^u_k(y_n)$ inside the leaf $\cF\oplus\cL^u_k(x_n)$, we have
		\begin{align}\label{eq:Hol-d}
			\Hol^{\cF}_{x,y}=
			\bar{A}_{\sigma^n(\tau)}\circ
			\Hol^{\cF_n}_{x_n,y_n} \circ
			\bar{A}_{\tau}^{-n} 
			=A_f^n\circ \Hol^{\cF_n}_{x_n,y_n}\circ A_f^{-n}.
		\end{align}
		In particular, $A_f$ is conformal along $\cL^u_k$.
		From the topological conjugacy, we have
		\begin{align*}
			\cF_n=\bar{A}^{-n}_{\tau}(\cF)
			&=\bar{A}^{-n}_{\tau}\circ H_{\omega}(\cF^s_{\tau})
			=A_f^{-n}\circ H_{\omega}(\cF^s_{\tau})\\
			&=H_{\omega}\circ f^{-n}(\cF^s_{\tau}) 
			=H_{\omega}\circ f_{\tau}^{-n}(\cF^s_{\tau}) \\
			&=H_{\omega}(\cF^s_{\sigma^{-n}(\tau)}).
		\end{align*}

		Since $\cF^u_{k,\sigma^{-n}(\tau)}=\cF^u_{k,\tau}=\cF^u_{k,f}$, both conjugacies $H_{\omega}$ and $H_{\sigma^{-n}(\tau)}$ satisfy 
		$$
		\cL^u_k=H_{\omega}(\cF^u_{k,\sigma^{-n}(\tau)})=H_{\sigma^{-n}(\tau)}(\cF^u_{k,\sigma^{-n}(\tau)}).
		$$
		Denote
		\begin{itemize}
			\item $x_n'=H_{\omega}^{-1}(x_n)$ and $y_n'=H_{\omega}^{-1}(y_n)\in\cF^s_{\sigma^{-n}(\tau)}(x_n')$;
			\item $\Hol^{\cF^s_{\sigma^{-n}(\tau)}}_{x_n',y_n'}:\cF^u_{k,f}(x_n')\to\cF^u_{k,f}(y_n')$ the holonomy map induced by $\cF^s_{\sigma^{-n}(\tau)}=H_{\omega}^{-1}(\cF_n)$ inside the leaf $\cF^s_{\sigma^{-n}(\tau)}\oplus\cF^u_{k,\sigma^{-n}(\tau)}(x_n')$;
			\item $x_n''=H_{\sigma^{-n}(\tau)}(x_n')$ and $y_n''=H_{\sigma^{-n}(\tau)}(y_n')\in H_{\sigma^{-n}(\tau)}(\cF^s_{\sigma^{-n}(\tau)}(x_n'))=\cL^s_{\sigma^{-n}(\tau)}(x_n'')$;
			\item $\Hol^{\cL^s_{\sigma^{-n}(\tau)}}_{x_n'',y_n''}:\cL^u_k(x_n'')\to\cL^u_k(y_n'')$ the holonomy map induced by $\cL^s_{\sigma^{-n}(\tau)}$ inside the leaf $\cL^s\oplus\cL^u_k(x_n'')$, which is a parallel translation between $\cL^u_k(x_n'')$ and $\cL^u_k(y_n'')$.
		\end{itemize}
		
		The holonomy map
		$\Hol^{\cF_n}_{x_n,y_n}:\cL^u_k(x_n)\to\cL^u_k(y_n)$ satisfies
		\begin{align}
			\Hol^{\cF_n}_{x_n,y_n} &=
			\left(H_{\omega}|_{\cF^u_{k,f}(y_n')}\right)
			\circ\left(\Hol^{\cF^s_{\sigma^{-n}(\tau)}}_{x_n',y_n'}\right)\circ 
			\left(H_{\omega}^{-1}|_{\cL^u_k(x_n)}\right)  \notag \\
			&=
			\left(H_{\omega}|_{\cF^u_{k,f}(y_n')}\right)
			\circ\left(H_{\sigma^{-n}(\tau)}^{-1}|_{\cL^u_k(y_n'')} \right)
			\circ\left(\Hol^{\cL^s_{\sigma^{-n}(\tau)}}_{x_n'',y_n''}\right)
			\circ\left(H_{\sigma^{-n}(\tau)}|_{\cF^u_{k,f}(x_n')} \right)
			\circ\left(H_{\omega}^{-1}|_{\cL^u_k(x_n)}\right)  \label{eq:convergence-d}
		\end{align}
		Here $\Hol^{\cL^s_{\sigma^{-n}(\tau)}}_{x_n'',y_n''}$ is a parallel translation where the derivative is identity $\id$.
		
		Since both conjugacies $H_{\omega}$ and $H_{\sigma^{-n}(\tau)}$ are $C^{1+}$-smooth along $\cF^u_{k,\sigma^{-n}(\tau)}=\cF^u_{k,f}$ with H\"older continuous and bounded derivatives, the holonomy map $\Hol^{\cF_n}_{x_n,y_n}$ is differentiable. 
		
		Moreover, since $\sigma^{-n}(\tau)\to\omega$ as $n\to+\infty$, the continuity of $DH|_{\cF^u_k}$ implies
		\begin{align*}
			\left(DH_{\sigma^{-n}(\tau)}|_{\cF^u_{k,f}(x_n')} \right)
			\circ\left(DH_{\omega}^{-1}|_{\cL^u_k(x_n)}\right)
			=\id+\Delta_n \to \id 
			\qquad \text{as}\quad n\to+\infty; \\
			\left(DH_{\omega}|_{\cF^u_{k,f}(y_n')}\right)
			\circ\left(DH_{\sigma^{-n}(\tau)}^{-1}|_{\cL^u_k(y_n'')} \right)
			=\id+\Delta_n' \to \id 
			\qquad \text{as}\quad n\to+\infty.
		\end{align*}
		From Equation \ref{eq:Hol-d} and \ref{eq:convergence-d}, the holonomy map $\Hol^{\cF}_{x,y}:\cL^u_f(x_n)\to\cL^u_f(y_n)$ is differentiable and satisfies
		\begin{align*}
			D\Hol^{\cF}_{x,y}
			&=A_f^n\circ\left(D\Hol^{\cF_n}_{x_n,y_n}\right)\circ A_f^{-n} \\
			&=A_f^n\circ(\id+\Delta_n')\circ(\id+\Delta_n)\circ A_f^{-n} \qquad (n\to+\infty)\\
			&=\id.
		\end{align*}
		The last equality use the fact that $A_f$ is conformal along $\cL^u_k$.
		This proves the claim.
	\end{proof}
	
	Since $H_{\omega}(\cF^u_{k,\tau})=H_{\omega}(\cF^u_{k,f})=\cL^u_k$ is a minimal linear foliation on $\TT^d$ and transverse to $\cF$,  Lemma \ref{lem:linear-Td} shows that $\cF$ is a linear foliation. Thus we have 
	$$
	H_f(\cF^s_g)=H_{\omega}(\cF^s_{\tau})=\cF=\cL^s.
	$$
	Symmetrically, we have $H_g(\cF^s_f)=\cL^s$. 
	This finishes the proof of the proposition.
\end{proof}

\section{Simultaneously linearization of Anosov actions}

In this section, we show that non-randomness of stable and unstable foliations of random Anosov diffeomorphisms implies simultaneously linearization and prove main theorems of this paper.

We begin with the following elementary lemma.

\begin{Lemma}\label{lem:translation}
	Let $\cL_1$ and $\cL_2$ be two minimal transversal linear foliations on $\TT^d$ where
	$$
	\dim\cL_1+\dim\cL_2=d.
	$$ 
	If a homeomorphism  $H:\TT^d\to\TT^d$ is homotopic to $\id_{\TT^d}$ and preserves $\cL_1,\cL_2$, then $H$ is a translation, i.e. there exists $v\in\TT^d$ such that 
	$$
	H(x)=x+v, \qquad \forall x\in\TT^d.
	$$
\end{Lemma}

\begin{proof}
	Denote by $\tilde{H}:\RR^d\to\RR^d$ a lift of $H$.  We only need to show that $\tilde{H}$ is a translation on $\RR^d$. Since $H$ is homotopic to $\id_{\TT^d}$, we have
	$$
	\tilde{H}(x+n)=\tilde{H}(x)+n,
	\qquad \forall x\in\RR^d,~\forall n\in\ZZ^d.
	$$
	Assume $\tilde{H}(0)=x_0$, then $\tilde{H}(n)=x_0+n$ for every $n\in\ZZ^d$. Denote by $\tilde{\cL}_1$ and $\tilde{\cL}_2$ lifts of the linear foliations  $\cL_1$ and $\cL_2$ to $\RR^d$, respectively. 
	
	Since $H(\cL_i)=\cL_i$ for $i=1,2$, we also have
	$$
	\tilde{H}(\tilde{\cL}_i)=\tilde{\cL}_i,
	\qquad i=1,2.
	$$
	Moreover, $\tilde{\cL}_1$ and $\tilde{\cL}_2$ are two linear transversal foliations satisfiying
	$$
	\dim\tilde\cL_1+\dim\tilde\cL_2=d.
	$$
	This implies $\tilde{\cL}_1$ and $\tilde{\cL}_2$ admit global product structure: for every pair of points $x,y\in\RR^d$, the intersection of  $\tilde{\cL}_1(x)$  with $\tilde{\cL}_2(y)$ is a singleton. 
	
	Now for every pair of points $m,n\in\ZZ^d$, we have 
	$$
	\tilde{H}(m)=x_0+m \qquad \text{and} \qquad \tilde{H}(n)=x_0+n.
	$$
	This implies
	$$
	\tilde{H}\left(\tilde{\cL}_1(m)\right)=\tilde{\cL}_1(m+x_0)=\tilde{\cL}_1(m)+x_0,
	\qquad  
	\tilde{H}\left(\tilde{\cL}_2(n)\right)=\tilde{\cL}_2(n+x_0)=\tilde{\cL}_2(n)+x_0.
	$$
	Thus for the intersecting point $z_{m,n}=\tilde{\cL}_1(m)\cap\tilde{\cL}_2(n)$, we have
	$$
	\tilde{H}(z_{m,n})~=~
	\tilde{H}\left(\tilde{\cL}_1(m)\right)\cap\tilde{H}\left(\tilde{\cL}_2(n)\right)
	~=~
	\left(\tilde{\cL}_1(m)+x_0\right)\cap\left(\tilde{\cL}_2(n)+x_0\right)
	~=~
	z_{m,n}+x_0.
	$$
	
	This proves that for the set 
	$$
	\Gamma=\left\{z_{m,n}:~m,n\in\ZZ^d\right\},
	\qquad \text{we~have} \qquad
	\tilde{H}|_{\Gamma}=\id_{\RR^d}+x_0.
	$$
	Since both foliations $\cL_1$ and $\cL_2$ are minimal on $\TT^d$, 
	the set $\Gamma$ is dense in $\RR^d$, we have 
	$$
	\tilde{H}=\id_{\RR^d}+x_0:~\RR^d\to\RR^d.
	$$
	This finishes the proof of the lemma.
\end{proof}

\subsection{Random perturbations of commuting automorphisms}

Now we can prove  Theorems \ref{thm:2-dim} and  \ref{thm:commuting}.  

\begin{proof}[Proof of Theorem \ref{thm:2-dim}]
	Let $\cU_f$ be a neighborhood of $f$ in $\diff^r(\TT^2)$ such that every $g\in\cU_f$ admits a common stable and unstable cone-fields, and $\Omega=\cU_f^\ZZ$. This implies the associated skew product system 
	$$
	F(\omega,x)=\left(\sigma(\omega),f_{\omega}(x)\right):~\Omega\times\TT^2\to\Omega\times\TT^2
	$$ 
	is fiberwise Anosov (Lemma \ref{lem:robust-Anosov}). 
	
	Let $\bar{A}=\sigma\times A:\Omega\times\TT^2\to\Omega\times\TT^2$ be the corresponding extension system (seeDefinition \ref{def:extension}), and let  $H:\Omega\times\TT^2\to\Omega\times\TT^2$  be the conjugacy satisfying $H\circ F=\bar{A}\circ H$ defined in Proposition \ref{prop:structure-stable}.
	
	Let $\nu$ be a probability supported on $\cU_f$ and let $\mu_{\rm SRB}$ be the unique $\nu$-stationary SRB-measure. Denoteby $\mu$ the fiberwise SRB-measure associated to $\mu_{\rm SRB}$ defined as Proposition \ref{prop:measure}. 
	
	As we assumed that $\lambda^+(\mu_{\rm SRB})=\lambda^+(A)$, the positive fiberwise Lyapunov exponent  $\lambda^+(\mu)=\lambda^+(A)$. Proposition \ref{prop:2-H-smooth} shows that 
	$$
	H_*(\mu)=\nu^\ZZ\times\Leb_{\TT^2},
	$$
	and $H$ is $C^r$-smooth along fiberwise unstable foilations on  $\supp(\mu)=\supp(\nu)^\ZZ\times\TT^2$. We fix $g\in\supp(\nu)$ and $H_g$ be the topological conjugacy
	$$
	H_g\circ g=A\circ H_g
	\qquad \text{and} \qquad
	H_g(\cF^s_g)=\cL^s_A,
	$$
	where $\cF^s_g$ and $\cL^s_A$ are stable foliations of $g$ and $A$ respectively.
	
	Proposition \ref{prop:s-nonrandom-T2} shows that
	$H_g(\cF^s_{g'})=\cL^s_A$ for every $g'\in\supp(\nu)$.
	This implies 
	$$
	\cF^s_g=H_g^{-1}(\cL^s_A)=\cF^s_g,
	\qquad \forall g'\in\supp(\nu).
	$$
	This proves the first item of Theorem \ref{thm:2-dim}.
	
	If we further assume that $\lambda^-(\mu_{\rm SRB})=\lambda^-(A)$, then we also have $\lambda^-(\mu)=\lambda^-(A)$.
We consider the inverse skew product system $F^{-1}:\Omega\times\TT^2\to\Omega\times\TT^2$.  Then $\mu$ is also an $F^{-1}$-invariant ergodic measure for $F^{-1}$. Moreover, we have
	\begin{itemize}
		\item $\cF^s$ is the fiberwise expanding foliation for $F^{-1}$;
		\item $\cF^u$ is the fiberwise stable foliation for $F^{-1}$;
		\item $H\circ F^{-1}=\bar{A}^{-1}\circ H$.
		\item $H_*\mu$ is fiberwise u-Gibbs state along $\cL^s_A$ for $\bar A^{-1}$.
	\end{itemize}
	{By Proposition \ref{prop:u-gibbs},}  $\mu$ is also a $u$-Gibbs measure along $\cF^s$ with respect to $F^{-1}$. Then we can apply Proposition \ref{prop:2-H-smooth} and Proposition \ref{prop:s-nonrandom-T2} again which shows that $H$ is $C^r$-smooth along $\cF^s$ on $\supp(\mu)$, and the $F^{-1}$-stable foliation $\cF^u$ is non-random:
	 $$
	 \cF^u_g=H_g^{-1}(\cL^u_A)=\cF^u_{g'},
	 \qquad \forall g'\in\supp(\nu).
	 $$
	 
	 For any fixed $g\in\supp(\nu)$, denote by $h=H_g:\TT^2\to\TT^2$ where $h\circ g=A\circ h$. Since $h$ is $C^r$-smooth along both stable and unstable foliations of $g$, Jouren\'e's theorem (Lemma \ref{lem:Journe}) implies $h$ is $C^{r-}$-smooth.
	 
	 For every $g'\in\TT^2$, we consider a homeomorphism
	 $$
	 h\circ g'\circ h^{-1}\circ A^{-1}:~\TT^2\to\TT^2,
	 $$
	 which satisfies
	 \begin{itemize}
	 	\item $h\circ g'\circ h^{-1}\circ A^{-1}$ is homotopic to $\id_{\TT^2}$;
	 	\item $h\circ g'\circ h^{-1}\circ A^{-1}$ preserves both stable and unstable foliations of $A$:
	 	   \begin{align*}
	 	   	     h\circ g'\circ h^{-1}\circ A^{-1}\left(\cL^{s/u}_A\right)
	 	   	    &=h\circ g'\circ h^{-1}\left(\cL^{s/u}_A\right)  \\
	 	   	    &=h\circ g'\left(\cF^{s/u}_{g'}\right)
	 	   	    =h\left(\cF^{s/u}_{g'}\right) \\
	 	   	    &=\cL^{s/u}_A.
	 	   \end{align*}
	 \end{itemize}
	 
	 We apply Lemma \ref{lem:translation} to $h\circ g'\circ h^{-1}\circ A^{-1}$, there exists $v_{g'}\in\TT^2$ such that 
	 $$
	 h\circ g'\circ h^{-1}\circ A^{-1}=\id_{\TT^2}+v_{g'}.
	 $$
	 This implies
	 $$
	 h\circ g'\circ h^{-1}=A+v_{g'},
	 \qquad \forall g'\in\supp(\nu).
	 $$
\end{proof}

Now we can prove high dimension theorem. Notice that Theorem \ref{thm:d-dim} is a special case of Theorem \ref{thm:commuting}.  %We only prove Theorem \ref{thm:commuting}. 
The proof is similar to Theorem \ref{thm:2-dim}.  We include it for completeness.  

\begin{proof}[Proof of Theorem \ref{thm:commuting}]
	Let $\{A_i\}_{i=1}^m\subseteq{GL}(d,\ZZ)$ be a family of commuting automorphisms which every $A_i$ satisfies the generic assumption and admits the same finest dominated splitting.
	Denote $\cL^s,\cL^u$ the stable and unstable foliations of $A_1$, which are the same to every $A_j$, $j=2,\cdots,m$.
	
	Let $\cU_i\subseteq\diff^2(\TT^d)$ be the neighborhood of $A_i$ for $i=1,\cdots,m$, such that for $\cU=\cup\cU_i$ and $\Omega=\cU^\ZZ$, the corresponding skew product system 
	$$
	F:\Omega\times\TT^d\to\Omega\times\TT^d, \qquad
	F(\omega,x)=\left(\sigma(\omega),f_\omega(x)\right)
	$$ 
	and its conjugacy $H:\Omega\times\TT^d\to\Omega\times\TT^d$ to the extension system $\bar{A}:\Omega\times\TT^d\to\Omega\times\TT^d$
	satisfy Proposition \ref{prop:leaf-conjugacy}. 
	
	If we assume all positive Lyapunov exponents of $\nu$-stationary SRB-measure $\mu_{\rm SRB}$ are equal to positive Lyapunov exponents of $\nu_0$-stationary measure, then for the corresponding fiberwise SRB-measure $\mu$ of $F$ defined by Proposition \ref{prop:measure}, all fiberwise positive exponents of $(F,\mu)$ are equal to $(\bar{A},H_*(\mu))$. 
	
	We apply Proposition \ref{prop:d-H-smooth}, which shows that $H$ intertwines intermediate  unstable  foliations and is $C^{1+}$-smooth along all these unstable invariant foliations on $\supp(\mu)=\supp(\nu)^\ZZ\times\TT^d$.
	
	We fix $f\in\supp(\nu)$ and $H_f:\TT^d\to\TT^d$ be the topological conjugacy with $H_f\circ f=A_f\circ H_f$, where $A_f=A_i$ if $f\in\cU_i$. For every $g\in\supp(\nu)$, Proposition \ref{prop:s-nonrandom-Td} shows that
	$$
	H_f(\cF^s_g)=\cL^s
	\qquad \text{thus} \qquad
	\cF^s_g=H_f^{-1}(\cL^s)=\cF^s_f.
	$$
	This proves the first item of Theorem \ref{thm:commuting}.
	
	If we further assume all negative Lyapunov exponents of $\mu_{\rm SRB}$ are equal to that of a $\nu_0$-stationary measure, then, as in the proof of Theorem \ref{thm:2-dim}, using  Proposition \ref{prop:u-gibbs}, the corresponding fiberwise SRB-measure $\mu$ is also a  fiberwise SRB-measure for $F^{-1}$.  We repeat the argument again, which shows that $H$ is $C^{1+}$-smooth along $\cF^s$ on $\supp(\mu)$ and 
	$$
	\cF^u_g=H_f^{-1}(\cL^u)=\cF^u_f,
	\qquad \forall g\in\supp(\nu).
	$$
	
	Denote $h=H_f:\TT^d\to\TT^d$, then for every $g\in\supp(\nu)$, the homeomorphism 
	$$
	h\circ g\circ h^{-1}\circ A_g^{-1}:~\TT^d\to\TT^d,
	\qquad \text{where} \qquad
	A_g=A_j~\text{if}~g\in\cU_j,
	$$
	is homotopic to $\id_{\TT^d}$ and preserves $\cL^s,\cL^u$. Lemma \ref{lem:translation} shows that $h\circ g\circ h^{-1}\circ A_g^{-1}$ is a translation on $\TT^d$. 
	This implies
\[
	h\circ g\circ h^{-1}=A_g+v_{g},
	\qquad \forall g\in\supp(\nu). \qedhere
	\]
\end{proof}

\subsection{Random perturbations of non-commuting automorphisms}

Finally, we prove Theorem \ref{thm:non-commuting}, which shows coincidence of the the positive Lyapunov exponents implies simultaneously linearization for non-commuting Anosov automorphisms on $\TT^2$.

\begin{proof}[Proof of Theorem \ref{thm:non-commuting}]
	Let $\{A_i\}_{i=1}^m\subseteq{GL}(2,\ZZ)$ be a  non-commuting, cone hyperbolic family of {Anosov} automorphisms.  % which \red{share common stable and unstable cone-fields.}
	
	Let $\cU_i$ be a $C^2$-neighborhood of $A_i$ in $\diff^r(\TT^2)$, and denote $\cU=\cup\cU_i$, $\Omega=\cU^\ZZ$, such that skew product system 
	$$
	F(\omega,x)=\left(\sigma(\omega),f_{\omega}(x)\right):~\Omega\times\TT^2\to\Omega\times\TT^2
	$$ 
	is fiberwise Anosov (Lemma \ref{lem:robust-Anosov}). 
	
	Moreover, for every pair $A_i\neq A_j$ which are not commuting, their corresponding stable foliations $\cL^s_i$ and $\cL^s_j$ transversal. The same holds for their unstable foliations $\cL^u_i$ and $\cL^u_j$.  
	We shrink each $\cU_i$ for $i=1,\cdots,m$ such that for every pair $A_i,A_j$ that are not commuting, every pair $f\in\cU_i$ and $g\in\cU_j$ satisfy
	\begin{itemize}
		\item their corresponding stable foliations $\cF^s_f$ and $\cF^s_g$ are transversal everywhere on $\TT^2$;
		\item their corresponding unstable foliations $\cF^u_f$ and $\cF^u_g$ are transversal everywhere on $\TT^2$.
	\end{itemize} 
	
	Let $\bar{A}:\Omega\times\TT^2\to\Omega\times\TT^2$ be the corresponding extension system (Definition \ref{def:extension}), and $H:\Omega\times\TT^2\to\Omega\times\TT^2$  be the conjugacy satisfying $H\circ F=\bar{A}\circ H$ defined in Proposition \ref{prop:structure-stable}.

	Let $\nu$ be a probability supported on $\cU$ with $\nu(\cU_i)=p_i\in(0,1)$ where $\sum_{i=1}^m p_i=1$. Denote  by $p$ be the probability on $\{A_i\}_{i=1}^m$ with $p(A_i)=p_i$. 
	Let $\mu_{\rm SRB}$ be the unique $\nu$-stationary SRB-measure, and $\mu$ be the fiberwise SRB-measure of skew product $F$ associated to $\mu_{\rm SRB}$ defined as Proposition \ref{prop:measure}. 
	
Denote by $\lambda^+(p)$ the positive Lyapunov exponent for the linear random walk 
	As we assumed that $\lambda^+(\mu_{\rm SRB})=\lambda^+(p)$, the positive fiberwise Lyapunov exponents for the induced skew system $F$ with respect to $\mu$ satisfy   $\lambda^+(\mu)=\lambda^+(p)$. Proposition \ref{prop:2-H-smooth} shows that 
	$$
	H_*(\mu)=\nu^\ZZ\times\Leb_{\TT^2},
	$$
	and $H$ is $C^r$-smooth along fiberwise unstable foilations on  $\supp(\mu)=\supp(\nu)^\ZZ\times\TT^2$.
	
	We fix $f,g\in\supp(\nu)$ where their linear part $A_f,A_g\in\{A_i\}_{i=1}^m$ are not commuting. Let $\cL^s_f$ and $\cL^s_g$ be the stable foliations of $A_f$ and $A_g$ respectively. Since $A_f,A_g$ are non-commuting, $\cL^s_f,\cL^s_g$ are two transversal linear minimal foliations on $\TT^2$. Denote $H_f$ be the topological conjugacy
	$$
	H_f\circ f=A_f\circ H_f
	\qquad \text{and} \qquad
	H_f(\cF^s_f)=\cL^s_f,
	$$
	where $\cF^s_f$ is the stable foliation of $f$. 
	The same for the conjugacy $H_g$ and stable foliation $\cF^s_g$ of $g$.
	
	Since the conjugacy $H$ is $C^r$-smooth along fiberwise unstable foliation on $\supp(\mu)=\supp(\nu)^\ZZ\times\TT^2$, 
	Proposition \ref{prop:s-nonrandom-T2} shows that
	$$
	H_f(\cF^s_g)=\cL^s_g
	\qquad \text{and} \qquad
	H_g(\cF^s_f)=\cL^s_f.
	$$
	
	Denote $h=H_f:\TT^2\to\TT^2$ where $h\circ f\circ h^{-1}=A_f$.
	For the homeomorphism $h\circ H_g^{-1}:\TT^2\to\TT^2$, it satisfies
	\begin{itemize}
		\item $h\circ H_g^{-1}$ is homotopic to $\id_\TT^2$;
		\item $h\circ H_g^{-1}$ preserves both stable foliations $\cL^s_f$ and $\cL^s_g$:
		$$
		h\circ H_g^{-1}(\cL^s_f)=h(\cF^s_f)=\cL^s_f
		\qquad \text{and} \qquad
		h\circ H_g^{-1}(\cL^s_g)=h(\cF^s_g)=\cL^s_g.
		$$
	\end{itemize}
	Lemma \ref{lem:translation} shows that there exists $u_g\in\TT^2$ such that
	$$
	h\circ H_g^{-1}=\id_{\TT^2}+u_g
	\qquad \text{that~is} \qquad
	h=H_g+u_g.
	$$
	
	Thus we have
	\begin{align*}
		h\circ g\circ h^{-1}
		&=\left(h\circ H_g^{-1}\right)\circ \left(H_g\circ g\circ H_g^{-1}\right)\circ \left(h\circ H_g^{-1}\right)^{-1} \\
		&=\left(\id_{\TT^2}+u_g\right)\circ A_g\circ\left(\id_{\TT^2}-u_g\right) \\
		&=A_g+v_g,
	\end{align*}
	where $v_g=(u_g-A_gu_g)\in\TT^2$.
	
	The same argument shows that for every $g'\in\supp(\mu)$ where its linear part $A_{g'}$ is not commuting with $A_f$, we have 
	$$
	h\circ g'\circ h^{-1}=A_{g'}+v_{g'}, \qquad v_{g'}\in\TT^2.
	$$
	
	If $A_{g'}$ is commuting with $A_f$, then it is not commuting with $A_g$, the same argument shows that
	$$
	H_g\circ g'\circ H_g^{-1}=A_{g'}+u_{g'}, \qquad u_{g'}\in\TT^2.
	$$
	Thus we have
		\begin{align*}
		h\circ g\circ h^{-1}
		&=\left(h\circ H_g^{-1}\right)\circ \left(H_g\circ g\circ H_g^{-1}\right)\circ \left(h\circ H_g^{-1}\right)^{-1} \\
		&=\left(\id_{\TT^2}+u_g\right)\circ\left(A_{g'}+u_{g'}\right)\circ\left(\id_{\TT^2}-u_g\right) \\
		&=A_{g'}+v_{g'},
	\end{align*}
	where $v_{g'}=(u_g+u_{g'}-A_{g'}u_g)\in\TT^2$. 
	
	Finally, we need to show $h$ is $C^{r-}$-smooth on $\TT^2$. We have 
	$h$ is $C^r$-smooth along $\cF^u_f$, and $H_g$ is $C^r$-smooth along $\cF^u_g$.
	This implies $h=H_g+u_g$ is $C^r$-smooth along $\cF^u_g$. Since $\cF^u_f$ and $\cF^u_g$ are two transversal foliations on $\TT^2$, Journ\'e's Theorem (Lemma \ref{lem:Journe}) shows that $h$ is $C^{r-}$-smooth. This finishes the proof of Theorem \ref{thm:non-commuting}.
\end{proof}

\bibliographystyle{plain}

\begin{thebibliography}{99}
	
	\bibitem{Ben}
	\newblock P. Baxendale,
	\newblock Lyapunov exponents and relative entropy for a stochastic flow of diffeomorphisms,
	\newblock \emph{Probab. Theory Related Fields}, \textbf{81} (1989), no. 4, 521--554.
	
	
	\bibitem{BLOR}
	\newblock A. Brown, H. Lee, D. Obata and Y. Ruan,
	\newblock Absolute continuity of stationary measures,
	\newblock \emph{arXiv:2409.18252}.
	
	\bibitem{CP}
	\newblock S. Crovisier and R. Potrie, 
	\newblock Introduction to partially hyperbolic dynamics, 
	\newblock \emph{School on Dynamical Systems}, ICTP, Trieste, (2015).
	
	
	\bibitem{DeW1}
	\newblock J. DeWitt,
	\newblock Local Lyapunov spectrum rigidity of nilmanifold automorphisms,
	\newblock \emph{J. Mod. Dyn.}, \textbf{17} (2021), 65--109.
	
	\bibitem{DeW2}
	\newblock J. DeWitt,
	\newblock Simultaneous linearization of diffeomorphisms of isotropic manifolds,
    \newblock \emph{J. Eur. Math. Soc.}, \textbf{26} (2024), no. 8, 2897--2969.
    
    \bibitem{DeG}
    \newblock J. DeWitt and A. Gogolev,
    \newblock Dominated splitting from constant periodic data and global rigidity of Anosov automorphisms,
    \newblock \emph{Geom. Funct. Anal.}, \textbf{34} (2024), no. 5, 1370--1398.
    
	
	
	\bibitem{DK}
	\newblock D. Dolgopyat and R. Krikorian,
	\newblock On simultaneous linearization of diffeomorphisms of the sphere,
	\newblock \emph{Duke Math. J.}, \textbf{136} (2007), no. 3, 475--505.
	
	
\bibitem{FKS}
\newblock D.~Fisher, B.~Kalinin, and R.~Spatzier,
\newblock Global rigidity of higher rank {A}nosov actions on tori and nilmanifolds,
\newblock \emph{J. Amer. Math. Soc.},
\newblock \textbf{26} (2013), no. 1,  167--198.

	
	
	\bibitem{Fr}
	\newblock J. Franks,
	\newblock Anosov diffeomorphisms on tori,
	\newblock \emph{Trans. Amer. Math. Soc.}, \textbf{145} (1969), 117--124.
	

	
	\bibitem{Go08}
	\newblock A. Gogolev,
	\newblock Smooth conjugacy of Anosov diffeomorphisms on higher-dimensional tori,
	\newblock \emph{J. Mod. Dyn.}, \textbf{2} (2008), no. 4, 645--700.
	
	
	
	\bibitem{GG}
	\newblock A. Gogolev and M. Guysinsky,
	\newblock $C^1$-differentiable conjugacy of Anosov diffeomorphisms on three dimensional torus,
	\newblock \emph{Discrete Contin. Dyn. Syst.}, \textbf{22} (2008), no. 1-2, 183--200.
	
	
	\bibitem{GKS}
	\newblock A. Gogolev, B. Kalinin and V. Sadovskaya, 
	\newblock Local rigidity for Anosov automorphisms. With an appendix by Rafael de la Llave, 
	\newblock \emph{Math. Res. Lett.}, \textbf{18} (2011), no. 5, 843--858.
	
	\bibitem{GKS1}
	\newblock A. Gogolev, B. Kalinin and V. Sadovskaya, 
	\newblock Local rigidity of Lyapunov spectrum for toral automorphisms,
	\newblock \emph{Israel J. Math.}, \textbf{238} (2020), no. 1, 389--403. 
	
	\bibitem{GRH}
	\newblock A. Gogolev and F. Rodriguez Hertz,
	\newblock Smooth rigidity for very non-algebraic expanding maps,
	\newblock \emph{J. Eur. Math. Soc.}, \textbf{25} (2023), no. 8, 3289--3323.
	
	\bibitem{GoS}
	\newblock A. Gogolev and Y. Shi,
	\newblock Joint integrability and spectral rigidity for Anosov diffeomorphisms,
	\newblock \emph{Proc. Lond. Math. Soc.}, (3) \textbf{127} (2023), no. 6, 1693--1748.
	
	\bibitem{Gr}
	\newblock M. Gromov,
	\newblock Groups of polynomial growth and expanding maps,
	\newblock \emph{Inst. Hautes \'Etudes Sci. Publ. Math.}, \textbf{53} (1981), 53--73.

	
	
%	\bibitem{FRH}
%	\newblock F. Rodriguez Hertz,
%	\newblock Stable ergodicity of certain linear automorphisms of the torus,
%	\newblock \emph{Ann. of Math.}, (2) \textbf{162} (2005), no. 1, 65--107.
	
	\bibitem{HW}
	\newblock F. Rodriguez Hertz and Z. Wang,
	\newblock Global rigidity of higher rank abelian Anosov algebraic actions,
	\newblock \emph{Invent. Math.}, \textbf{198} (2014), no. 1, 165--209.
	
	
	\bibitem{HPS}
	\newblock M. Hirsch, C. Pugh and M. Shub, 
	\newblock Invariant manifolds, 
	\newblock \emph{Lecture Notes in Mathematics}, \textbf{583} (1977), Springer--Verlag, Berlin.
	
	\bibitem{J}
	\newblock J.-L. Journ\'e,
	\newblock A regularity lemma for functions of several variables,
	\newblock \emph{Rev. Mat. Iberoamericana}, \textbf{4} (1988), no. 2, 187--193.
	
%	\bibitem{K}
%	\newblock B. Kalinin,
%	\newblock Liv$\check{\rm s}$ic theorem for matrix cocycles,
%	\newblock \emph{Ann. of Math.}, (2) \textbf{173} (2011), no. 2, 1025--1042.
	
	\bibitem{KS07}
	\newblock B. Kalinin and V. Sadovskaya, 
	\newblock On classification of resonance-free Anosov $\ZZ^k$ actions,
	\newblock \emph{Michigan Math. J.}, \textbf{55} (2007), no. 3, 651--670.
	
%	\bibitem{KS}
%	\newblock B. Kalinin and V. Sadovskaya, 
%	\newblock On Anosov diffeomorphisms with asymptotically conformal periodic data,
%	\newblock \emph{Ergodic Theory Dynam. Systems}, \textbf{29} (2009), no. 1, 117--136. 
	
    \bibitem{KSW1}
	\newblock B. Kalinin, V. Sadovskaya and Z. Wang,
	\newblock Smooth local rigidity for hyperbolic toral automorphisms,
	\newblock \emph{Commun. Amer. Math. Soc.}, \textbf{3} (2023), 290--328.
	
	
	\bibitem{KSW2}
	\newblock B. Kalinin, V. Sadovskaya and Z. Wang,
	\newblock Global smooth rigidity for toral automorphisms,
	\newblock \emph{arXiv:2407.13877}.
	
	\bibitem{KH}
	\newblock A. Katok and B. Hasselblatt,
	\newblock Introduction to the modern theory of dynamical systems.
	\newblock \emph{Encyclopedia Math. Appl.}, \textbf{54} Cambridge University Press, (1995), xviii+802 pp.
	
	\bibitem{KKS}
	\newblock A. Katok, S. Katok and K. Schmidt,   
	\newblock Rigidity of measurable structure for $\ZZ^d$-actions by automorphisms of a torus,
	\newblock \emph{Comment. Math. Helv}. {\bf77} (2002), no. 4, 718--745.
	
	\bibitem{KK}
	\newblock K. Khanin and Y. Kifer,
	\newblock Thermodynamic formalism for random transformations and statistical mechanics,
	\newblock \emph{Amer. Math. Soc. Transl. Ser. 2}, \textbf{171} (1996), 107--140.
    
%    \bibitem{Kif}
%    \newblock Y. Kifer,
%   \newblock General random perturbations of hyperbolic and expanding transformations,
%    \newblock \emph{J. Analyse Math.}, \textbf{47} (1986), 111--150.
    
    \bibitem{Kif1}
    \newblock Y. Kifer,
    \newblock Equilibrium states for random expanding transformations,
    \newblock \emph{Random Comput. Dynam.}, \textbf{1} (1992/93), no. 1, 1--31.
    
    \bibitem{Kif2}
    \newblock Y. Kifer,
    \newblock Thermodynamic formalism for random transformations revisited,
    \newblock \emph{Stoch. Dyn.}, \textbf{8} (2008), no. 1, 77--102.
    
    \bibitem{Krz}
    \newblock K. Krzyzewsky,
    \newblock Some results on expanding mappings,
    \newblock \emph{Asterisque}, \textbf{50}, (1977), 205--218.
    	
    \bibitem{LJ}
    \newblock Y. Le Jan,
    \newblock \'Equilibre statistique pour les produits de diff\'eomorphismes al\'eatoires ind\'ependants,
    \newblock \emph{Ann. Inst. H. Poincar\'e Probab. Statist.}, \textbf{23} (1987), no. 1, 111--120.
    

 	\bibitem{Liu}
	\newblock P.~Liu, 
	\newblock Random perturbations of Axiom A basic sets,
	\newblock \emph{J. Stat. Phys.} \textbf{90} (1998), 467--490.
       	
	\bibitem{LQ}
	\newblock P. Liu and M. Qian,
	\newblock Smooth ergodic theory of random dynamical systems,
	\newblock \emph{Lecture Notes in Math.}, \textbf{1606}, Springer-Verlag, Berlin, (1995), xii+221 pp.
	
	\bibitem{dlL}
	\newblock R. de la Llave,
	\newblock Invariants for smooth conjugacy of hyperbolic dynamical systems. II,
	\newblock \emph{Comm. Math. Phys.}, \textbf{109} (1987), no. 3, 369--378.
	
	\bibitem{dlL92}
	\newblock R. de la Llave,
	\newblock Smooth conjugacy and S-R-B measures for uniformly and nonuniformly hyperbolic systems, 
	\newblock \emph{Comm. Math. Phys.}, \textbf{150} (1992), no. 2, 289--320.
	
	\bibitem{dlLM}
	\newblock R. de la Llave and R. Moriy\'on,
	\newblock Invariants for smooth conjugacy of hyperbolic dynamical systems. IV,
	\newblock \emph{Comm. Math. Phys.}, \textbf{116} (1988), no. 2, 185--192.
	
	
	
	\bibitem{Mn}
	\newblock A. Manning,
	\newblock There are no new Anosov diffeomorphisms on tori,
	\newblock \emph{Amer. J. Math.}, \textbf{96} (1974), 422--429.
	
	\bibitem{MM1}
	\newblock J. M. Marco and R. Moriy\'on,
	\newblock Invariants for smooth conjugacy of hyperbolic dynamical systems. I,
	\newblock \emph{Comm. Math. Phys.}, \textbf{109} (1987), no. 4, 681--689.
	
%	\bibitem{MM3}
%	\newblock J. M. Marco and R. Moriy\'on,
%	\newblock Invariants for smooth conjugacy of hyperbolic dynamical systems. III.
%	\newblock \emph{Comm. Math. Phys.}, \textbf{112} (1987), no. 2, 317--333.
	
	\bibitem{Pes}
	\newblock Y. Pesin,
	\newblock Lectures on partial hyperbolicity and stable ergodicity,
    \newblock \emph{Zur. Lect. Adv. Math.}, European Mathematical Society, Zürich, (2004), vi+122 pp.
    
%    \bibitem{PeSi}
%    \newblock Y. Pesin and Ya. Sina\u{i},
%    \newblock Gibbs measures for partially hyperbolic attractors,
%    \newblock \emph{Ergodic Theory Dynam. Systems}, \textbf{2} (1982), no. 3-4, 417--438.
	

	
%	\bibitem{PSW}
%	\newblock C. Pugh, M. Shub and A. Wilkinson,
%	\newblock H\"{o}lder foliations,
%	\newblock \emph{Duke Math. J.}, \textbf{86} (1997), 517--546.
	
	
	
	%\bibitem{PSW1}
	%\newblock C. Pugh, M. Shub and A. Wilkinson,
	%\newblock H\"{o}lder foliations, revisited, 
	%\newblock \emph{J. Mod. Dyn.}, \textbf{6} (2012), no. 1, 79--120. 
	
	%\bibitem{RGZ}
	%\newblock Y. Ren, S. Gan and P. Zhang,
	%\newblock Accessibility and homology bounded strong unstable foliation for Anosov diffeomorphisms on 3-torus,
	%\newblock \emph{Acta Math. Sin. (Engl. Ser.)}, \textbf{33} (2017), no. 1, 71--76.
	
	
	\bibitem{SY}
	\newblock R. Saghin and J. Yang, 
	\newblock Lyapunov exponents and rigidity of Anosov automorphisms and skew products, 
	\newblock \emph{Adv. in Math.}, \textbf{355} (2019), Article no. 106764. 
	
	\bibitem{Sa15}
	\newblock V. Sadovskaya, 
	\newblock Cohomology of fiber bunched cocycles over hyperbolic systems,
	\newblock \emph{Ergodic Theory Dynam. Systems}, \textbf{35} (2015), 2669--2688.
	
	\bibitem{Sa}
	\newblock V. Sadovskaya,
	\newblock Linear cocycles over hyperbolic systems,
	\newblock \emph{A vision for dynamics in the 21st century—the legacy of Anatole Katok}, Cambridge University Press, Cambridge, (2024), 384--434.
	
	\bibitem{Sack}
    \newblock R. Sacksteder,
    \newblock The measures invariant under an expanding map, 
    \newblock \emph{G\'eom\'etrie Diff\'erentielle}, Lecture Notes in Mathematics \textbf{392}, (1974), Springer-Verlag, Berlin-New York, pp. 179--194.	
	
	\bibitem{Sh}
	\newblock M. Shub,
	\newblock Endomorphisms of compact differentiable manifolds,
	\newblock \emph{Amer. J. Math.}, \textbf{91} (1969), 175--199.
	
	\bibitem{ShS}
	\newblock M. Shub and D. Sullivan,
	\newblock Expanding endomorphisms of the circle revisited, 
	\newblock \emph{Ergodic Theory Dynam. Systems}, \textbf{5} (1985), 285--289.
	
	\bibitem{Sm}
	\newblock S. Smale,
	\newblock Differentiable dynamical systems,
    \newblock \emph{Bull. Amer. Math. Soc.}, \textbf{73} (1967), 747--817.
	
	
\end{thebibliography}

\vskip5mm

\flushleft{\bf Aaron Brown} \\
Department of Mathematics, Northwestern University, Evanston, IL 60208, USA\\
\textit{E-mail:} \texttt{awb@northwestern.edu}\\

\flushleft{\bf Yi Shi} \\
School of Mathematics, Sichuan University, Chengdu, 610065,  China\\
\textit{E-mail:} \texttt{shiyi@scu.edu.cn}\\

\end{document}